\documentclass[a4paper,11pt]{article}
\usepackage[utf8]{inputenc}
\usepackage{amsfonts, amssymb, amsmath, amsthm,tikz}
\usepackage{color}

\usepackage{srcltx}

\usepackage{epsfig,graphicx}
\usepackage[margin=1.25in]{geometry}

\newtheorem{theorem}{Theorem}
\newtheorem*{theorem*}{Theorem}

\newtheorem{lemma}{Lemma}
\newtheorem{claim}[lemma]{Claim}
\newtheorem{question}[lemma]{Question}

\newtheorem*{thmB}{Theorem B'}

\newtheorem{conjecture}[lemma]{Conjecture}
\newtheorem{corollary}[lemma]{Corollary}
\newtheorem{proposition}[lemma]{Proposition}
\newtheorem{remark}[lemma]{Remark}

\newtheorem*{main-conjecture}{Main conjecture}

\theoremstyle{definition}
\newtheorem{definition}[lemma]{Definition}

\def\DD{{\mathbb D}}  
\def\TT{{\mathbb T}}  

\def\ZZ{{\mathbb Z}}
\def\AA{{\mathbb A}}
\def\RR{{\mathbb R}}

\def\NN{{\mathbb N}}

\def\La{\Lambda}
\def\la{\lambda}
\def\e{\epsilon}
\def\eps{\epsilon}

\def\ga{\gamma}
\def\Ga{\Gamma}
\def\de{\delta}

\def\La{\Lambda}
\def\be{\beta}
\def\eps{\epsilon}
\def\al{\alpha}

\def\Ga{\Gamma}
\def\si{\sigma}

\def\emp{\emptyset}
\def\diff{\operatorname{Diff}}

\def\noi{\noindent}

\date{\today}
  \title{Critical set for surface diffeomorphisms revisited}
\author{Sylvain Crovisier, Enrique Pujals}

\begin{document}
\maketitle
\begin{abstract}
We propose a notion of \emph{critical set} for surface diffeomorphisms, defined intrinsically in terms of the projective action of the derivative, that plays a role analogous to that of critical points in one-dimensional dynamics. Under a non-degeneracy assumption on the derivative cocycle, a compact invariant set admits a dominated splitting if and only if it contains no critical points. 
For dissipative diffeomorphisms of the disk we derive structural
consequences from the critical set and its recurrence: a dichotomy for
the closure of unstable manifolds of saddle periodic points, and, when the critical set is
non-recurrent, a finite decomposition of the dynamics analogous to the
one-dimensional Misiurewicz case.

\end{abstract}
\section{Introduction}


The motivation of the present paper stems from the classical understanding that, in the one-dimensional setting, the presence of critical points (where the derivative vanishes)  is a key obstruction to hyperbolicity. Conversely, smooth  one-dimensional endomorphisms without critical points are either hyperbolic, conjugate to hyperbolic or conjugate to an irrational rotation; see \cite{M}, where a real one dimensional version of Fatou's hyperbolicity criteria for holomorphic endomorphisms of the Riemann sphere was proved. In this sense, any compact invariant set is ``essentially hyperbolic" if and only if it does not contain critical points.
Extending this perspective to two dimensions, we pursue a similar criterion: to isolate the phenomena that {\em obstruct hyperbolicity for surface diffeomorphisms} and, dually, {\em those whose absence guarantees it}. In this context, the critical set plays a role analogous to that of critical points in one-dimensional dynamics, and designed to capture precisely those dynamical features that impede hyperbolic behavior.

Since it was proved in \cite{PS1} that any compact invariant set exhibiting a {\it dominated splitting} for a generic $C^2$ surface diffeomorphism is a hyperbolic set, it follows that searching for dynamical obstructions to hyperbolicity is essentially equivalent to identifying obstructions to domination. The critical set we propose is therefore naturally aligned with this goal: {\it a well defined set whose presence obstructs domination, and whose absence corresponds to  the existence of a dominated splitting}.

A similar task was carried out in \cite{PRH}; we follow that perspective while simplifying the notion of the critical set introduced there, removing certain cumbersome features that limited its applicability and obtaining a more accessible and intrinsic description of criticality in surface dynamics. 
This concept was further extended in \cite{V} to the context of automorphisms of $\mathbb{C}^2$, underscoring its broader relevance in the study of dynamical systems beyond the real two-dimensional setting.

Informally, to exhibit a dominated splitting means that there is an invariant cone field that is forwardly contracted and backwardly expanded; in that sense, a critical point is a point with {\it ``a cone that is not contracted neither for the future nor for the past"}. More precisely, this notion is recast in terms of the projective dynamics induced by the derivative (see paragraph a- in the present section for a formal definition). Moreover, under the assumption of ``far from homotheties'' (which is implied by sufficient dissipation, see paragraph b), the absence of critical points is equivalent to the existence of a dominated splitting, and no full orbit consists entirely of critical points.

Tangencies associated with the stable and unstable manifold of periodic points (see paragraph i- in section \ref{examples} for the precise definition and discussion) are a natural obstruction to dominated splitting. On the one hand, however, tangencies are invariant under iteration and in that sense would not qualify as a good definition for a critical point.

However, there is a robust reason why the critical set cannot be reduced to tangencies
of periodic points. Kupka--Smale diffeomorphisms (those for which all
periodic points are hyperbolic and all intersections between their stable and
unstable manifolds are transversal) form a residual subset of
$Diff^r(M)$; hence, generically, there are no tangencies associated with
periodic points at all. On the other hand, by the results of Newhouse
\cite{N} there exist \emph{open} sets of surface diffeomorphisms whose
non-wandering set is not hyperbolic and does not admit a dominated splitting.
By theorem \ref{p.existence}, every diffeomorphism in such an open set
carries critical points. Critical points are therefore present on open sets
of diffeomorphisms on which no periodic tangency exists: the notion is not a
reformulation of tangencies, but a strictly more general, intrinsically
defined obstruction (see paragraphs i-, iii- and iv- in section
\ref{examples} for the precise relations between the two notions).

 
The relation between tangencies and critical points is made precise in section \ref{examples}: lemma \ref{lem:tang-crit} shows that associated with any tangency there is a critical point whose critical direction coincides with the tangent line to the stable and unstable manifolds at the tangency point. In the reverse direction, non-recurrent critical points (see remark \ref{rm:crit-dom} and section \ref{sec:misiu}) are related to tangencies between stable and (center-)unstable manifolds of points in sets with dominated splitting.

As in the one-dimensional setting, one can impose conditions on critical points to gain a more detailed understanding of the non-wandering set, for instance, obtaining a finite decomposition into transitive components or ensuring the existence of SRB (Sinai–Ruelle–Bowen) measures. A notable class of such systems is given by Misiurewicz maps, where the critical points are non-recurrent. These maps exhibit a finite number of transitive pieces and often display rich chaotic behavior, including the presence of  SRB measures. In the two-dimensional context, it is natural to expect that a suitably defined notion of criticality could play an analogous role. In this paper, we explore this possibility by introducing and analyzing a class of systems we refer to as two-dimensional Misiurewicz diffeomorphisms, where a version of non-recurrence of the critical set enables similar decomposition into a finite number of transitive pieces.

That approach opens the possibility to consider other settings where a more general condition on the critical set could have strong dynamical implications. This perspective is  carried out in \cite{CLPY} where a strong condition on the critical sets for infinite renormalizable dissipative diffeomorphisms  allows one to obtain a general description of the dynamics (see paragraph v- in section \ref{examples}).







\bigskip

\paragraph{a -- The critical set.}

Consider a compact boundaryless riemannian surface $M$ and a $C^1$-diffeomorphism $f$.
The derivative of $f$  induces a dynamics  on the unit tangent bundle,  $G_x\colon T^1_xM \to T^1_{f(x)}M$,  usually called the {\it projective tangent bundle map}, 
 defined by $$G_x(v)=
\frac{D_xf(v)}{|D_xf(v)|},$$ and the family $\{G_x\}_{x\in M}$ induces a projective cocycle over $f$.  Let  $g$ be the derivative of $G$ along the fibers
$T_x^1M$ of the unit tangent bundle; more precisely, 
$g_x(v)= |D_wG_x|_{w=v}|.$
We denote the composition $G^n_x= G_{f^{n-1}(x)}\circ\dots \circ G_{x}$ and
to simplify the notation, we denote $D_wG^n_x$ by  $g^n_x$, abbreviated as $g^n$ when no confusion arises. Observe that $g^n_x(v)= \Pi g_{f^i(x)}(G^i_x(v)).$ It is easy to show that 

\begin{eqnarray}\label{eq_f-g}
    |g^n(v)|=  \frac{|\det(Df^n)|}{|Df^n(v)|^2}.
\end{eqnarray}
Let $\Lambda\subset M$ be an invariant compact set. That set  may contain attracting, repelling or elliptic  periodic points. Unlike  in \cite{PRH}, it is not assumed that $f_{|\Lambda}$ is  dissipative or that  $\La$ is contained in the limit set of $f$.

\begin{definition}\label{def:crit} Let $\Lambda\subset M$ be an invariant compact set of  a surface diffeomorphism $f.$ It is said that $x$ is a {\it critical point} of $f$ in $\La$ if there exists $v\in T_xM$ such that 
$|g^n(v)|\geq 1$ for all $n\in \ZZ$. The direction $v$ is called the critical direction and 
 $${\rm Crit}(f,\Lambda)=\{x\in\Lambda, \; \exists v\in T_x^1M,\;
\forall n\in \ZZ, \; |g^n(v)|\geq 1\},$$ is called the critical set of $f$ in $\Lambda.$ Throughout the paper, we refer to elements of ${\rm Crit}(f,\Lambda)$ as
critical points.
\end{definition}


In short,  a critical point is a point having a direction $v$ such that $g^n(v)$ is bounded from below for all integers.  By equation (\ref{eq_f-g}), the condition $|g^n(v)|\geq 1$ for all $n\in\ZZ$ is
equivalent to
$$\|Df^n(v)\|^2\;\leq\; |\det(D_xf^n)|, \qquad \forall\, n\in \ZZ.$$
In other words, a critical direction is a direction that is contracted at
least as strongly as the conformal rate $|\det(D_xf^n)|^{1/2}$ (the geometric
mean of the singular values of $D_xf^n$) in \emph{both} time directions: $v$
never carries the dominant expansion, neither in the future nor in the past.
In the dissipative case, where $|\det(D_xf)|\leq b<1$ on $\La$, the forward
inequality gives
$$\|Df^n(v)\|\leq b^{n/2},\qquad n\geq 0,$$
so the critical direction is exponentially contracted in the future, while
for $n\leq 0$ the inequality forces the complementary directions to carry an
expansion of order at least $|\det(D_xf^{n})|^{1/2}$. This is precisely the
behavior of the common tangent direction at a point of tangency between a
stable and an unstable manifold, contracted in the future, dominated by a
transverse expanded direction in the past, and it is the two-dimensional
counterpart of the one-dimensional picture, where the orbit of a critical
point is decoupled from any expansion of the derivative.
 


\begin{definition}\label{def:FFH} Let $f\in Diff^1(M)$ and let $\La$ be a
compact invariant set of $f$. It is said that $f_{|\La}$ is {\it far from
homotheties} if there exists $\de>0$ such that there is no point $x$ and
vector $v\in T^1_xM$ satisfying
\begin{equation}\label{FFH}
(1-\de)^n< g^n(v)<(1+\de)^n, \qquad \forall\, n\geq 0.
\end{equation}
\end{definition}

Observe that if (\ref{FFH}) does not hold for any point $x$ and any vector $v$ then obviously also holds for any $\de'<\de$.

By compactness, it is equivalent to say that  there is a positive integer $N_0$ and $\de$ such that for any $x\in \La$ and $v\in T^1_xM$ there is a positive integer $n\leq N_0$ such that either $g^{n}(v)\geq (1+\de)^n$ or $g^{n}(v)\leq (1-\de)^n.$


 To motivate the notion of  ``far from homotheties''  suppose that for all $n \geq 0$ and all $v \in T^1_xM$ one has $(1-\delta)^n < g^n(v) < (1+\delta)^n$. Since $g^n(v) = |\det(Df^n)|/\|Df^n(v)\|^2$, dividing the window conditions for two vectors $v$ and $w$ gives
$$
\left(\frac{1-\delta}{1+\delta}\right)^{n/2} <
\frac{\|Df^n(v)\|}{\|Df^n(w)\|} <
\left(\frac{1+\delta}{1-\delta}\right)^{n/2},
$$
so that all directions are scaled by $Df^n$ comparably, with discrepancy at most an exponential factor in $\delta$. Taking $\delta\to 0$, the ratio $\|Df^n(v)\|/\|Df^n(w)\|$ is forced to $1$ for all pairs of directions and all iterates, an homothety for all iterates.  Being $\delta$-far from homotheties rules this out uniformly: for every $x \in \Lambda$ there must exist some direction $v \in T^1_xM$ and some $n \in \mathbb{Z}$ such that $g^n(v)$ exits the window $((1-\delta)^n,(1+\delta)^n)$, meaning the derivative cocycle must exhibit a definite discrepancy in the scaling of different directions at some iterate.

Recall that a compact invariant set $\La$ has a dominated splitting if there exists a subbundle decomposition $T_\Lambda M = E \oplus F$ with $\dim(E) = \dim(F) = 1$ and a positive constant $\la < 1$ such that $\frac{\|Df_{|E}\|}{\|Df_{|F}\|} < \la$; the subbundle $E$ is called the center-stable subbundle and $F$ the center-unstable subbundle.
 
In terms of the projective dynamics, this is equivalent to saying that $E$ is a repeller and $F$ is an attractor under $g$, i.e. $g(E) > 1 > g(F)$, or equivalently that for any vector $v$ either $g^n(v) \to 0$ or $g^{-n}(v) \to 0$ as $n \to +\infty$; equivalently again, $\La$ admits a contracting and a repelling cone field. The next result establishes a tight relation between the existence of critical points and dominated splitting.
 
\begin{theorem}\label{p.existence}
Let $f\in Diff^1(M)$  and $\La$ be a  compact invariant set of $f$. If  $f^N_{|\La}$ is far from homotheties for all $N\geq 1,$ it follows that $\La$ admits a dominated splitting $T_\Lambda=E\oplus F$ with $\dim(E)=\dim(F)=1$ if and only if  ${\rm Crit}(f,\Lambda)$ is empty.

\end{theorem}

Combined with the main theorem of \cite{PS1}, the previous result yields a
two-dimensional version of Ma\~n\'e's theorem \cite{M} quoted at the beginning
of the introduction: in the absence of critical points, the dynamics is
hyperbolic.

\begin{corollary}\label{cor:mane2d}
For a generic $f\in Diff^2(M)$ the following holds: any compact invariant set
$\La$ such that $f^N_{|\La}$ is far from homotheties for all $N\geq 1,$ is a hyperbolic set if and
only if ${\rm Crit}(f,\La)=\emptyset$.
\end{corollary}

This dichotomy orients the whole paper. When ${\rm Crit}(f,\La)=\emptyset$
the set $\La$ admits a dominated splitting and is, in the generic $C^2$
setting, hyperbolic (corollary \ref{cor:mane2d}); this is the
``essentially hyperbolic'' regime, the two-dimensional counterpart of a
one-dimensional map without critical points. The substance of the theory
lies in the opposite regime, ${\rm Crit}(f,\La)\neq\emptyset$, where
domination fails and the critical set records precisely how: the remainder
of the paper studies what its presence, and conditions on its dynamics,
impose on the global behavior of $f$.

If one considers $C^2-$diffeomorphisms, not necessarily generic, it follows from \cite{PS2} that for any compact invariant set with a dominated splitting,  the limit set  can be decomposed into a finite number of normally hyperbolic closed curves with dynamics conjugate to irrational rotation, finite number of normally hyperbolic periodic intervals and compact sets conjugate to a hyperbolic ones such that the period of non-hyperbolic periodic points is bounded. For a detailed discussion of results about dominated splitting needed in this paper, see section \ref{sec:preliminaries}.  

\begin{remark}\label{rmk:why-powers}
Observe that the hypothesis of theorem \ref{p.existence} requires the far from
homotheties property not only for $f$ but for all its powers. The main reason is that  condition (\ref{FFH}) for $f$ alone does not pass
to the iterates. Consider an attracting periodic point $q$ of period $n$
whose derivative at the period $D_qf^{n}$ has its two eigenvalues of equal
modulus (a homothety up to a rotation; sufficient dissipation only forces
$|\det D_qf^{n}|\leq b^{n}$ and does not exclude this). Along the period the
individual derivatives $D_{f^j(q)}f$ need not be conformal, so $q$ may
satisfy (\ref{FFH}) for $f$ with some $\de$, while at scale $n$ one has
$g^{nk}(v)=1$ for every $v$ and every $k$, so the analogue of (\ref{FFH})
for $f^{n}$ fails at $q$ for every $\de$. A periodic domain, in
particular one of large period, can thus be far from homotheties for $f$
yet not for an iterate, and the obstruction to domination it carries is
visible only at the coarse time scale: this is exactly what the all-powers
hypothesis captures.
\end{remark}

One may ask whether the existence of a critical point in an invariant set is a property preserved under conjugacy. This is not the case for $C^0-$conjugacies: one can have two $C^0-$conjugate surface diffeomorphisms such that the invariant set of one of them is a hyperbolic set (and therefore does not contain critical points) and the other exhibiting a (cubic) tangency (see paragraph ii- in section \ref{examples} for more details). However, having critical points is a property invariant by $C^1-$conjugacies, since  dominated splittings are preserved by diffeomorphisms. 

One may further ask whether the notion of critical point is independent of the choice of metric; to address that issue, one could define  $x$ to be  a \emph{critical point}
if there exists $v\in T^1_xM$ and $K>0$ such that
$\forall n\in \ZZ, \; |g^n(v)|\geq K.$ The threshold $|g^n(v)| \geq 1$ depends on the choice of Riemannian metric, but the critical set $\mathrm{Crit}(f,\Lambda)$ does not. Since $M$ is compact, any two smooth Riemannian metrics are Lipschitz equivalent: there exists $C > 1$ such that $C^{-1}\|w\| \leq \|w\|' \leq C\|w\|$ for all tangent vectors $w$. As $g^n(v) = |\det(Df^n)|/|Df^n(v)|^2$, passing from one metric to the other multiplies $g^n(v)$ by a factor uniformly bounded between $C^{-2}$ and $C^2$, independently of $n$. Hence $|g^n(v)| \geq 1$ for all $n \in \mathbb{Z}$ in one metric implies $|g^n(v)| \geq C^{-2} > 0$ for all $n \in \mathbb{Z}$ in the other, and $x \in \mathrm{Crit}(f,\Lambda)$ with respect to one metric if and only if it belongs to $\mathrm{Crit}(f,\Lambda)$ with respect to the other. Throughout the paper we fix an arbitrary smooth metric and adopt the threshold $K=1$ for notational convenience.

\paragraph{b-- Hypothesis that imply that  $f_{|\La}$ is far from homotheties.}

The main result of this section asserts that, under a quantitative form of dissipation that we will make precise below, any iterate  $f^N_{|\La}$ is far from homotheties off the basins of any attracting periodic points.  The mechanism (for the particular case of $N=1$) is the following: if a direction $v$ at a point $x$ satisfies inequalities (\ref{FFH}), then by (\ref{eq_f-g}), $|Df^n(v)|^2= |\det(Df^n)|/g^n(v)$ is squeezed between $|\det(Df^n)|/(1+\de)^n$ and $|\det(Df^n)|/(1-\de)^n$. When $f$ is sufficiently dissipative, this asymptotic proximity to homotheties forces the orbit into the basin of an attracting periodic point whose local basin contains a ball of definite size. Compactness of $M$ then bounds the number of such attractors, and it  follows that $f$ is far from homotheties. The condition we now introduce is the explicit numerical form of sufficient dissipation that makes this argument go through. A similar argument holds for every iterate of $f$, and therefore removing all the basins of attraction of all sinks, the result follows.

\begin{definition} Let $f\in \diff^1(M^2)$ and let $\La$ be a compact invariant set such that $f_{\La}$ is dissipative. If $f$ satisfies  
\begin{equation}\label{cond-diss-FFH}
   \sqrt{b}\,M_0 + 2b  <1/2,
\end{equation}where $b=\max_{x\in \La}|det(D_xf)|$, $M_0=\max_{x\in \La}||D_xf||$, we  say that {\it $f$ is sufficiently dissipative}.
    
\end{definition}

\begin{theorem}\label{dissipative-FFH} 
Let $f\in \diff^1(M^2)$ and let $\La$ be a compact invariant set such that
$f_{|\La}$ is sufficiently dissipative. Then any power of  $f$ restricted to the set of
points of $\La$ that do not belong to the basin of attraction of any
attracting periodic orbits is  far from homotheties. 
\end{theorem} 


The proof proceeds by a finite-time version of the argument sketched above, whose main ingredient is proposition \ref{prop:pinning}. The sufficient dissipation hypothesis can plausibly be weakened to mere dissipation; this would require a stronger version of proposition \ref{prop:pinning} and is discussed in subsection \ref{subsec:relaxing}. 

We now illustrate the sufficient dissipation condition for the  H\'enon maps of the plane.

\begin{remark}\label{Henon-diss}{\bf H\'enon maps and attracting domains. } Let $f_{\mu,b}$ be the classical H\'enon family of diffeomorphisms of the   real plane $$f_{\mu, b}(x,y)=(x^2+\mu+b.y,x).$$ For dissipative H\'enon maps (i.e., $|b|<1$), there exists a connected open disk  $P$ of the parameter space $\{(\mu,b): |b|<1,\, \mu\in \RR\}$ such that for each $(b, \mu)\in P$ it follows that there is a disk $\DD_{\mu,b}$ such that:

\begin{itemize}
    \item[--] $f_{\mu,b}(\DD_{\mu, b})\subset \DD_{\mu,b},$
    \item for any $(x,y)\in \RR^2$ either there is a forward iterate that is contained in $\DD_{\mu,b}$ or the forward iterates converge to $-\infty.$ 
\end{itemize}
    
\end{remark}

\begin{remark}\label{Henon-FFH}{\bf H\'enon maps and far from homotheties. } If $\sqrt{|b|}< \frac{\sqrt2-1}{2}\approx 0.2071,$ it follows that ${f_{\mu,b}}_{|D_{\mu,b}}$ is sufficiently dissipative.
    
\end{remark}

\paragraph{c-- Properties of the critical set.} In what follows we collect several properties of the critical set. The first group concerns recurrence behavior, paralleling classical results on critical points in one-dimensional dynamics. The second concerns the dependence of the critical set on the diffeomorphism under small perturbations.

\begin{proposition}\label{properties}
Let $f\in Diff^1(M^2)$ and let $\La$ be a compact invariant set, the following properties hold
 \begin{itemize}


  \item[--] If a critical point is periodic, then the derivative at the period has only one eigenvalue: the derivative at the period is either a homothety, elliptic, or parabolic (up to scalar multiplication). In particular, in the dissipative context, the periodic point is an attracting periodic point.

\item[--]  If $f_{/\La}$  is far from homotheties, then it cannot hold that all iterates of the critical point and the critical direction  are  critical points and critical directions.

    \item[--] The critical set is closed and depends semi-continuously on $f$: for any  neighborhood $U$ of ${\rm Crit}(f,\Lambda)$ there exists a neighborhood $V$ of $\La$ such that if $g$ is sufficiently close to $f$ and $\La'$ is a compact invariant set of $g$ contained in $V$ then ${\rm Crit}(g,\Lambda') \subset U.$

 \end{itemize}

\end{proposition}

The second item of the proposition for dissipative dynamics is similar to  an obvious result for one-dimensional endomorphisms: if a critical point is periodic, then it is an attracting periodic point. In the dissipative context, the proof proceeds by considering the return of a critical direction under the iterate that makes the critical point periodic: if the direction is invariant, then the periodic point is an attracting periodic point such that the derivative has   either one or two invariant directions (with the same rate of contraction); if the direction is not invariant, then the derivative has a complex eigenvalue with modulus smaller than one.  

The following remark and proposition  are about relating the critical direction of a critical point with the subbundles of a set with dominated splitting that exists in the complement of the critical set. The next remark, describes this relation under the assumption that the critical point is not forward (backward) recurrent.
To be precise,  let us recall first that a dominated splitting $E\oplus F$ of a compact invariant set $\La$ can be extended to an invariant dominated splitting $\tilde E\oplus \tilde F$ on a small neighborhood $U$ of $\La$ in the sense that if $x\in U$ and $f(x)\in U$ then $D_f(\tilde E_x)= \tilde E_{f(x)}$ and $D_f(\tilde F_x)= \tilde F_{f(x)}$. The extension $\tilde F$ is unique for the points that all its backward iterates remain in $U$ and similarly the extension $\tilde E$ is unique for the points that all its forward iterates remain in $U$. Moreover, if $U$ is a sufficiently small neighborhood it holds that $\tilde F (\tilde E)$ is close to $F$ in the sense that given $x\in U$ there is $x'\in \La$ such that the slope between $\tilde F_x$ and $F_{x'}$ is small (respectively with $\tilde E_x$ and $\tilde E_{x'}$).

\begin{remark}\label{rm:crit-dom} Given a critical point $c$, if the  $\omega-$limit of $c$ (denoted as $\omega(c)$) is disjoint from the critical set, then $c$ belongs to the center-stable manifold of $\omega(c)$ and the forward iterates of the critical direction coincide with $\tilde E$, which is uniquely defined on the forward orbit of $c$. If  the $\alpha-$limit of $c$  (denoted as $\alpha(c)$) is disjoint from the critical set, then $c$ belongs to the center-unstable manifold of $\alpha(c)$ and the backward iterates of the critical direction coincide with $\tilde F$ (uniquely defined on the backward orbit of $c$). 
    
\end{remark}
This follows immediately from the fact that if the critical direction does not coincide with the center-unstable subbundle $\tilde F$ defined in a neighborhood of $\La$,  then the backward iterates converge to $\tilde E$ and therefore $g^{-n}(v)\to 0,$ contradicting the definition of critical direction. Under the non-recurrence assumption above, a critical point behaves as the point in the orbit of tangency where the unstable and stable directions are interchanged (see paragraph i- of \ref{examples}).

A more general statement, which allows for recurrence of critical points to the critical set, is the following compatibility result between the subbundles of a dominated splitting and the critical directions of a critical point.

\begin{proposition}\label{pr:compatibility} Let $f\in \mathrm{Diff}^1(M^2)$ and let $\Lambda$ be a compact invariant set. Suppose there exists $\delta_0>0$ such that $\Lambda_{\delta_0}$ is non-empty, where $\Lambda_{\delta_0}=\bigcap_{n\in \mathbb{Z}} f^n\!\left(\Lambda\setminus B_{\delta_0}(\mathrm{Crit}(f, \Lambda))\right)$ .

Then, for any $\varepsilon>0$ and $\delta<\delta_0$, there exists a positive integer $N_0$ such that the following holds. If $c$ is a critical point satisfying $f^{-j}(c)\notin B_{\delta}(\mathrm{Crit}(f,\Lambda))$ for all $0<j\leq n$ with $n>N_0$, then there exists a positive integer $n^*\leq n$ such that:
\begin{itemize}
    \item[--] $f^{-n^*}(c)\in B_\varepsilon(\Lambda_{\delta})$, and
    \item[--]  $\mathrm{slope}(G^{-n^*}(v), F)< \varepsilon$, where $F$ is the center-unstable subbundle of $\Lambda_{\delta}$.
\end{itemize}
Similarly, if $c$ is a critical point satisfying $f^{j}(c)\notin B_\delta(\mathrm{Crit}(f,\Lambda))$ for all $0< j\leq n$ with $n>N_0$, then there exists a positive integer $n^*\leq n$ such that:
\begin{itemize}
    \item[--] $f^{n^*}(c)\in B_\varepsilon(\Lambda_\delta)$, and
    \item[--]   $\mathrm{slope}(G^{n^*}(v), E)< \varepsilon$, where $E$ is the center-stable subbundle of $\Lambda_\delta$.
\end{itemize}

\end{proposition}

In short, the previous proposition states that the critical direction behaves like a center-unstable direction when iterated backward (provided the critical point takes a long time to return backward to a neighborhood of the critical set), and like a center-stable direction when iterated forward (provided it takes a long time to return forward).

A simple consequence of that proposition is that for points close to a set with dominated splitting whose forward (backward) iterate approaches a critical point, the forward (backward) iterate of the center-unstable (stable) direction approaches the critical direction of that critical point.

\begin{corollary}\label{cor:compatibility} 
Let $f\in \mathrm{Diff}^1(M^2)$ and let $\Lambda$ be a compact invariant set. Suppose there exists $\delta_0>0$ such that the set $\Lambda_{\delta_0}=\bigcap_{n\in \mathbb{Z}} f^n\!\left(\Lambda\setminus B_{\delta_0}(\mathrm{Crit}(f, \Lambda))\right)$ is non-empty.

For any $\varepsilon>0$ and $\de_1<\de_0$, there exist a positive integer $N_0$, and constant $\delta_2>0$ such that the following holds. If  $f^{j}(x)\notin B_{\delta_1}(\mathrm{Crit}(f,\Lambda))$ for all $0\leq j\leq n-1$ with $n\geq N_0$, and $f^n(x)\in B_{\delta_2}(\mathrm{Crit}(f,\La))$, then 

\begin{itemize}
    \item[--] there is a positive integer $n^*<n$ such that $f^{n^*}(x)\in B_\e(\La_{\de_1}),$
    \item[--]  
there exists $c\in \mathrm{Crit}(f, \Lambda)$ such that $\mathrm{slope}(G^{n-n^*}(\widetilde{F}_{f^{n^*}(x)}), v_c)<\varepsilon.$
\end{itemize}

Similarly, if  $f^{-j}(x)\notin B_{\delta_1}(\mathrm{Crit}(f,\Lambda))$ for all $0\leq j\leq n-1$ with $n\geq N_0$, and $f^{-n}(x)\in B_{\delta_2}(\mathrm{Crit}(f, \Lambda)$ then 

\begin{itemize}
    \item[--] there is a positive integer $n^*<n$ such that $f^{-n^*}(x)\in B_\e(\La_{\de_1}),$
    \item[--] there exists $c\in \mathrm{Crit}(f, \Lambda)$ such that $\mathrm{slope}(G^{-(n-n^*)}(\widetilde{F}_{f^{-n^*}(x)}), v_c)<\varepsilon.$
\end{itemize}

\end{corollary}

In one dimension, exponential self-recurrence of a critical point forces it into the basin of an attracting periodic point; Proposition \ref{exponential recurrence} below is the two-dimensional analog, valid under no exponential growth of the norm.

\begin{proposition}\label{exponential recurrence}
Let $f$ be a $C^2$ surface diffeomorphism and $\La$ be a dissipative invariant compact set,  satisfying $\underset{k\to +\infty}\limsup \tfrac 1 k \log \|Df^k|_\La\|\leq 0$ which contains a critical points $c$
with exponential recurrence, i.e. $\underset{k\to +\infty}\liminf \tfrac 1 k \log d(f^k(c),c)< 0.$
Then $\La$ contains a sink.
\end{proposition}

\paragraph{d-- Critical points for dissipative dynamics of the disk.}

In what follows, we assume that $\DD$ is an open disk that is properly invariant by $f$, that is $\overline{f(\DD)}\subset \DD$ and $f: \DD\to f(\DD)$ is a diffeomorphism.To simplify notation, we denote this class of maps by $ Diff^r(\DD)$ and observe that any map in that class can be extended to a diffeomorphism of the sphere.

In dimension one, every unstable branch of a repelling periodic point either contains a critical point or lies in a periodic interval on which the dynamics is a diffeomorphism. Theorem \ref{disk2} is the two-dimensional analog, for mildly and sufficiently dissipative diffeomorphisms of the disk; the second option of the 1D statement is replaced by the existence of a normally hyperbolic invariant arc, since the map is already a diffeomorphism.

\begin{definition}
      Given a boundaryless surface $S$ and a dissipative  smooth diffeomorphisms $f:S\to f(S)\subset S$,  it is said that $f$ is {\it mildly dissipative} if for any ergodic measure $\mu$ which is not supported on an attracting periodic point,and   for $\mu$-almost every point, each of its  two stable branches   meets the boundary of $S$.
 \end{definition}

In \cite{CP} it was shown
that mild dissipation is satisfied for large classes of systems. For instance, it holds for some $C^2$
open sets of diffeomorphisms of the disc including diffeomorphisms close to one-dimensional endomorphisms, and also for polynomial automorphisms of $\RR^2$ whose
Jacobian is  close to zero, including the diffeomorphisms from the H\'enon family
with Jacobian of modulus less than $1/4$ (up to restricting to an appropriate trapped disc).  This class captures certain properties of one-dimensional maps, but keeps two-dimensional
features showing all the well-known complexity of dissipative surface diffeomorphisms.
The dynamics of the new class, in some sense, is intermediate between one-dimensional
dynamics and general surface diffeomorphisms.

The two-dimensional result that resembles the one-dimensional stated before  is proved for the closure of an unstable branch; moreover, the second option that holds in the one-dimensional situation does not make much sense, given that the map is already a diffeomorphism. So, that option is replaced by saying that the accumulation set  of the unstable branch is in the basin of attraction of a periodic point or it is an invariant closed arc $\ga$ that is {\it normally hyperbolic}; i.e. $\ga$ has a dominated splitting $E^s\oplus F$ such that $F$ is tangent to the curve $\ga$ and $E^s$ is a contracting direction. To be precise, given a saddle periodic point $p$, which we may assume without loss of generality to be fixed,  the accumulation set of one of its unstable branches (which can also be assumed fixed) is the accumulation set of the union of the iterates of a fundamental domain of that unstable branch. This set is an invariant connected set that contains $p$.

 \begin{theorem}\label{disk2} Let $f\in Diff^2(\DD)$ such that $f_{|\DD}$ is mildly and sufficiently dissipative. Given an unstable branch  $\Ga^u(p)$ of a saddle fixed point $p$ either:

\begin{enumerate}


\item[--]  $\Ga^u(p)$ is a normally hyperbolic attracting  invariant arc with one endpoint $p$ and the other is a (semi-)attracting fixed point, 

 \item[--] or there is a critical point either in  $\Ga^u(p)$ or in the accumulation set of $\Ga^u(p)$.

    \end{enumerate}


 
\end{theorem}

Observe that in the previous theorem it is not claimed that the critical point belongs to the limit set. Examples of that situation are considered in  section \ref{examples}, paragraph {\em vi}. 

\paragraph{f-- Misiurewicz diffeomorphisms.}

A smooth one-dimensional endomorphism is Misiurewicz if all its critical points are non-recurrent; in that case the dynamics is uniformly expanding away from the critical set. We introduce the two-dimensional analog (Definition below) and prove, in Theorem \ref{thm:finite classes}, that for mildly and sufficiently dissipative Misiurewicz diffeomorphisms of the disk the number of generalized homoclinic classes is finite.



\begin{definition}
 Let $\La$ be a compact invariant set of a $C^{2}$ surface diffeomorphism $f$. We say that $f_{|\La}$ is a  {\it Misiurewicz diffeomorphism} if there exists a neighborhood $V$ of the critical set ${\rm Crit}(f, \La)$ such that  any $c\in {\rm Crit}(f, \La)$ satisfies $\al(c)\cap  V=\emp$ and $\omega(c)\cap V =\emp$.
\end{definition}





\vskip 3pt

%


Differently that in the one dimensional case where the Misiurewicz condition only involves the forward orbit
of the critical points, in the two-dimensional case one has also to consider the backward iterates.   This is not
a technical precaution but reflects the structure of a critical point
described in remark \ref{rm:crit-dom} and proposition \ref{pr:compatibility}:
the critical direction behaves as a center-stable direction in the future and
as a center-unstable direction in the past, so the critical orbit interacts
with the rest of the dynamics through both of its $\alpha-$ and $\omega-$ limits, exactly as a point
of tangency, whose forward orbit follows a stable manifold and whose backward
orbit follows an unstable one. Accordingly, the proofs of section
\ref{sec:misiu} use the condition $\omega(c)\cap V=\emptyset$ and
$\alpha(c)\cap V=\emptyset$ in an entirely symmetric way.
 
The known examples of Misiurewicz maps are described in section \ref{examples}, in paragraph ii. Naturally, one can wonder if these are the only natural classes of examples. More broadly, one can ask how the dynamics on the critical set constrains the global dynamics on the non-wandering set; the natural question is the following:

\begin{question} \label{misiu3}Let $\La$ be a compact invariant set of a $C^{2}$  surface diffeomorphisms $f$. If  $f_{|\La}$ is   Misiurewicz does it follow that $\Lambda$ can be decomposed into a finite number of transitive sets and a finite number of normally hyperbolic arcs?
 
\end{question}

If one  assumes that all the periodic points are hyperbolic, and  if the previous question is answered positively, then the dynamics  decomposes into finitely many transitive pieces. If one allows non-hyperbolicity,  one can have infinitely many  transitive isolated sets; for instance,  a normally hyperbolic arc such that the restriction to that arc is a one-dimensional diffeomorphism exhibiting infinitely many attracting fixed points (as in the case of the map $x\to x^r\sin(\frac 1{x})$). 

In the following theorem \ref{thm:finite classes}, we provide  a partial answer  to question \ref{misiu3} for mildly dissipative diffeomorphisms of the disk  focusing on non-trivial {\it homoclinic classes} instead of general transitive sets. In this paper, a homoclinic class is the closure of the set of  {\it crossing} points of the stable and unstable manifolds of a saddle periodic point (see definition \ref{defi:crossing}).    In the classical definition, only transversal intersections are considered: an intersection is transversal when the tangent spaces of the stable and unstable manifolds at the intersection point are not collinear. We instead allow crossing intersections, meaning that the unstable manifold meets both connected components of any local-stable ``band'' around the intersection point and for brevity we keep calling  these classes as homoclinic classes (see definition \ref{defi:generalizedHC} in section \ref{sec:misiu} for details) .

\begin{theorem}\label{thm:finite classes} Let $f$ be a mildly and sufficiently dissipative     $C^2-$diffeomor\-phism of the disk, and let
$\La=\bigcap_{n\geq 0}f^n(\DD)$ be its maximal invariant set, which is
compact and invariant. Suppose that  $f_{|\La}$ is Misiurewicz. Then, the number of non-trivial  homoclinic classes is finite.

Moreover, there exists a positive integer $n_0$ such that any saddle periodic point  with period larger than $n_0$ not contained in any of the finitely many non-trivial homoclinic classes,  belongs to  a non-trivial  normally hyperbolic periodic arc.
    
\end{theorem}

 Observe that there could exist infinitely many saddle periodic orbits that do not belong to a non-trivial homoclinic class and there  could also exist infinitely many normally hyperbolic arcs. 

We note that in the one-dimensional case, Misiurewicz maps are not $C^2$-open:
they are typically accumulated by maps for which the critical point is recurrent
but not periodic. The analogous question for surface diffeomorphisms is beyond
the scope of the present paper. Here our goal is more modest: to show that the
non-recurrence condition on the critical set has strong dynamical consequences,
namely the finiteness of the number of homoclinic classes, in the spirit of the
one-dimensional case where non-recurrence of critical points implies a finite
decomposition into transitive pieces.

\paragraph{A dictionary.} The following table summarizes how the results of
this paper transpose the classical role of critical points in one-dimensional
dynamics to surface diffeomorphisms.

\begin{center}
\begin{tabular}{p{0.42\textwidth}|p{0.42\textwidth}}
{\bf One-dimensional endos} & {\bf Surface diffeomorphisms}\\
\hline\hline
critical point: $f'(c)=0$ &
critical point: definition \ref{def:crit}, via the projective cocycle
\\
\hline
no critical points $\Rightarrow$ essentially hyperbolic (Ma\~n\'e \cite{M}) &
theorem \ref{p.existence}; corollary \ref{cor:mane2d} (generic $C^2$,
far from homotheties)\\
\hline
a periodic critical point is attracting &
proposition \ref{properties}, first item\\
\hline
exponentially recurrent critical point lies in the basin of an attracting
periodic orbit &
proposition \ref{exponential recurrence}\\
\hline
an unstable branch of a repelling periodic point contains a critical point,
or lies in a periodic interval where the map is a diffeomorphism &
theorem \ref{disk2}: critical point in the closure of the branch, or
normally hyperbolic invariant arc\\
\hline
Misiurewicz: non-recurrent critical orbits $\Rightarrow$ finitely many
transitive pieces &
Misiurewicz diffeomorphisms (paragraph f-); theorem
\ref{thm:finite classes}: finitely many generalized homoclinic classes\\
\end{tabular}
\end{center}


\paragraph{Organization of the paper.}

The paper is organized as follows. Section \ref{examples} collects examples that calibrate the definition of the critical set against known constructions: tangencies between stable and unstable manifolds, Misiurewicz maps in the H\'enon family, recurrent critical points without exponential growth of the projective cocycle, the constructions of Benedicks--Carleson and Wang--Young, odometers and unicritical sets, and critical points outside the limit set. Section \ref{sec:preliminaries} gathers, for reference, the results on dominated splitting that are used in the sequel; this section is mostly expository and follows \cite{PS1,PS2}. Section \ref{proof-existence} proves Theorem \ref{p.existence}, the characterization of dominated splitting by the absence of critical points under the far-from-homotheties hypothesis. Section \ref{sec:dissipative} proves Theorem \ref{dissipative-FFH}, that under sufficient dissipation every power of $f$ is far from homotheties off the basins of the attracting periodic orbits; this is the longest and most technical part of the paper, organized around a singular-frame factorization of the derivative cocycle, and concludes in subsection \ref{subsec:relaxing} with a discussion of how the sufficient dissipation hypothesis might be weakened. Section \ref{sec:properties} establishes the properties of the critical set listed in Proposition \ref{properties}, including the compatibility between critical directions and the subbundles of a nearby dominated splitting (Proposition \ref{pr:compatibility}). Section \ref{sec:disk} proves Theorem \ref{disk2} on the dichotomy for unstable branches of saddle points for mildly and sufficiently dissipative diffeomorphisms of the disk. Section \ref{sec:misiu} introduces a generalized notion of homoclinic class and proves Theorem \ref{thm:finite classes} on the finiteness of such classes for mildly and sufficiently dissipative Misiurewicz diffeomorphisms. Section \ref{sec:questions} collects some questions left open by this work, concerning the structure of the critical set, the strengthening of the critical-point condition, and the extraction of SRB measures from conditions on the critical orbit.

\section{Examples}
\label{examples}

In this section we present examples of non-hyperbolic diffeomorphisms for which the critical set can be identified.

\paragraph{{\em i-} Tangencies and critical points.}

Homoclinic and heteroclinic tangencies have played a fundamental role establishing the existence of open sets of non-hyperbolic surface diffeomorphisms. In our context, these systems have critical points.  So, tangencies may provide a natural candidate for critical points.

\begin{definition} We say that $f$
has a {\it heteroclinic tangency} if there is a pair of periodic points $p, q$ such that
there is a point $x\in W^s(p)\cap W^u(q)$ with $T_xW^s(p)=
T_xW^u(q).$ In this case, we say that $x$ is a point of tangency.
\end{definition}

Lemma 1.3.2 in \cite{PRH}  relates the tangencies with the critical points.
Observe that if $x$ is a point of tangency, then any iterate of $x$ is also a point of
tangency. However, the lemma says that in the orbit of tangency there is exactly one   point $x$ and an integer $k$  such that $|g^{-n}_x(v)|\geq 1$ for any $n\geq 0$  and $|g^n_{f^k(x)}(G^k(v))|\geq 1$ also for any $n\geq 0,$ where $v$ spans  $T_xW^s(p)$ with $T_xW^u(q)$. In that sense, a tangency seems to be  a critical point up to an iterate. However, in the following lemma it is shown that given a tangency, there is an iterate that is a critical point in the sense of definition \ref{def:crit}.

\begin{lemma}\label{lem:tang-crit} Let $x$ be a point of  heteroclinic or homoclinic tangency then, there is an iterate of $x$ that it is a critical point.
    
\end{lemma}

\begin{proof} We provide two proofs, one that relies on theorem \ref{p.existence} and another one that follows from an easy claim about finite products of positive numbers (which also highlights the idea beyond the proof of theorem \ref{p.existence}). For simplicity, let us assume that the tangency is associated with a pair of fixed points.

For the first proof, consider the closure of the orbit of the tangency; this set consists of the orbit of the tangency together with the two fixed points where it accumulates; that set does not have dominated splitting and so by theorem \ref{p.existence} it contains a critical point that cannot be one of the fixed points.

The second proof follows from the next claim:

\begin{claim}\label{cl:product} Let $a_1,\dots, a_N$ be a sequence of positive numbers such that $\prod_{i=1}^N a_i\geq 1;$ then there exists $K\leq N $ such that $\prod_{i=0}^{j} a_{K-i}^{-1}\geq 1$ for $j=1,\dots K-1$ and $\prod_{i=1}^{j} a_{K+i}\geq 1 $ for $j=1,\dots, N-K.$
    
\end{claim}

The proof of this claim is simple; consider all cumulative products up to the index $j,$ i.e. $P_j=\prod_{i=1}^j a_i$; take  $K$ to be an index at which the cumulative product is minimal..

To apply the claim, observe that one can take a point $x'$ in the orbit of the tangency, a vector $v'$ (the common tangent direction of the stable and unstable manifolds at point $x'$) and a positive integer $N$ such that $|g^{n}(v')|\geq 1$ for any $n<0$ or $n>N$ and $g^n(G^k(v'))\geq 1$ for any $n>0$. Observing that $|g^N(v)|=\prod_{i=0}^{N-1}|g(G^i(v))|\geq 1$, by the claim one gets a positive integer $K$ such that $ g^{-j}(G^K(v'))=\prod_{i=0}^{j} g^{-1}(G^{K-i}(v))\geq 1$ for $j=1,\dots K-1$ and $g^j(G^K(v'))=\prod_{i=1}^{j} g(G^{K+i}(v'))\geq 1 $ for $j=1,\dots N-K;$ therefore, $g^{|n|}(G^k(v'))\geq 1$ for any $n.$ 
    
\end{proof}

Observe that in case of tangencies, the notion of critical point is stronger; in fact., if $x$ is a critical point associated with an orbit of tangencies, then for its critical direction $v$ it follows that there exist $\la>1$ and a positive integer $N_0$ such that for any  $n>N_0$ and $k>0$ holds $|g^k(G^n(v))|>\la^k$ and $|g^{-k}(G^{-n}(v))|>\la^k.$ In another interpretation, one can consider the derivative along the orbit of the tangency; this yields a sequence of matrices (indexed by $\ZZ$) $(H_n)$ such that: for any $n\neq 0$ the matrix $H_n$ is diagonal with the horizontal being contracted and the vertical being expanded for any $n<0$, and conversely,  with the horizontal being expanded  and the vertical being contracted for any $n>0$, and $H_0$ being a matrix that sends the vertical into the horizontal. 


The notion of tangencies can be extended beyond periodic points; in fact, one can consider tangencies between the stable and unstable manifold of points in a hyperbolic set (as was considered in the paper by Newhouse; see \cite{N}). As noted in the introduction, Kupka–Smale diffeomorphisms, those for which all periodic points are hyperbolic and all intersections between stable and unstable manifolds of periodic points are transversal, form a residual set and therefore exhibit no heteroclinic tangencies. Nevertheless, by Newhouse's results there is an open set of surface diffeomorphisms that are not hyperbolic and do not admit a dominated splitting; the diffeomorphisms in this open set therefore contain critical points. One can ask, though, whether if they could be related to heteroclinic tangencies associated with non-periodic points in a hyperbolic set (see paragraphs iii- and  iv- in the present section).


\paragraph{{\em ii-}  Misiurewicz diffeomorphisms.}  In \cite{R} the author introduces an example of a homoclinic class that is semiconjugate to a hyperbolic set exhibiting a quadratic homoclinic tangency taking place in the limit set.  In \cite{E}, a surface diffeomorphism on the boundary of the Anosov class (and topologically conjugate to an Anosov diffeomorphism) is constructed, exhibiting a cubic tangency associated with a heteroclinic point.
  Both are  examples  of Misiurewicz diffeomorphisms.

Another example of a Misiurewicz diffeomorphism can be found in the H\'enon family $ f_{\mu,b}(x,y)=(x^2+\mu+by,x)$.  For sufficiently large negative $\mu$, this map has a hyperbolic Cantor set. Increasing $\mu$, one can find a first bifurcation parameter $\mu_0$ at which the hyperbolicity of the Cantor set is destroyed due to an internal tangency of the stable and unstable manifolds of a saddle point (see \cite{T} and \cite{BS2}).
Moreover, for the H\'enon family with $|b|<1$, the parameter set on which the topological entropy attains its maximal value $\log 2$ has the following structure (see Theorem 2 of \cite{BS}): in its interior, the set of non-escaping points to infinity forms a hyperbolic horseshoe; on its boundary, the diffeomorphism exhibits an internal tangency.


All the examples above lie on the boundary of the Axiom A class. To show that there is exactly one critical point in each example, the one associated with the tangency, it suffices to show that at every point outside the orbit of the critical point there is a well-defined cone field contracted by some iterate of $f$.

\paragraph{{\em iii-} Recurrent critical points and no exponential growth of the projective cocycle.} In \cite{HMS} it is shown that there exist diffeomorphisms in the boundary of a Axiom A surface diffeomorphisms such that the lack of hyperbolicity arises from the presence of a cubic tangency between a stable manifold and an unstable manifold, one of which is not associated with a periodic point; moreover, that tangency provides a critical point which is recurrent and so this critical point is not a heteroclinic tangency associated with points in a hyperbolic set.

The following example shows that given a horseshoe and a transverse homoclinic point (not necessarily associated with a periodic point), the diffeomorphism can be deformed so that the transverse intersection becomes a tangency, thereby producing tangencies, and hence critical points, at recurrent points. Moreover, the construction can be carried out so that the critical point does not exhibit exponential growth of the projective cocycle.

\begin{proposition} \label{recurrent critical}

There exists a $C^1$-diffeomorphism  with an invariant compact set
conjugate to the shift on $\{0,1\}^\ZZ$ and with a single critical orbit.
The critical orbit can be taken  at any non-periodic point of $\{0,1\}^\ZZ$.

Moreover, the construction can be performed in such a way that the derivative of the projective cocyle does not grow exponentially. 
\end{proposition}

\begin{proof}
    Let us consider a dissipative diffeomorphism $f_0$ exhibiting a  classical horseshoe $K_0$ 
which is the maximal invariant set in an open domain $U$
and let us fix any non-periodic point $x\in K_0$.
We compose $f_0$ with a fibered rotation centered at $x$ with arbitrarily small
support and which almost sends $E^u_x$ to $E^s_x$.
We obtain a diffeomorphism $f_1$.
If the support $B_0$ of the rotation is small enough, a (narrow) cone field is still preserved,
the dynamics is still hyperbolic and conjugated to the shift,
the conjugacy is $C^0$-close to the identity,
the bundles $E^u,E^s$ are almost unchanged outside
the union of the $N_0$-iterates of the support $B_0$ of the rotation (where $N_0$ is large,
independent of the size of the support).
The new horseshoe $K_1$ is still the maximal invariant set in $U$.

We repeat the construction and get a sequence of diffeomorphisms $f_n$,
a sequence of horseshoes $K_n$, and a sequence of perturbation balls $B_n$.
One can suppose that the diameter of $B_n$ and the angle of the rotation that we perform at step $n$
are less than $2^{-n}$ (after step $n$ the angle between $E^s$ and $E^u$
at $x$ is less than $2^{-n-1}$, so the rotation at step $n+1$ by an angle
smaller than $2^{-n-1}$ can still close the cone).
The convergence is thus strong enough, $f_n$ converges $C^1$ to
a diffeomorphism $f$ and $K_n$ to an invariant compact set $K$,
still the maximal invariant set in $U$,
and topologically conjugate to the shift.
Of course we still have dissipation.

The diffeomorphism $f_0$ preserved two transverse cone fields.
Outside these cones $g$ contracts (in the future or in the past depending on the cone).
Outside $B_0$ the cones are disjoint, so any point outside $B_0\cup f_0(B_0)$
does not satisfy the definition of critical set, for any diffeomorphisms $f_n$
(they all coincide there).
We want to repeat this argument so that after step $n$,
the critical set is contained in $B_n$. At the limit the critical set is reduced
to $x$ (which can be checked to be critical).
After step $n$, we consider two transverse invariant cone fields,  $C^s$ and $C^u$,
that are very thin, so that for $m_n$ large enough,
$g^{m_n}$ is expanding only inside $C^s$ and $g^{-m_n}$ is expanding only inside $C^u$.
Moreover after perturbation, we suppose that $f_{n+1}$ still sends $C^u$
disjoint from $C^s$, so that no point outside $B_{n+1}$ can be critical after step $n+1$.
\end{proof}

\paragraph{{\em iv-} Critical points for non-uniform hyperbolic diffeomorphisms  in the  works of Benedicks-Carlesson and Wang-Young.}

The idea of identifying a critical set for two-dimensional maps, a set designed to play the role analogous to that of critical points in one dimension, goes back to \cite{BC}. In fact, this approach is implemented in the landmark paper of Benedicks-Carleson \cite{BC} to prove the existence of strange attractors for a positive measure set of parameters for the H\'enon family. In their setting the critical set is defined as a certain type of tangencies between the unstable manifold of a fix point and stable ones that are obtained after an inductive process and for certain parameters of the  H\'enon family (in particular, that tangency is  not a heteroclinic tangency associated with points in a hyperbolic set).
In \cite{WY}, the inductive construction of \cite{BC} is reworked, with built-in geometric properties for the critical set incorporated as part of the induction. 
The critical set introduced in \cite{WY} is intrinsically defined, characterized as the set of points on the attractor at which stable and unstable directions are interchanged. It is shown to have a special geometric structure: it is the intersection of a nested sequence of rectangles whose geometric and dynamical properties are explicit. For a precise description of the geometric properties of the critical set in 
this setting and its relation to the interchange of stable and unstable 
directions, we refer the reader to section~2 of \cite{WY}. 
Recently, the critical set has been used in \cite{BoS} to define a kneading theory for some H\'enon maps.


\paragraph{{\em v-} Critical points in Odometers and unicritical sets.}

In \cite{CLPY}, the dynamics of odometers and their renormalization is studied, introducing a class of infinitely renormalizable unicritical diffeomorphisms of the disk (with a non-degenerate ``critical point”). Odometers appear naturally for mildly dissipative diffeomorphisms of the disk. Since they do not admit a dominated splitting (see \cite{CPT}), they contain critical points. In \cite{CLPY} a it is assumed a that there is a {\it unique critical point} in the odometer and  the definition is strengthen by requiring the projective cocycle to grow exponentially along the critical direction in both time directions, i.e. $|g^n(v)| > (1+\de)^{|n|}$ for all $n$ and some $\de>0$. Under these conditions, together with unique ergodicity, one can show the existence of local stable and center-stable manifolds that are tangent at the critical point; this is used to prove that under renormalization, the dynamics converges to a unimodal one-dimensional endomorphism.

\paragraph{{\em vi-} Critical points outside the Limit set.} The classical Smale's horseshoes dynamic in the sphere provides a simple example of critical points that  are not in the limit set: the dynamics decomposes into a repelling fixed point and a trapping disk; inside the disk, the limit set is the union of a hyperbolic Cantor set and a fixed attracting point. The critical point is the ``fold'' of the unstable manifold of the hyperbolic set, lying in the basin of attraction of the attracting fixed point. This is in fact the typical situation for non-trivial dissipative Axiom A diffeomorphisms on a trapping disk. 

Another easy example arises by considering the unstable branch of a fixed hyperbolic saddle $p$ that is attracted by an attracting fixed point $q$: if the attracting fixed point has two different eigenvalues, it holds that associated with the smallest one there is a unique invariant strong-stable foliation in the local basin of attraction $q$ and observe that $|g(v)|>1$ for any vector tangent to that foliation. If the unstable manifold of $p$ is at some point tangent to this strong-stable foliation, then along the orbit of that point there is a critical point (for details, see paragraph i- in the present section and claim \ref{cl:product}). Similarly, if both eigenvalues of $q$ are equal, then for any vector $v$ in the tangent bundle of a point in the local basin of $q$ satisfies $|g^n(v)|\leq 1.$



\section{Preliminaries about dominated splitting}
\label{sec:preliminaries}
This notion was  introduced independently
by Ma\~n\'e, Liao, and Pliss,as a first step toward proving the stability conjecture: that structurally stable systems are hyperbolic. Simple examples of invariant sets of surface diffeomorphisms that exhibit a dominated splitting but are not hyperbolic include: normally hyperbolic periodic arcs; normally hyperbolic closed invariant curves on which the dynamics is conjugate to an irrational rotation; and homoclinic classes associated with non-hyperbolic periodic points. In \cite{PS2} a general description of the limit set was provided under the assumption of dominated splitting.

If $f_{|\Lambda}$ is dissipative and admits a dominated splitting $E \oplus F$ then $E$ is uniformly contracted by $Df$. Since the present paper deals with dissipative diffeomorphisms, we will write dominated splittings in the form $E^s \oplus F$, with $E^s$ uniformly contracted.

For a set $\La$ with a dominated splitting $E^s\oplus F$, there are two complementary $C^1$ local laminations: one is  the  local stable  lamination $\{W^{s}_\eps(x)\}_{x\in \La}$ satisfying $f(W^s_\e(x))\subset W^s_{\la.\e}(f(x))$ for some positive $\la$ smaller than one, and the other is the local center-unstable $\{W^{cu}_\eps(x)\}_{x\in \La}$ tangent to $F$. On the other hand, for any point $x$ and $\eps>0$ one can define   the local unstable set of size $\eps$, denoted as   $W^{u}_\eps(x)$, as the set of points $y$ such that $dist(f^{-n}(x), f^{-n}(y))<\eps$ for any positive integer $n$ and  $dist(f^{-n}(x), f^{-n}(y))\to 0$ when  $ n\to +\infty$.  It holds that in a set with dominated splitting, the local unstable is contained in the  center-unstable. 

The following lemma has been proved in \cite{PS1}.

\begin{lemma}\label{bounded-arcs} Let $\La$ be a set with dominated splitting $E\oplus F$ of a $C^2$ surface diffeomorphisms. There exist $\e>0$  and $\de>0$ such that for any  arc $J$ transversal to $E$ such that the forward iterates of $J$ remains in an $\e-$neighborhood of $\La$ and $length(f^n(J))<\de$ for all $n>0$ it holds that there is a normally hyperbolic periodic arc $I$ such that $J\subset W^s(I).$
    
\end{lemma}


The previous lemma has a stronger version when the set $\La$ is a continuous set. 

 \begin{lemma}\label{centerunstable2} Let $f\in Diff^2(M)$ and let $\La$ be a connected   compact invariant  with a dominated splitting $T_\La M=E^s\oplus F$. Let $U$ be a neighborhood of $\La$ where the dominated splitting can be extended. Let $J$ be a $C^1-$compact arc transversal to $E^s$ such that the forward iterates of $J$ are contained in $U.$ Then

 \begin{itemize}
     \item[--] either $length(f^n(J))\to \infty$ as $n\to +\infty,$

     \item[--] or the   length of a subsequence of  forward iterates of $J$ is uniformly bounded  above and   

     \begin{itemize}

       \item[--] either the forward iterates of $J$ converges to a normally hyperbolic periodic arc or 
         \item[--] $length(f^n(J))\to 0$ and  $J$ lies in the basin of attraction of a (semi)attracting periodic point. 
     \end{itemize}

 \end{itemize}

 \end{lemma}


The local center-unstable manifold $W^{cu}_\e(x)$, minus the point $x$, has two components that are called local center-unstable branches and denoted as $\Ga^{cu}_\e(x)$. As a consequence of lemma \ref{bounded-arcs} one can obtain the following dynamical characterization of the center  unstable  manifolds.

\begin{lemma}\label{centerunstable}
For any center-unstable branch $\Ga^{cu}_\e(x)$ of a point $x$ one has the following characterization:

\begin{itemize}

    \item[--] either $x$ is a (semi)attracting periodic point and there is $\e(x)$ such that $\Ga^{cu}_{\e(x)}$ is contained in the basin of attraction of that point,

    \item[--] or there exists $\e(x)$ and $N$ such that for any $n> N$ it holds that

    $f^{-n}(\Ga^{cu}_{\e(x)}(x))\subset \Ga^{cu}_\e(f^{-n}(x))$ and either
    \begin{itemize}
        \item $\Ga^{cu}_{\e(x)}(x)\subset W^u_\e(x)$ or

        \item $x$ is a periodic point, $\Ga^{cu}_{\e(x)}(x)$ is normally hyperbolic  and there is a sequence of periodic points of the same period of $x$ contained in $\Ga^{cu}_{\e(x)}(x)$ and accumulating on $x.$
        
    \end{itemize}

\end{itemize}

\end{lemma}

Previous lemma has a stronger version summarized in following remark.

\begin{remark} For recurrent non-periodic points, a stronger statement holds: both local unstable branches are ``large", or at least one is large and the other extends to an intersection with the stable manifold of a semi-attracting periodic point.

 More precisely, there is $\e>0$ such that for any recurrent non-periodic point one of the following holds
 
 \begin{itemize}
     \item[--] $W^{cu}_{\e}(x)=W^{u}_{\e}(x)$

     \item[--] one center-unstable branch, called $\Ga^{cu,+}_{\e}(x)$, coincides with $W^{u,+}_{\e}(x)$ and in the other branch $\Ga^{cu,-}_{\e}(x)$ contains a point $x'$ in $ W^s(q)$ for some   semi-attracting periodic point $q$, with $\Ga^{cu,-}_{[x, x']}(x)= W^{u,-}_{\e}(x)$, where $\Ga^{cu,-}_{[x, x']}(x)$ denotes the arc of center-unstable branch with extreme points $x$ and $x'$.   
 \end{itemize} 
    
\end{remark}


As a consequence of previous statement, one can get the following proposition that resembles the local product structure for hyperbolic sets but in the dominated splitting  case, for recurrent non-periodic points.

\begin{proposition}\label{prop:LPE} Let $\La$ be a compact invariant  set with dominated splitting  of a $C^2$ surface diffeomorphism. There exists a positive constant $\de$ such that for  any recurrent points   $x$ and $y$ at distance smaller than  $\de$, it holds that $W^{\si}_\eps(x)\cap W^{\tau}_\eps(y)\neq \emptyset$ for both $(\si,\tau)=(u,s)$ and $(\si, \tau)=(s,u)$.

\end{proposition}

The following proposition follows from theorem A and corollary 4.1.1 in \cite{PS2}.

\begin{proposition}\label{PS2-1} Let $\La$ be a compact invariant set with dominated splitting  of a $C^2$ surface diffeomorphism. There exists a positive integer $N_0$ and a positive constant $\de$ such that any pair of periodic points $p, q\in \La$ is homoclinically related provided that their periods are larger than $N_0$ and their distance is smaller than  $\de$.
    
\end{proposition}

Now, we state a few simple results (used later in section \ref{sec:misiu}) about sets with dominated splitting contained in the limit set without assuming that the whole limit set has dominated splitting.

Let $\La$ be a set with dominated splitting; it is said that $z\in \La$ is an {\it interior point} if there exist recurrent points $z', z'' \in \La$ arbitrarily close to each other such that 

\begin{itemize}
    \item[--] $W^{s}_\eps(z')\cap W^u_\eps(z'')\neq \emptyset $ and $W^{u}_\eps(z')\cap W^s_\eps(z'')\neq \emptyset;$ 

    \item[--] $z$ is in the interior of the rectangle bounded by the local stable and unstable arcs of $z'$ and $z''.$ 
\end{itemize}

It is said that a set $\La$ with dominated splitting is a {\it heterocycle} if it contains a finite number of  saddle periodic orbits and any other point in $\La$ is a transversal intersection point between the stable and unstable manifolds of some  periodic orbits in $\La.$  

The following remark  is about characterizing heterocycles and is used in section \ref{sec:misiu}.

\begin{remark}\label{rmk:heterocycle}
Observe that a heterocycle provides a multi-rectangle $S$ bounded by  the stable and unstable arcs of finite number of periodic points and the points in the heterocycle are either those periodic points or are contained in   those arcs. By the classical $\la-$lemma it follows that one can get an arbitrarily small rectangle $R$ containing a periodic point of the heterocycle in one corner and bounded by two parallel stable arcs and two parallel unstable arcs of that point. Moreover, if $\omega(x)$ is a heterocycle, any sufficiently large forward iterates of $x$ are in the interior of the multi-rectangle $S$ and there are forward iterates of $x$ in the interior  of $R.$
\end{remark}



    
    

The next proposition establishes a dichotomy about the possible limit sets with dominated splitting.

\begin{proposition}\label{dicho-omega} Let $f\in Diff^2(M)$ and let $x\in M$ such that $\omega(x)$ ($\alpha(x)$) has a dominated splitting. Then either

\begin{itemize}

    \item[--] $\omega(x)$ ($\alpha(x)$)  is a periodic orbit,  moreover, either  $x$ coincides with   the periodic orbit or it belongs to the stable set of that periodic orbit;

    \item[--] $\omega(x)$ ($\alpha(x)$)  contains interior points;
    
    \item[--] $\omega(x)$ ($\alpha(x)$) is a heterocycle.
    
\end{itemize}
    
\end{proposition}

\begin{proof}
Either $\omega(x)$ is finite or infinite. If it is finite, then $\omega(x)$
is a periodic orbit and $x$ either coincides with that periodic orbit or
belongs to its stable set.

Assume now that $\omega(x)$ is infinite, and let $z\in\omega(x)$ be a
\emph{recurrent} point; such a point always exists, since $\omega(x)$ is
compact and invariant and therefore contains a minimal set, all of whose
points are recurrent. Observe that every forward iterate of $z$ is again a
recurrent point, so that proposition \ref{prop:LPE} applies to any pair of
forward iterates of $z$.

If the forward orbit of $z$ is finite, then $z$ is a periodic point. If it
is infinite, one can consider three large forward iterates of $z$ that are
arbitrarily close to each other. Using proposition \ref{prop:LPE}, either
one of the three iterates \emph{lies between the other two} as described in
the definition of interior point, or two of them share the same local stable
manifold, or two of them share the same local unstable manifold. In the
first case one gets an interior point. In the second case, $z$ belongs to
the stable manifold of a periodic orbit; since $z$ is recurrent, $z$ belongs
to $\omega(z)$, which is that periodic orbit, and so $z$ is periodic. In the
last case, since the two iterates belong to the local unstable manifold of a
periodic point but are themselves arbitrarily large forward iterates of $z$,
again $z$ must be that periodic point.

Therefore, either $\omega(x)$ contains an interior point, or every recurrent
point of $\omega(x)$ is a periodic point. Let us assume from now on that the
second option holds. In that case there are only finitely many periodic
orbits in $\omega(x)$: otherwise, by proposition \ref{PS2-1}, two of them
with sufficiently large periods and sufficiently close to each other would
be homoclinically related, producing transversal homoclinic intersections in
$\omega(x)$ and hence interior points.

Given now an arbitrary point $y\in\omega(x)$, the recurrent points of
$\omega(y)\subset \omega(x)$ are periodic and belong to the finite family of
periodic orbits above; hence $\omega(y)$, having finitely many periodic
orbits as its only recurrent points, forces $y$ to be either one of those
periodic orbits or a point of the stable manifold of one of them. Applying
the symmetric argument to the backward orbit of $y$ (the points of
$\alpha(y)$ and proposition \ref{prop:LPE} for $f^{-1}$), one concludes that
$y$ is either periodic or belongs to the unstable manifold of one of the
periodic orbits of the family.

In conclusion, $\omega(x)$ consists of finitely many periodic orbits (which
are saddles) together with points belonging simultaneously to the stable
manifold of one of them and to the unstable manifold of another (possibly
the same). Since $\omega(x)$ has a dominated splitting, these intersections
are transversal, and one gets a heterocycle.
\end{proof}

To finish, we discuss the dynamical decomposition of the limit set provided by Theorem A in \cite{PS2}. It states  that if the limit set of a $C^2$ diffeomorphisms of a compact surfaces has dominated splitting, then it can be decomposed into a finite number of isolated periodic points,  finite number of normally hyperbolic arcs, a finite number of normally hyperbolic invariant closed curves with dynamics conjugate to irrational rotation and a finite number of homoclinic classes. 
 
\begin{remark}\label{decomposition}
 
 The same result stated in Theorem A in \cite{PS2} can be obtained if one restricts the limit set to an attracting open domain. More precisely, if $f$ is a $C^2$ surface diffeomorphism and $B$ is an open set such that $\overline{f(B)}\subset B$ and the limit set of $f$ restricted to $B$ has dominated splitting, then the same thesis holds. Moreover, if $B$ is a disk and $f_{|B}$ is dissipative, then there are no normally hyperbolic closed periodic arcs.
    
\end{remark}

\section{Proof of theorem \ref{p.existence}}
\label{proof-existence}

The proof of theorem \ref{p.existence}  uses the  following two lemmas. The first (Lemma \ref{l.H}) provides, under the far-from-homotheties hypothesis, points with bounded backward (resp. forward) projective growth in some direction. The second (Lemma \ref{DScond}) characterizes the existence of a dominated splitting in terms of a uniform decay condition on the projective cocycle. We then deduce Theorem \ref{p.existence} by showing that failure of the dominated splitting condition produces a critical point.

\begin{lemma}\label{l.H}  Let $f\in Diff^1(M)$ and let $\Lambda$ be a compact invariant set such that $f_{|\La}$ is far from homotheties, then the following sets are non-empty:
$$H^+(f,\Lambda)=\{x\in\Lambda, \; \exists \,\,v\in T_x^1M,\;
\forall n\in \NN, \; |g^{n}(v)|\geq 1\},$$
$$H^-(f,\Lambda)=\{x\in\Lambda, \; \exists\,\, v\in T_x^1M,\;
\forall n\in \NN, \; |g^{-n}(v)|\geq 1\}.$$
\end{lemma}
\begin{proof}
Let us assume by contradiction that $H^+(f,\Lambda)$ is empty. Then every
unit vector $v\in T^1_\Lambda M$ has some positive iterate $n$ with
$|g^{n}(v)|<1$. By compactness of the unit tangent bundle over $\Lambda$
and continuity, there exist $N_0\geq 1$ and $\theta<1$ such that every
$v\in T^1_\Lambda M$ has some $n\leq N_0$ with $|g^{n}(v)|\leq\theta$.

We iterate this. Given $v$, choose $n_1\leq N_0$ with
$|g^{n_1}(v)|\leq\theta$; then choose $n_2\leq N_0$ with
$|g^{n_2}(G^{n_1}(v))|\leq\theta$; and so on. After $k$ blocks the total
time $T_k=n_1+\dots+n_k$ satisfies $k\leq T_k\leq kN_0$ and, by the cocycle
relation, $|g^{T_k}(v)|\leq\theta^{k}\leq\theta^{T_k/N_0}=(\theta^{1/N_0})^{T_k}$.
Writing $a=\sup_{T^1_\Lambda M}\max(|g|,|g|^{-1})$ to control the iterates
between consecutive return times $T_k$, one gets, for every $N\geq 0$,
$$|g^{N}(v)|\leq a^{N_0}\,(\theta^{1/N_0})^{N},$$
so for $N$ large enough (depending only on $N_0$ and $\theta$),
$$\forall v\in T^1_\Lambda M,\qquad |g^{N}(v)|<1.$$
This is impossible: $G^{N}_x$ is a diffeomorphism of the fiber circle
$T^1_xM$ onto $T^1_{f^{N}(x)}M$ with derivative modulus $g^{N}$, so
$\int_{T^1_xM} g^{N}\,d\theta$ equals the length of $T^1_{f^{N}(x)}M$, which
equals that of $T^1_xM$; if $g^{N}<1$ everywhere this integral would be
strictly smaller, a contradiction. Hence $H^+(f,\Lambda)\neq\emptyset$. The
non-emptiness of $H^-(f,\Lambda)$ follows by applying the same argument to
$f^{-1}$.
\end{proof}
\bigskip
The second lemma is a slight variation of a proposition  proved in \cite{PRH} (see the Main Proposition in that paper); for completeness, we provide the proof.

\begin{lemma}\label{DScond} Let $f\in Diff^1(M)$ and let $\Lambda$ be a compact invariant set such that $f_{|\La}$ is far from homotheties. It follows that
$\Lambda$ admits a dominated splitting $T_\Lambda=E\oplus F$, with $\dim(E)=\dim(F)=1$
if and only if the following condition holds.
\begin{description}
\item[(*)] For any  $\delta>0$ arbitrary small, and any  $N\geq 1$ it holds that for any $x\in \Lambda$, there exists $v\in T^1_xM$, such that
for any $m\geq 0$ one has $|g^N(G^m(v))|<(1+\delta)^N$.
\end{description}
\end{lemma}

We prove the `if' direction, the `only if' being straightforward. In terms of the `if' direction, the following remark is essential:

\begin{remark}\label{rmk:single-pair}
The proof of the `if' part of lemma \ref{DScond} only uses condition (*): if there exist $\de$ with $2\de<\de_0$
(where $\de_0$ is the constant in the far from homotheties hypothesis) and
an integer $N\geq 1$ such that for every $x\in\La$ some $v\in T^1_xM$
satisfies $|g^{N}(G^m(v))|<(1+\de)^{N}$ for all $m\geq 0$, then $\La$ admits
a dominated splitting. Consequently, if $\La$ does not admit a dominated
splitting, then for \emph{every} $\de$ with $2\de<\de_0$ and \emph{every}
$N_0\geq 1$ there exists $x\in\La$ such that for any $v\in T^1_xM$ there is
$m\geq 0$ with $|g^{N_0}(G^{m}(v))|\geq (1+2\de)^{N_0}$.
\end{remark}

Before proving the previous lemma, we show how it is used to conclude theorem \ref{p.existence}. The proof needs a classical lemma due to Pliss, of which we state a short version  (for a more general statement see lemma 3.1 in \cite{CP}):
    \vskip 3pt

\paragraph{\bf Pliss's lemma.}{\it    
Given $0 < \gamma_0 < \gamma_1$ and $a > 0$, there exist $n_0 = n_0(\gamma_0,\gamma_1,a)$ and
$l = l(\gamma_0,\gamma_1,a) > 0$ such that for $n\geq n_0$ and  any sequence of numbers $\{a_i\}, 0\leq i\leq n$  satisfying 
$a^{-1} < a_i < a$ and $\prod ^n_{i=0} a_i > \gamma_1^{n}$,  there exist $1 \leq n_1 < n_2 < ....< n_r$ with $r > ln$ such that
$$\prod ^k_{i=n_j+1}  a_i > \gamma_0^{k-n_j},\,\,\,\, n_j < k < n.$$}
    \vskip 3pt


\begin{proof}[Proof of theorem ~\ref{p.existence}]

The ``only if" part of the theorem is clear, we will prove the ``if" one.

\medskip
 Let us fix $\delta>0$ small and let us set $N_0$  large. Since
$\La$ is assumed to have no dominated splitting, lemma \ref{DScond}, in
the single-pair form established at the end of its proof,
provides a point $x\in\Lambda$ satisfying the following:

\begin{description}
\item[(**)]
for  any $v\in T^1_xM$
there is $m^+_0\geq 0$ such that $|g^{N_0}(G^{m^+_0}(v))|\geq (1+2\delta)^{N_0}$.
\end{description}
That is possible by remark \ref{rmk:single-pair}: failure of a dominated splitting yields the failure of (*) at every pair $\delta,N_0$ with $2\delta<\delta_0$.

One can also apply lemma~\ref{l.H} to the $\alpha$-limit set of $x$. So, there is a point $y\in \alpha(x)$ and $u\in T_y^1M$ such that for any positive integer $k$ it holds that $|g^{-k}(u)|\geq 1$; so, it also holds that   $|g^{-k}(u)|\geq (1-\de)^k$. Choosing $m^-$ sufficiently large in such a way that $f^{-m^-}(x)$ is sufficiently close to $y$,  and by continuity  one can get a vector $w\in T^1_{f^{-m^-}(x)}M$ close to $u\in T^1_yM$  and $N$ large such that $\forall\, 0\leq k\leq N,\; |g^{-k}(w)|\geq (1-2\delta)^k.$ Moreover, $N$ is arbitrary large provided that $f^{-m-}(x)$ is sufficiently close to $y.$

Now, we take the vector $v:=G^{m^-}(w)\in T^1_xM$. On one hand, it holds that

\begin{equation}\label{e.negative}
\forall\,\, 0\leq k\leq N,\; |g^{-k}(G^{-m^-}(v))|\geq (1-2\delta)^k.
\end{equation}

On the other hand, one can  apply property (**) to $v\in T^1_xM$, and so there is $m^+_0\geq 0$ such that $|g^{N_0}(G^{m^+_0}(v))|\geq (1+2\delta)^{N_0}$. From the fact that $N_0$ was choosen large, one can apply  Pliss lemma and  one deduces that there is $m^+\geq 0$ such that
\begin{equation}\label{e.positive}
\forall\,\, 0\leq k\leq N,\;|g^k(G^{m^+}(v))|\geq (1+\delta)^k.
\end{equation}

Consider the largest integer $m$ satisfying that  $-m^-\leq m\leq m^+$ and \eqref{e.negative} still holds, that means, $m$ satisfies
that 

\begin{equation}\label{e.negative-bis}
 \forall  0\leq k\leq N+ m-m^- , \; |g^{-k}(G^{m}(v))|\geq (1-2\delta)^k. 
\end{equation}

We show by induction on $k=0,\dots, m^+-m$ that $|g(G^{m+k}(v))|> (1+\delta):$ for $k=0$  if $|g(G^{m}(v))|<(1+\delta)$  then $|g^{-1}(G^{m+1}(v))|>(1+\delta)^{-1}$  and since $(1+\de)^{-1}> 1-2\de$, combining with  (\ref{e.negative-bis}) this gives (\ref{e.negative-bis}) but starting at  $m+1$ contradicting the maximal property of $m.$ The inductive step for $k > 0$ is identical, applied at  $m+k$ instead of $m$. Combining with   \eqref{e.positive} this gives
$$\forall \,\, 0\leq k\leq m^+-m+N,\; |g^{k}(G^{m}(v))|\geq (1+\delta)^k.$$
Using the last inequality and (\ref{e.negative-bis}) we have thus found a point $x_{\delta,N}=f^m(x)$ and a direction $v_{\delta,N}$ which satisfy
$$\forall \,\, 0\leq k\leq N,\; |g^{k}(G^{m}(v_{\delta,N}))|\geq (1+\delta)^k,$$
$$\forall\,\, 0\leq k\leq N,\; |g^{-k}(G^{m}(v_{\delta,N}))|\geq (1-2\delta)^k,$$
Considering a limit point $z_N$ and a limit direction $v_N$ (obtained by compactness of $\La$ and the unit tangent bundle) when $\delta\to 0$ one deduces that 
$$\forall \,\, k\in \{-N,\dots, N\},\; |g^{k}(v_N)|\geq 1.$$
By letting $N\to +\infty$ one gets a critical point, more precisely, a diagonal extraction as $N \to \infty$ yields $(z, v_\infty) \in T^1 \Lambda$ with $|g^k(v_\infty)| \geq 1$  for all $k \in \mathbb{Z}$ so $z\in \mathrm{Crit}(f, \Lambda)$.
\end{proof}

\begin{proof}[Proof of lemma ~\ref{DScond}] Recalling how the notion of dominated splitting can be recast in terms of the projective cocycle, to prove that a set $\La$ has a dominated splitting, it is enough to prove that there exist positive integers  $n_0$ and $m_0$ and $\de>0$ such that for any $x \in \La$ it holds that there is $v\in T^1_xM$ such that \begin{equation}\label{dom}
|g^{n_0}(G^m(v))| < (1-\de)^{n_0},\,\,\forall m>m_0.
\end{equation}
Now, by  (*),  taking $\de$  small and $N=1$ (below we consider the general case for any $N\geq 1$),   for any $x$ there is  $v$ such that  $|g(G^m(v))| < (1+\de)$ for any $ m>0.$ Observe that in particular, it holds that  \begin{equation}\label{dom1}|g^n(G^m(v))| < (1+\de)^n,\,\,\,\forall n>0, \,m>0.
\end{equation}
The goal is to show that  this vector actually satisfies equation (\ref{dom}) for some $n_0$. If it is not the case, one gets sequences $(n_i), (m_i)$ such that $|g^{n_i}(G^{m_i}(v))| \geq  (1-\de)^{n_i}$; by Pliss's lemma, for some $\de'>\de$ but close to it, there is a sequence $(m_i')$ such  that $m_i-m_i'\to \infty$ and 
\begin{equation}\label{dom2}|g^n(G^{m_i'}(v))| > (1-\de')^n,\,\,\,\, 0\leq n\leq m_i-m_i'.
\end{equation}
Taking an accumulation point  $z, w$ of the sequence $(f^{m_i'}(x)), (G^{m_i'}(v))$, one get from equations (\ref{dom1}) and (\ref{dom2}) that $(1-\de')^n\leq g^n(w)\leq (1+\de)^n$ and that contradicts the fact that $f_{|\La}$ is far from homotheties.  Choosing $\delta$
(and hence $\delta'$) less than $\delta_0$, where $\delta_0$ is the constant in the far-from-homotheties hypothesis, one gets a contradiction.

Observe that the argument above only used condition (*) for 
$\de,N=1$. For $\de,N$ with $N>1$, the same
proof applies in the same way  to $F=f^{N}$: the projective derivative cocycle of
$F$ is $g_F=g^{N}$, so the hypothesis reads as condition (*) with
 $\de_F$ and $N=1$ for $F$, where $1+\de_F=(1+\de)^{N}$, and
$F_{|\La}$ is far from homotheties by the hypothesis of theorem \ref{p.existence}. A dominated
splitting for $F$ is a dominated splitting for $f$. In particular, if
$\La$ admits no dominated splitting, then condition (*) fails for every
 $\de$ and $N>0$ with $\de$ small.

\end{proof}

\section{Far from homotheties. Proof of theorem \ref{dissipative-FFH}}
\label{sec:dissipative}

Sufficient dissipation is inherited by every power of $f$. Write
$b=\max_{x\in\La}|\det(D_xf)|$ and $M_0=\max_{x\in\La}\|D_xf\|$; so by
submultiplicativity the corresponding quantities for $f^{N}$ satisfy $b_N\leq  b^{N}$ and    $M_{0,N}\leq M_0^{N}$. Therefore, 
$$
\sqrt{b_N}\,M_{0,N}+2b_N
\leq \sqrt b\,M_0+2b<\;\frac{1}{2},
$$
so $f^{N}_{|\La}$ is sufficiently dissipative for every $N\geq 1$.

After this observation, to prove theorem \ref{dissipative-FFH}, it is shown first that for each $N$, $f^N_{|\La}$ is far from homotheties off the basin of finite number of attracting periodic points.   In other words, it is proved  the following weak  version of 
theorem \ref{dissipative-FFH} for any power  $f^N$ of $f$:

\begin{thmB}\label{dissipative-FFH2}

Let $f\in \diff^1(M^2)$ and let $\La$ be a compact invariant set such that
$f_{|\La}$ is sufficiently dissipative. Then  there is a finite number of attracting periodic points of  $f$ such that $f$ restricted to the set of
points of $\La$ that do not belong to the basin of attraction of these finite
attracting periodic orbit is far from homotheties. 
\end{thmB}

Theorem \ref{dissipative-FFH} now follows immediately: {\it since  for each positive integer $N$ it is true that $f^N_{|\La}$ is sufficiently dissipative, one can apply theorem B' and after removing a finite number of attracting periodic points, it is true that $f^N_{|\La}$ is far from homotheties; so, after removing all the attracting periodic points of $f_{|\La}$ one finds that it satisfies that all its powers are far from homotheties.}

One can restate Theorem B' in the following equivalent form, which we then prove: there exist $\de$ and $N_0$ such that for any $x\in \La$ 
\begin{itemize}
    \item[--] if there is $v\in T^1_M$ satisfying   $g^n(v)>(1-\de)^n$ for $n=1\dots N_0$ then there is $0\leq k\leq N_0$ such that $g^k(v)>(1+\de)^k,$ 

    \item[--] if there is $v\in T_M$ satisfying   $g^n(v)<(1+\de)^n$ for $n=1\dots N_0$ then there is $0\leq k\leq N_0$ such that $g^k(v)<(1-\de)^k.$ 
\end{itemize}

The structure of the section is as follows. We first show (lemma \ref{finite-basin}) that a finite orbit with frequent contraction lands in the basin of an attracting periodic point with definite local basin; by compactness, there are only finitely many such attractors (remark \ref{finite-attracting}). The technical heart of the section is proposition \ref{prop:pinning}, which gives a bound on $\|A_2 A_1\|$ for two-step compositions under a pinning condition; the bound is small precisely when the sufficient dissipation condition holds. We deduce theorem \ref{dissipative-FFH} by applying the proposition to consecutive pairs of derivatives along an orbit satisfying \eqref{FFH}, extracting frequent contraction, and invoking lemma \ref{finite-basin}. Subsection \ref{subsec:relaxing} discusses how the threshold $\sqrt{b}M_0 < 1$ might be weakened by considering longer blocks of matrices.

\begin{lemma}\label{finite-basin}
    
Let $f\in Diff(M)$ and let $\la<1$, 
there exists $r>0$ and a positive integer $m$ such that given a finite orbit $\{f^j(x)\}_{j=0}^n$ with $n$ larger than $m$ and a subsequence of iterates $x_{j_i}$ such that   either 
\begin{itemize}
    \item[--] $f(x_{j_i})=x_{j_{i+1}}$ and $||D_{x_{j_i}}f||<\la$ or

    \item[--] $f^2(x_{j_i})=x_{j_{i+1}}$ and $||D_{x_{j_i}}f^2||<\la.$
\end{itemize} 
Then $x$ is in the basin of attraction of a periodic point of period smaller and equal than $m;$ moreover, the basin of attraction of that point contains a ball of radius $r.$
\end{lemma}

\begin{proof} By the continuity of $Df$  follows that there exist $r>0$ and $\la<\la'<1$ such that if $f(x_{j_i})=x_{j_{i+1}}$ then $f(B_r(x_{j_i}))\subset B_{\la'.r}(x_{j_{i+1}})$ and if $f^2(x_{j_i})=x_{j_{i+1}}$ then $f^2(B_r(x_{j_i}))\subset B_{\la'.r}(x_{j_{i+1}})$; since one can cover the set $\La$ with finite attracting  balls of size $r$, the lemma follows.
    
\end{proof}

\begin{remark}\label{finite-attracting} For any $r$, there exists a finite number of attracting periodic points that their basin contains a ball of radius at least $r.$ 
    
\end{remark}

\begin{definition}\label{defi:pinning} Let $A_1, A_2$ be a pair of matrices in $Gl(2,\RR), $ we say that they satisfies the \emph{pinning condition} if there exists a unit vector $v \in S^1$ such that
\[
g_{A_1}(v) = 1 \qquad \text{and} \qquad g_{A_2}\bigl(G_{A_1}(v)\bigr) = 1,
\] where $G_{A_i}$ and $g_{A_i}$ are the projective map and projective derivative on the unit circle for $A_1, A_2.$   
    
\end{definition}

The next proposition establish that given a pair of matrices $A_1, A_2$ satisfying the pinning condition and a relation between their norms, then $A_2.A_1$ is a contraction.  The hypotheses on their norms, $M_1=||A_1||, M_2=||A_2||$, below ($M_1 \geq 1$, 
$M_1 M_2 \geq 1$) are dictated by the application in the proof of Theorem \ref{dissipative-FFH}: the proposition will be invoked precisely in the case where neither $A_1$
 alone nor the pair $A_2A_1$ is already contracting.

\begin{proposition}\label{prop:pinning}
Let $b$ and $M_0$ be positive real numbers with $b < \frac{1}{M_0^2}< 1 < M_0.$
Let $A_1, A_2 \in \mathrm{GL}_2(\mathbb{R})$ be invertible matrices satisfying the pinning condition with $|\det A_i| = b$, and singular values $M_i$ (maximal) and $b/M_i$ (minimal) also verifying
\[
M_i \leq M_0 \quad (i = 1, 2), \qquad M_1 \geq 1, \qquad \frac 1{M_1} \leq  M_2.
\]
 Then it follows that 
\begin{equation}\label{boundA1A2}
\|A_2 A_1\| \leq \sqrt{b}\,(M_1 + M_2) + b\left(\frac{M_1}{M_2} + \frac{M_2}{M_1}\right) +  \frac{b^{5/2}(M_1+M_2)}{(M_1 M_2)^2}.
\end{equation}

In particular,
\begin{eqnarray}\label{boundbM0}
    \|A_2 A_1\| \leq 2\sqrt{b}\,M_0 + 4b.
\end{eqnarray}
\end{proposition}

The statement of the previous proposition may  appear surprising. Each individual matrix $A_i$ has top singular value $M_i$, which under our hypotheses may be as large as $M_0$, so the trivial submultiplicative bound gives only $\|A_2 A_1\| \leq M_1 M_2 \leq M_0^2$, a quantity that is in general much larger than $1$. The theorem replaces this naive bound by an estimate of order $\sqrt b\,(M_1 + M_2)$, which is small under the assumption $2\sqrt b\,M_0 < 1$. The mechanism behind this gain is geometric, and we describe it informally before turning to the proof strategy.

\noindent\emph{The geometric mechanism.} The pinning condition $g_{A_1}(v_1) = 1$ means that $\|A_1 v_1\| = \sqrt b$, a value strictly between the minimal contraction $b/M_1$ and the maximal expansion $M_1$. In the singular frame of $A_1$ (that is described in detail later), such a vector $v_1$ must lie at a very specific angle from the contracting direction $e^1_1$, namely $\arctan(\sqrt b/M_1)$, which is small: $v_1$ is close to the direction of \emph{maximum contraction}. Correspondingly, its image $v_2 = G_{A_1}(v_1)$ in the output frame of $A_1$ is close to the direction of \emph{maximum expansion} $f^2_1$, since the small component of $v_1$ along $e^2_1$ gets amplified by $M_1$ while the large component along $e^1_1$ is crushed by $b/M_1$.

The second part of the pinning, $g_{A_2}(v_2) = 1$, forces $v_2$ to also be close to the minimal contraction direction $e^1_2$ of $A_2$ (by the analogous argument applied to $A_2$). Combining the two observations:

\begin{quote}
The direction of \emph{maximal expansion output} of $A_1$ is approximately aligned with the direction of \emph{minimal contraction input} of $A_2$   (see figure \ref{pinning}).
\end{quote}

This alignment provides the conclusion: The most expanded image of $A_1$ is close to the most contracting direction of $A_2$ (which crushes it by a factor $b/M_2$). Their product is $M_1 \cdot b/M_2 = bM_1/M_2$, a quantity comparable to $b$ rather than $M_1 M_2$. The opposite pairing, bottom output of $A_1$ feeding into top input of $A_2$ — gives the symmetric small quantity $bM_2/M_1$. Both singular values of $A_2 A_1$ end up of order $b$, not $M_i^2$.

\begin{figure}[h]
\centering
\includegraphics[scale=0.7]{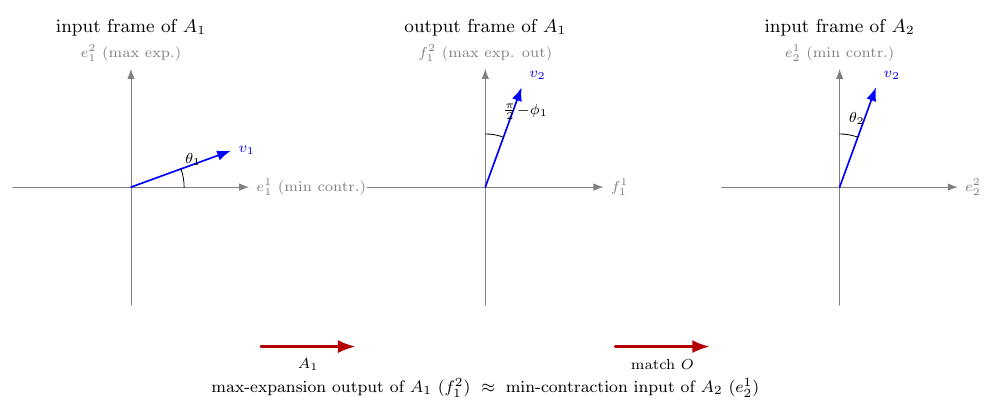}
\caption{Pinning argument}
\label{pinning}
\end{figure}

\noindent The proof formalizes this alignment by isolating the ``perfect alignment'' configuration and treating the actual configuration as a small perturbation of it.

\vskip 5pt

Let's move now to the  preparative material for the proof of proposition \ref{prop:pinning}.

\paragraph{Singular frames and projective dynamics}

Let $A$ be an invertible $2\times 2$ real matrix with $|\det A| = b > 0$ and singular values $b/M$ and $M$ for some $M > \sqrt{b}$. The singular value decomposition gives orthonormal frames defined below.

\begin{definition}[Singular frames]
The \emph{right singular frame} of $A$ is the orthonormal pair $(e^1, e^2)$ of unit vectors such that:
\begin{itemize}
\item $e^1$ is the direction of \emph{minimal contraction}: $\|A e^1\| = b/M$,
\item $e^2$ is the direction of \emph{maximal expansion}: $\|A e^2\| = M$.
\end{itemize}
The \emph{left singular frame} is the orthonormal pair $(f^1, f^2)$ defined by
\[
f^1 = \frac{A e^1}{\|A e^1\|} = \frac{M}{b}\, A e^1, \qquad f^2 = \frac{A e^2}{\|A e^2\|} = \frac{1}{M}\, A e^2.
\]
\end{definition}


\paragraph{Projective derivative and projective map in singular frame}

For $v \in S^1$ a unit vector, recall the definition of the projective map and its projective derivative associated to a linear transformation $A$:
$G_A(v) = \frac{Av}{\|Av\|}, \qquad g_A(v) = \frac{|\det A|}{\|Av\|^2} = \frac{b}{\|Av\|^2}.$
Writing $v = \cos\theta\, e^1 + \sin\theta\, e^2$ with $\theta \in [0, 2\pi)$, then
$g_A(v) = \frac{M^2}{b} \frac{1+\tan^2\theta}{1 + \frac{M^4}{b^2}\tan^2\theta}.$


In particular, let $A$ be as above with $b/M < \sqrt{b} < M$. The equation $g_A(v) = 1$, equivalently $\|Av\|^2 = b$, has exactly four solutions on $S^1$, given by $v = \pm v_+$ and $v = \pm v_-$, where
\[
v_\pm = \cos\theta\, e^1 \pm \sin\theta\, e^2 \quad \text{with} \quad \tan\theta = \frac{\sqrt{b}}{M}, \quad \theta \in (0, \pi/2).
\]


The image $w := G_A(v)$ in terms of the image singular frame is given by
\[
w = \cos\phi\, f^1 + \sin\phi\, f^2 \quad \text{with} \quad \tan\phi = \frac{M}{\sqrt{b}}, \quad \phi \in (0, \pi/2).
\]


\begin{remark}
The angle $\theta$ that $v$ makes with the contracting direction $e^1$ is small (of order $\sqrt{b}/M$), while the angle $\phi$ that $w = G_A(v)$ makes with the bottom output direction $f^1$ is large (close to $\pi/2$). Equivalently, $v$ is close to the contracting direction $e^1$ and $w$ is close to the expanding output direction $f^2$.
\end{remark}

\paragraph{Reduction to canonical coordinates (revised)}

For any orthogonal $Q$ the projective derivative satisfies $g_Q \equiv 1$, so the chain rule
\[
g_{AB}(v) = g_A\bigl(G_B(v)\bigr)\, g_B(v)
\]
implies that the pinning condition is preserved under such modifications.

We exploit this freedom as follows. Choose orthogonal $Q_1$ on the input side of $A_1$ so that the right singular frame of $A_1$ is mapped to the standard basis with the convention $
e^1_1 = (1,0) \quad (\text{minimal contraction}), \, e^2_1 = (0,1) \quad (\text{maximal expansion}).$

Choose orthogonal $Q$ between $A_1$ and $A_2$ so that the left singular frame of $A_1$ is also mapped to the standard basis: $
f^1_1 = (1,0), \qquad f^2_1 = (0,1).$ 
With these choices, $A_1$ takes the diagonal form
\[
A_1 = \begin{pmatrix} b/M_1 & 0 \\ 0 & M_1 \end{pmatrix}.
\]

For $A_2$, choose orthogonal $Q_2$ on the output side so that the left singular frame is mapped to the standard basis with the \emph{swapped} convention
$f^1_2 = (0,1)$ (image of minimal contraction direction),$f^2_2 = (1,0)$ (image of maximal expansion direction).

The right singular frame $(e^1_2, e^2_2)$ of $A_2$ remains a free orthonormal pair, to be determined by the pinning condition.

\paragraph{The matching rotation $O$:} Now we introduce an  orthogonal $O$ is defined to implement the {\it alignment} between maximal expansion of $A_1$ and minimal contraction of $A_2$.  More precisely, let $O \in SO(2)$ be the rotation that maps the vertical direction $(0,1) = f^2_1$ (the image under $A_1$ of its maximal expansion direction) to the minimal contraction direction $e^1_2$ of $A_2$: $ O(0,1)^T = e^1_2.$ By orthogonality of $O$, this also determines $O\,(1,0)^T = e^2_2$ (with sign fixed by the rotation orientation).

Under the pinning condition $g_{A_1}(v_1) = 1$ and $g_{A_2}(G_{A_1}(v_1)) = 1$, with the canonical-coordinate chosen, the rotation $O$ is given by
\[
O = \begin{pmatrix} \cos\theta & \sin\theta \\ -\sin\theta & \cos\theta \end{pmatrix},
\qquad
\theta = \arctan\!\frac{\sqrt{b}}{M_1} + \arctan\!\frac{\sqrt{b}}{M_2}.
\]
Equivalently,using that $\sin\arctan(x) = \frac{x}{\sqrt{1+x^2}},$ and $ \cos\arctan(x) = \frac{1}{\sqrt{1+x^2}},$ applied to  $x = \sqrt b/M_1$ and $x = \sqrt b/M_2$ it follows that 
\[
\sin\theta = \frac{\sqrt{b}\,(M_1 + M_2)}{\sqrt{(M_1^2 + b)(M_2^2 + b)}},
\qquad
\cos\theta = \frac{M_1 M_2 - b}{\sqrt{(M_1^2 + b)(M_2^2 + b)}}.
\]

In fact, let $v_2 := G_{A_1}(v_1)$. We locate $v_2$ from two viewpoints: The angle that $v_2$ makes with the bottom output direction $f^1_1 = (1,0)$ is $\phi_1 = \arctan(M_1/\sqrt b)$. Equivalently, the angle that $v_2$ makes with the vertical direction $f^2_1 = (0,1)$ is
$\frac{\pi}{2} - \phi_1 = \arctan\!\frac{\sqrt b}{M_1}.$ The angle that $v_2$ makes with the minimal contraction direction $e^1_2$ is $\theta_2 = \arctan\!\frac{\sqrt b}{M_2}.$

Choosing the orientation so that $v_2$ lies between the vertical $(0,1)$ and the direction $e^1_2$ on the unit circle, the total angle from $(0,1)$ to $e^1_2$ is the sum: $\theta = \arctan\!\frac{\sqrt b}{M_1} + \arctan\!\frac{\sqrt b}{M_2}.$ 
The matrix $O$ rotating $(0,1)$ to $e^1_2$ by this angle takes the stated form. 

\paragraph{The diagonal matrix $A_2 O A_1$}

Under the canonical coordinates described and with $O$ as defined before,
\begin{equation}\label{diagonal}
A_2 O A_1 = \begin{pmatrix} bM_2/M_1 & 0 \\ 0 & bM_1/M_2 \end{pmatrix}.
\end{equation}
In particular,
\[
\|A_2 O A_1\| = b\,\max\!\left(\frac{M_1}{M_2},\,\frac{M_2}{M_1}\right), \qquad |\det(A_2 O A_1)| = b^2.
\]





\paragraph{Factorization of $A_2 A_1$}

With $D := A_2 O A_1$ then,
\[
A_2 A_1 = D \cdot \bigl(A_1^{-1} O^{-1} A_1\bigr).
\]


With $A_1 = \mathrm{diag}(b/M_1, M_1)$ and $O$ rotation by angle $\theta$,
\[
A_1^{-1} O^{-1} A_1 = \begin{pmatrix} \cos\theta & -(M_1^2/b)\sin\theta \\[4pt] (b/M_1^2)\sin\theta & \cos\theta \end{pmatrix}.
\]


In the canonical coordinates,
\begin{equation}\label{A2A1}
    A_2 A_1 = \begin{pmatrix} (bM_2/M_1)\cos\theta & -M_1 M_2\,\sin\theta \\[4pt] (b^2/(M_1 M_2))\,\sin\theta & (bM_1/M_2)\cos\theta \end{pmatrix},
    \end{equation}

with $\theta= \arctan\!\frac{\sqrt b}{M_1} + \arctan\!\frac{\sqrt b}{M_2}.$

Observe that the largest entry, the off-diagonal $M_1 M_2 \sin\theta$, is bounded by $\sqrt b\,(M_1 + M_2)$ because the explicit value of $\sin\theta$ contains a compensating factor $1/(M_1 M_2)$: the smallness of the rotation angle $\theta$ exactly cancels the potentially large product $M_1 M_2$. The explicit estimates are done in what follows, using the Frobenius norm.

\begin{proof}[Proof of proposition \ref{prop:pinning}] First, we estimate the entries of $A_2.A_1$ given by equation \ref{A2A1}, and we bound each term.

\noindent\emph{Step 1: bound on each entry.}
Using $|\cos\theta| \leq 1$ and the explicit value of $\sin\theta$ ,
\[
|\sin\theta| = \frac{\sqrt b\,(M_1+M_2)}{\sqrt{(M_1^2+b)(M_2^2+b)}} \leq \frac{\sqrt b\,(M_1+M_2)}{M_1 M_2},
\]
since $(M_i^2 + b) \geq M_i^2$. Therefore:
\begin{align*}
\bigl|(A_2 A_1)_{11}\bigr| &\leq \frac{b M_2}{M_1}, \\[4pt]
\bigl|(A_2 A_1)_{12}\bigr| = M_1 M_2 |\sin\theta| &\leq \sqrt b\,(M_1+M_2), \\[4pt]
\bigl|(A_2 A_1)_{21}\bigr| = \frac{b^2}{M_1 M_2}|\sin\theta| &\leq \frac{b^{5/2}(M_1+M_2)}{(M_1 M_2)^2}, \\[4pt]
\bigl|(A_2 A_1)_{22}\bigr| &\leq \frac{b M_1}{M_2}.
\end{align*}

\medskip

\noindent\emph{Step 2: bound the operator norm by the Frobenius norm.} Using $\|A_2 A_1\| \leq \|A_2 A_1\|_F = \sqrt{\sum_{i,j} |(A_2 A_1)_{ij}|^2}$:
\[
\|A_2 A_1\|_F^2 \;\leq\; b(M_1+M_2)^2 \;+\; b^2\!\left(\frac{M_1^2}{M_2^2} + \frac{M_2^2}{M_1^2}\right) \;+\; \frac{b^5(M_1+M_2)^2}{(M_1 M_2)^4}.
\]

\medskip

\noindent\emph{Step 3: take the square root.} Using $\sqrt{x_1 + x_2 + x_3} \leq \sqrt{x_1} + \sqrt{x_2} + \sqrt{x_3}$ for $x_i \geq 0$,
\[
\|A_2 A_1\|_F \;\leq\; \sqrt b\,(M_1+M_2) \;+\; b\sqrt{\frac{M_1^2}{M_2^2} + \frac{M_2^2}{M_1^2}} \;+\; \frac{b^{5/2}(M_1+M_2)}{(M_1 M_2)^2}.
\]

\medskip

\noindent\emph{Step 4: simplify the middle term.} Using $\sqrt{a^2 + c^2} \leq a + c$ for $a, c \geq 0$,
\[
\sqrt{\frac{M_1^2}{M_2^2} + \frac{M_2^2}{M_1^2}} \;\leq\; \frac{M_1}{M_2} + \frac{M_2}{M_1}.
\]

Combining,
\[
\|A_2 A_1\| \;\leq\; \sqrt b\,(M_1+M_2) \;+\; b\!\left(\frac{M_1}{M_2} + \frac{M_2}{M_1}\right) \;+\; \frac{b^{5/2}(M_1+M_2)}{(M_1 M_2)^2}. 
\]

\noindent\emph{Step 5: bounding in terms of $b$ and $M_0$.}  A simple calculation shows that the right hand of the equation (\ref{boundA1A2})  on the rectangle $[1, M_0] \times [1/M_0, M_0]$ is realized at the corner $(M_0, M_0)$, so the maximum value is 
\[
f(M_0, M_0) = 2\sqrt b\, M_0 + 2b + \frac{2 b^{5/2}}{M_0^3}, 
\]
and from the fact that $M_0\geq \sqrt b,$ it holds that $\frac{2 b^{5/2}}{M_0^3}< 2. b$, and so 
\[
f(M_0, M_0) \leq 2\sqrt b\, M_0 + 4b. 
\]

\end{proof}

\vskip 5pt

\medskip

To understand the bound (\ref{boundA1A2}) observe that the term $\sqrt b\,(M_1+M_2)$  comes from the off-diagonal entry $M_1 M_2 \sin\theta$. The product $M_1 M_2$ is large, but $\sin\theta \approx \sqrt b\,(1/M_1 + 1/M_2)$ is small, and the cancellation gives $\sqrt b\,(M_1+M_2)$. This is the dominant term and reflects the imperfection of the alignment between the direction of maximal expansion of $A_1$ and minimal contraction of $A_2.$
 The  term $b(M_1/M_2 + M_2/M_1)$  comes from the diagonal entries, which are essentially the singular values of the perfect-alignment matrix $D$. This is the genuine size of the composition when the alignment is exact, and it dominates only when the matrices are very ``unbalanced'' relative to each other ($M_1 \gg M_2$ or vice versa).

 The condition $\sqrt b\,M_0 < 1$ ensures both terms remain manageable: it guarantees that the rotation angle is genuinely small (controlling the first term) and that the unbalance ratio $M_1/M_2$ is not too extreme (controlling the second) and this is done in the step 5 above.

 Now, before completing the proof of theorem \ref{dissipative-FFH}, we provide two remarks, the first one is about relaxing the pinning condition and the second one is about considering matrices that their determinant are not necessarily the same.

\begin{remark}\label{rmk:pinning} Theorem \ref{dissipative-FFH} assumes that the right-hand side of (\ref{boundbM0}) is smaller than one. Choose $\la_0<1$ and $\la_1$ also smaller than one but close to one and $\de>0$ small,  such that the right hand  (\ref{boundbM0}) is smaller than $\la_0$. In  inequality \ref{boundA1A2}, if one relaxes  the pinning condition to assume  that $1-\de < g_{A_1}(v)< 1+\de$ and $1-\de < g_{A_2}(G_{A_1}(v))< 1+\de$, and   if one  takes $M_1>\la_1$ and $M_2.M_1>\la_1$ it follows that the right-hand side of (\ref{boundA1A2}) is smaller than $\la_0.$ 
    
\end{remark}

\begin{remark}\label{rmk:det-normalization}
The proposition is stated for matrices with equal determinants $|\det
A_i|=b$; in the proof of theorem \ref{dissipative-FFH} it is applied to
derivatives along an orbit, whose determinants $b_i$ satisfy only $b_i\leq
b$. All the estimates in the proof are monotone non-decreasing in each
$b_i$ (in particular $\sin\theta=\frac{\sqrt{b_1}M_2+\sqrt{b_2}M_1}
{\sqrt{(M_1^2+b_1)(M_2^2+b_2)}}\leq\frac{\sqrt
b\,(M_1+M_2)}{M_1M_2}$), so the bounds (\ref{boundA1A2}) and
(\ref{boundbM0}) hold with $b=\max_{x\in\La}|\det(D_xf)|$.
\end{remark}

\begin{proof}[Proof of theorem B'] Let us consider $\la_0, \la_1$ as in the previous remark,  $\la=\max\{\la_0, \la_1\}$, and a finite orbit of a point $x$ such that there exists a vector $v$ satisfying  
\begin{equation}\label{flex-pinning}
    1-\de< g_{f^k(x)}(G^k(v))<1+\de,\,\,\mbox{for}\, k=0\dots m.
\end{equation}
    
 By induction, one can chose a subsequences of iterates  $\{x_{i_j}\}$ such that   either 
\begin{itemize}
    \item[--] $f(x_{i_j})=x_{i_{j+1}}$ and $||D_{x_{i_j}}f||<\la$ or

    \item[--] $f^2(x_{i_j})=x_{i_{j+1}}$ and $||D_{x_{i_j}}f^2||<\la,$
\end{itemize}  

Once $x_{i_j}$ has been chosen, we distinguish three cases according the norm of $D_{x_{i_j}}f$. Case (i), if $||D_{x_{i_j}}f||<\la_1$ then $x_{i_{j+1}}=f(x_{i_j})$. Case (ii),  if $||D_{x_{i_j}}f|| \geq \la_1$ and  $||D_{x_{i_j}}f||||D_{f(x_{i_j})}f||<\la_1<1$ then one takes $x_{i_{j+1}}=f^2(x_{i_j})$ and $||D_{x_{i_j}}f^2||<\la<1$. Case (iii)  if $||D_{x_{i_j}}f||||D_{f(x_{i_j})}f||\geq \la_1$ one can apply remark  \ref{rmk:pinning} and again it follows that $||D_{x_{i_j}}f^2||=||D_{x_{i_j}}f.D_{f(x_{i_j})}f||<\la_0<1.$

We now  apply lemma \ref{finite-basin}, it follows that if $n\geq m$ then  $x$ is in the basin of attraction of an attracting  periodic point that has a local basin containing a ball of radius $r$. By remark \ref{finite-attracting}, there are finitely many  attracting periodic points. So, if we considered $\La'$, the complement in $\La$ of the basins of attraction of these finite attracting periodic points, it holds that that in $\La'$ there are not finite trajectories of size larger than $m$ satisfying (\ref{flex-pinning}); therefore, taking $0<\de'$ smaller than $\de$, there is not trajectory of size larger than $m_0$ satisfying $(1-\de')^k< g^k(v)< (1+\de)^k$ for $k=0\dots n.$ In particular, no infinite forward orbit in $\Lambda'$ lies in the window, so $f|_{\Lambda'}$ is far from homotheties.

\end{proof}
\medskip

\begin{remark} Observe that the condition $\sqrt b\, M_0 < 1$ serves two purposes. In the context of proposition \ref{prop:pinning} it ensures that the rotation angle $\theta$ needed for alignment  is small enough for the optimization of $b$ and $M_0$. In the degenerate case that $\frac{b}{M_2}=M_2$, it implies that $A_2$ is $\sqrt b. O$ where $O$ and so $||A_2||=\sqrt b$  and so $||A_2.A_1||<1.$
    
\end{remark}

\subsection{Relaxing the  sufficiently dissipative condition}\label{subsec:relaxing}

The sufficient dissipation condition (\ref{cond-diss-FFH}) is used in this
section only through proposition \ref{prop:pinning}, whose threshold
$\sqrt b\,M_0<1$ is what forces the quantitative form of (\ref{cond-diss-FFH}).
We conjecture that this can be improved to only require dissipation.

\begin{conjecture}\label{conj:relaxing}
Theorem \ref{dissipative-FFH} holds assuming only that $f_{|\La}$ is
dissipative, that is, $b=\max_{x\in\La}|\det(D_xf)|<1$, without the
quantitative condition (\ref{cond-diss-FFH}).
\end{conjecture}

The mechanism behind proposition \ref{prop:pinning} is that a pinned pair of
matrices, after the optimal alignment, contracts at a rate controlled by $b$
rather than by the individual norms $M_i$. The threshold $\sqrt b\,M_0<1$
measures how much the imperfection of the alignment is allowed to relax over a
\emph{single} pair. The following non-rigorous computation suggests that, by
grouping the cocycle into longer blocks, the relaxation is getting better: the
 threshold should relax from $\sqrt b\,M_0<1$ towards to  $b<1$ as the block length grows.

\emph{The heuristic.} Consider $2^n$ matrices $A_1,\dots,A_{2^n}$ with
$|\det A_i|=b$, singular values $b/M_i\le M_i$, $M_i\in[1,M_0]$, and the
sequences of  pinning $g_{A_1}(v_1)=1$, $v_{i+1}=G_{A_i}(v_i)$,
$g_{A_{i+1}}(v_{i+1})=1$. Assume at each pairing step that the matching
rotation between adjacent singular frames is the identity (perfect
alignment), so that each pairwise product is the diagonal matrix
(\ref{diagonal}), with top singular value at most $b\,M_0$. Regrouping into
blocks of $2^k$ and iterating the pairing, the chain rule transmits the
pinning to the pair-matrices, the determinant exponent doubles at each
level while the singular-value ratios inherit a comparable bound, so a block
of $2^k$ matrices has top singular value at most $b^{2^{k-1}}M_0$ under
perfect alignment. The criterion $b^{2^{k-1}}M_0<1$, that is
$b<M_0^{-1/2^{k-1}}$, relaxes towards $b<1$ as $k\to\infty$.

\emph{What a proof would require.} Making this rigorous means controlling, at
each level of pairing, the deviation of the matching rotations from the
identity (this is the analogue of the perturbation analysis done  for a
single pair in the proof of proposition \ref{prop:pinning}) with the added
difficulty that the singular-value ratios entering level $k+1$ depend on the
alignment errors accumulated through level $k$. The errors must be shown to
decay fast enough across levels that their product remains controlled; we do
not know whether the geometric gain $b^{2^{k-1}}$ in the determinant always
dominates this accumulation. We leave this analysis for future work.

\section{Properties of the critical set}\label{sec:properties}  
\label{s.properties}



\begin{proof}[ Proof of proposition \ref{properties}:]   


To prove the first item, observe that if $f^n(c)=c$ and $v$ is the critical direction, then $|g^{k.n}(v)|\geq 1$ for any integer $k$. On the other hand, since $G^n:\TT^1\to \TT^1$ is a Mobius transformation, then $G^n$ is either hyperbolic, elliptic or parabolic. If it is hyperbolic, then either $v$ is a fixed point of $G^n$ (one repeller and one attractor, at the repeller the projective derivative grows exponentially, and in the other decrease exponentially to zero) or its forward orbit converges to the attracting  fixed point of $G^n$ and so the projective derivative  decreases exponentially to zero; in both cases, it can not hold that $|g^{k.n}(v)|\geq 1$ for every integer $k$. Therefore, $G^n$  can only be either parabolic or elliptic and so the first item of the proposition is proved.

The second item follows from the fact that if all iterates $(f^n(c), G^n(v))$ critical points and critical directions, then $g(G^n(v))=1$ for any integer $n$ and so $g^n(v)=1$ for all $n$ contradicting that $f$ is far from homotheties.

The last item follows immediately by continuity: if $(c_n)$ is a sequence of critical points of a sequence of diffeomorphisms $(g_n)$ converging to a diffeomorphism $f$, it follows that any accumulation point of the sequence $(c_n)$ is a critical point of $f.$

\end{proof}

One can get a quantitative version of  the second item of proposition \ref{properties}:







\begin{proof}[Proof of  proposition \ref{pr:compatibility}] Observe that given $\eps'>0$ there exists a positive integer $N_1$ such that if $n$ is large enough such that $f^{-j}(x)\notin B_\de({\rm Crit}(f,\La))$ for any $0\leq j\leq n$ then there is $n^*\leq N_1$ such that $f^{-n^*}(x)\in B_{\eps'}(\La_\de).$ Moreover, there is $m^*= m^*(n)$ such that $m^*(n)\to +\infty$ as $n\to \infty$ and for any $n^*\leq  j \leq m^*$ it holds that $f^{-j}(x)\in B_{\eps'}(\La_\de).$

Now, provided $\eps$, there is  $\eps'$ small $0<\ga<1$ and  $N_2$, such  that if  $f^{-j}(y)\in B_{\eps'}(\La_\de)$ for all $0\leq j\leq n$ with $n\geq N_2$ then for any $w\in T_yM$ such that $slope(w, F)>\eps$ follows $g^{-n}(w)<\ga^n.$

Now, we take $N_0$ such that if $n\geq N_0$ then it holds that for any critical $c$ if $f^{-j}(c)\notin B_\de({\rm Crit}(f,\La)), 0\leq j\leq n$ then 
\begin{itemize}
    \item[--] there is a positive integer $n^*\leq N_1$ such that  $f^{-n^*}(c)\in B_{\eps'}(\La_\de),$
    \item[--] $f^{-j}(c)\in B_{\eps'}(\La_\de)$ for any $n^*\leq j\leq m^*,$
    \item[--] $m^*-n^*\geq N_2.$
\end{itemize}
Since $n^*$ is uniformly bounded for any critical point, $N_2$ can be taken large such that  $g^{-n^*}(v_c)\ga^{N_2}<1,$ therefore, if $slope(G^{-n^*}(v_c), F)>\eps,$ it follows that  $g^{m^*}(v_c)= g^{m^*+n^*}(G^{-n^*}(v_c)).g^{-n^*}(v_c)<1,$ contradicting that $v_c$ is the critical direction.

The proof for forward iterations is similar.

\end{proof}

\vskip 5pt

Corollary \ref{cor:compatibility} follows from the proof of the previous proposition. Suppose  the conclusion of the corollary. Take $\de_2$ sufficiently small,   one gets a critical point whose pre-iterate   (the first backward visit to a neighborhood of $\La_{\de_1}$) carries a critical direction not close to the center-unstable subbundle. This contradicts Proposition \ref{pr:compatibility}.

\vskip 5pt

 \begin{proof}[Proof of proposition \ref{exponential recurrence}]

 \noindent 1. We consider $\varepsilon>0$ and a
 critical point $c\in \La$ associated with a direction $u_0$, and with recurrence $(1-\varepsilon)^n$: for some large times $n$,
it satisfies $f^n(c)\in B(c,(1-\varepsilon)^n)$.
\medskip

\noindent
Since $\La$ is dissipative, $|\det Df^j |_\La|<b^j$ for some $b\in (0,1)$ and $j\geq 1$.
\medskip

\noindent
Let $\Delta$ be an upper bound on derivatives of $G$ and $\log\|Df\|$ and $\lambda$ a lower one on $Df$.
\bigskip

\noindent
2. We choose successively $\delta,\delta',\alpha,\widehat \delta,\gamma,\beta,\rho,\eta>0$ small such that
$\widehat \delta<\delta'<\delta<\varepsilon/2$ and
\begin{equation}\label{e.choice-alpha}
(1-\varepsilon)<(1-\delta')\lambda^{\alpha},\qquad b^\alpha(1-\rho)^{-\alpha}<1-\widehat \delta,\qquad \Delta^{\gamma\alpha}(1-\delta')(1+\eta)^{\alpha}<(1-\widehat \delta),
\end{equation}
\begin{equation}\label{e.choice-beta}
\Delta^{\beta} (1-\widehat \delta)<1, \qquad b^{\beta/4}(1+\eta)<(1-\eta)^2,
\end{equation}
\begin{equation}\label{e.choice-eta}
1-\rho>b^{\gamma/3}(1+\eta)^{(1-\gamma)},\qquad (1+\eta)^2(1-\rho)^{\min(\alpha,\beta)}<1-\eta.
\end{equation}

\noindent
3. We will consider a large return time $n$ and introduce the ball $B:=B(c,(1-\delta)^n)$.\newline
Since $\underset{k\to +\infty}\limsup \tfrac 1 k \log \|Df^k|_\La\|\leq 0$,
there is $N\geq 1$ such that for $x$ in a neighborhood of $\La$,
\begin{equation}\label{e.bounds}
\|Df^{N}(x)\|\leq (1+\eta)^N.
\end{equation}
In particular, there exists $k_0$ such that
$$\forall k_0\leq k\leq 2n,\qquad f^k(B(c,(1+\delta)^n))\subset B(f^k(c),(1+\delta')^n).$$

\noindent
4. We define $m:=[\alpha n]$ and the ball $B':=B(f^{n-m}(c),(1-\delta')^n)$ so that $f^{n-m}(B)\subset B'$.\hspace{-1cm}\mbox{}
\smallskip

\noindent
From the choice of $\alpha$ in~\eqref{e.choice-alpha}, the bounds in~\eqref{e.bounds}, the recurrence in \S1, we have $c\in f^{m}(B')$.
\bigskip

\noindent
5. Let us assume that $\|Df^m\|<(1-\rho)^m$ on $B'$. The choice of $m$ and estimates in~\eqref{e.choice-eta}, \eqref{e.bounds} give
$\|Df^n|_B\|<(1+\eta)^n(1-\rho)^{\alpha n}<1.$
Since $0<\delta<\varepsilon$ and $d(f^n(c),c)<(1-\varepsilon)^n$,
the map $f^n$ is a contraction from $B$ into itself, $B$ is included in the basin of a sink and the proposition holds in this case.
\smallskip

\noindent
One can thus assume: \emph{There exist $y_0\in B'$ and a unit vector $v_0$ with} 
\begin{equation}\label{e.g}
\|Df^m.v_0\|\geq (1-\rho)^m.
\end{equation}

\noindent
6. {\bf Claim 1.}
\emph{For any $z\in f^{m}(B')$, the linear map $Df^{-m}(z)$ is not a homothety.}
\medskip

\noindent
Let $e_z$ be the less expanded direction for $Df^{-m}(z)$  (it maximizes $g^{-m}$).
Note that from \S 5, the direction $e_c$ is defined at $c$.
\medskip

\noindent
\quad {\bf Claim 2.}
\emph{The oscillations of $e_z$ on $f^{m}(B')$ are bounded by $4(1-\widehat \delta)^n$.}
\bigskip

\begin{proof}[{\rm 7.} Proof of claim 1.]
We consider $y_0$ with a direction $v_0$ as in \S5. At any $y'\in B'$ let
$v'_0$ be the direction parallel to $v_0$.
The iterates of $v_0,v'_0$ under $G^k$ are denoted $v_k,v'_k$.
\medskip

\noindent
The choice of $y_0,v_0$ and~\eqref{e.g} give $|g^m(v_0)|/|\det Df^m(y_0)|\leq (1-\rho)^{-2m}.$
\smallskip

\noindent
We decompose the orbits of $y_0, v_0$ up to time $m$
into blocks of size $N$. More precisely, to control $||Df^m||$ we split the orbit into segments of length $N$, on each of which the upper bound (\ref{e.bounds}) gives a uniform estimate. The fraction $\chi$ of 'bad' blocks (where the bound is violated) is controlled by a chain-rule estimate combined with the fact that  $|\det Df^m| < b^m$. By~\eqref{e.bounds} for any $x$ in a neighborhood of $\La$ and any direction $v$,
one has $|g^N(v)|/|\det Df^N(x)|\geq (1+\eta)^{-2N}.$
\smallskip

\noindent
Hence the density $1-\chi$ of blocks satisfying $|g^N(v_k)|/|\det Df^N(y_k)|\leq b^{-2N/3}$ satisfies
$$ (1-\rho)^{-2}\geq b^{-\tfrac 2 3\chi}(1+\eta)^{-2(1-\chi)}.$$
From~\eqref{e.choice-beta}, $\chi<\gamma$.
Moreover, for these good blocks, one gets a contraction $|g^N(v_k)|\leq b^{N/3}$.\hspace{-1cm}\mbox{}
\medskip

\noindent
Let us inductively compare $v_k$ and $v'_k$.
Note that $y_k,y'_k$ are at distance smaller than $(1-\delta')^n(1+\eta)^k$, which bounds the difference between
$Df^N(y_k)$ and $Df^N(y'_k)$.
For all the good blocks the angle between $v_k$ and $v'_k$ decreases after iteration by $Df^N(y_k)$.
For the other blocks, the angle may increase by a factor bounded by $\Delta^N$.
\medskip

\noindent
The bound $\gamma$ on the density and~\eqref{e.choice-alpha} imply that $\angle(v_k,v'_k)$ for $0\leq k\leq m$ is bounded by
$$\Delta^{\gamma m}(1-\delta')^n(1+\eta)^m<(1-\widehat \delta)^n.$$
Since the orbits of $v_0,v'_0$ remain $(1-\widehat \delta)^n$-close during $m$ iterates,
\begin{equation}\label{e.expansion}
\|Df^m.v_0'\|\geq \exp(m.\Delta.(1-\widehat \delta)^n)\|Df^m.v_0\|\geq \tfrac 1 2 (1-\rho)^m.
\end{equation}
Hence $Df^m(y')$ is not an homothety. Since $y'$ is arbitrary in $B'$, one gets claim 1.
\end{proof}
\medskip

\begin{proof}[{\rm 8.} Proof of claim 2.]
As in the previous proof, let $z=f^m(y')$ and $\theta:=\angle(v'_m,e_{z})$. Then
$$\|Df^{-m}(v'_m)\|\simeq \theta \|Df^{-m}(e^\perp_{z})\|\geq \theta |\det Df^{-m}(z)|/\|Df^{-m}(e_{z})\|.$$
$$\text{With~\eqref{e.expansion} and~\eqref{e.choice-alpha},}\qquad \theta\leq |\det Df^{m}(y')|\cdot \|Df^{m}(v'_0)\|^{-2}\leq 4b^m (1-\rho)^{-2m}
<(1-\widehat \delta)^n.$$
We have shown in the proof of claim 1 that $\angle(v_m,v'_m)<(1-\widehat \delta)^n$.
One thus concludes that the fluctuations of $e_z$ are bounded by $4 (1-\widehat \delta)^n$.
\end{proof}
\bigskip

\noindent
9. {\bf Claim 3.} \emph{$\angle(u_0,e_c)<(1-\widehat \delta)^n$.}
\begin{proof}
From~\eqref{e.expansion} in the proof of claim 1, $\|Df^{-m}.e_c\|\leq 2(1-\rho)^{-m}$.
Equation~\eqref{e.g} gives:
$$|g^{-m}(u_0)|\leq |\det Df^{-m}(c)|/(\theta\|Df^{-m}.e_c^\perp\|)^2=\|Df^{-m}.e_c\|^2|\det Df^{m}(f^{-m}(c))|/\theta^2 .$$
Since $u_0$ is the critical direction at $c$, one has $|g^{-m}(u_0)|\geq 1$, and with~\eqref{e.choice-alpha},
\begin{equation*}\theta^2\leq 4(1-\rho)^{-2m}|\det Df^{m}(f^{-m}(c))|\leq 4(1-\rho)^{-2m}b^m\leq (1-\widehat \delta)^{2n}.\tag*{\qedhere}
\end{equation*}
\end{proof}
\bigskip

\noindent
10. {\bf Claim 4.} \emph{For any $x\in B$ with a unit vector $w$
satisfying $\|Df^m(G^{n-m}(w))\|\geq (1-\rho)^m$
$$\angle(G^n(w),e_{f^n(x)})<(1-\widehat \delta)^n .$$}
The proof is the same as for claim 2.
\bigskip

\noindent
11. For any unit vector $w$ at points $x\in B$, we will prove $\|Df^{2n}.w\|<(1-\eta)^{2n}$.
\bigskip

\noindent
12. If $\|Df^m.(G^{n-m}(w))\|\leq (1-\rho)^m$, then $\|Df^{2n}.w\|\leq (1-\rho)^m(1+\eta)^{2n}\leq (1-\eta)^{2n}$ from~\eqref{e.choice-eta}.
We are thus reduced to the case where $\|Df^n.w\|> (1-\rho)^n$.
\bigskip

\noindent
13. From \S3, and claims 2, 3, 4, $d(f^n(x),c)$ and $\angle(G^n(w),u_0)$
are smaller than $6(1-\widehat \delta)^n$.
Let $\ell:=[\beta n]$.
For any $1\leq k\leq \ell$, both $d(f^{n+k}(x),f^k(c))$ and
$\angle(G^{n+k}(w),G^k(u_0))$ are bounded by $6\Delta^k.(1-\widehat \delta)^n$.
The definition of $\ell$ and~\eqref{e.choice-beta} give:
$$\|Df^\ell(G^n(w))\|\leq \|Df^\ell(u_0)\|[\exp(6\ell\Delta^\ell.(1-\widehat \delta)^n)]\leq 
2\|Df^\ell(u_0)\|.$$
Note that by~\eqref{e.g} and since $|g^\ell(u_0)|\geq 1$,
we have $\|Df^\ell(u_0)\|\leq |\det Df^\ell(c)|^{1/2}$.
One deduces $\|Df^\ell.G^n(w)\|\leq 2 |\det Df^\ell(c)|^{1/2}\leq b^{\ell/2}$, then
$\|Df^{2n}.w\|\leq b^{\ell/2} (1+\rho)^{2n}$.
Our choice of $\ell$ and \eqref{e.choice-beta} then imply $\|Df^{2n}.w\|\leq (1-\eta)^{2n}$
as required.
\bigskip

\noindent
14. We have proved the inequality $\|Df^{2n}\|<(1-\eta)^{2n}$ on the ball $B=B(c,(1-\delta)^n)$.
Since $d(c,f^n(c))<(1-\varepsilon)^n$, we have $d(f^n(c),f^{2n}(c))<(1-\varepsilon)^n(1+\eta)^n$
and $d(c,f^{2n}(c))<(1-\varepsilon)^n[1+(1+\eta)^n]$. Since $\delta<\varepsilon$,
the ball $B$ is mapped into itself by $f^{2n}$, is contracted, and $c$ belongs to the basin of a sink.
The proposition is proved.

\end{proof}

 





\section{Critical points for  mildly dissipative diffeomorphisms of the disk. Proof of theorem \ref{disk2}}\label{sec:disk}

In what follows, we assume that $\DD$ is an open disk that is properly invariant under $f$, that is, $\overline{f(\DD)}\subset \DD$ and $f: \DD\to f(\DD)$ is a diffeomorphism and $f_{|\DD}$ is mildly dissipative.

In the dissipative context, given a fixed (or periodic) point $p$  there is  a well defined local center-unstable manifold denoted as $W^{cu}_\e(p)$, and   $W^{cu}_\e(p)\setminus\{p\}$ has two connected components.  The unstable set $W^u(p)$ of $p$ is the set of points whose backward orbit
converges to $p$. The local unstable set is defined as the set of points whose backward orbit
converges to $p$ and remains in a small neighborhood of $p;$  when this set is not reduced to $p$, it is a $C^1-$curve which
contains $p$ (not necessarily in the interior). In this case, $p$ is called a saddle point. At least one of the center-unstable branches at $p$ is contained in the unstable manifold of $
p$. The other branch, if not contained in the unstable manifold, is either a normally hyperbolic arc or contained in the basin of attraction of $p$; in the latter case, $p$ is called semi-attracting.  Similarly, as  in the case of the center-unstable manifold, each connected component $\Ga^u(p)$ of $W^u(p)\setminus\{p\}$ is called an unstable branch of  $p$. 

\vskip 5pt

The proof of theorem \ref{disk2} is decomposed into a few lemmas and claims but before getting into details, we provide an outline of the proof that may help the readers as a road map.

\noindent \emph{Outline of the proof of theorem \ref{disk2}:} Let $\La^u$ be the accumulation set of the
unstable branch $\Ga^u(p)$. The argument has three stages. \emph{First}
(the far-from-homotheties reduction),  by theorem \ref{dissipative-FFH}, after removing the basins of
all attracting periodic orbits $(q_i)$ of $\La$, every power of $f$ is far
from homotheties on what remains of $\La^u$. Theorem \ref{p.existence}
applied to this sink-free part gives the first dichotomy: either it already
contains a critical point, and the theorem is proved, or it admits a
dominated splitting that can be extended to a neighborhood $U$
of the sink-free part. In the latter case the splitting implies the
far-from-homotheties property of all powers, in the neighborhood $U$; and the full orbit of every sink \emph{except
finitely many} $\{q_1,\dots,q_k\}$ already lies in $U$. It remains only to
recover the property on the local basins of these finitely many exceptional
sinks, which is the role of claim \ref{cl:att-fp}: an exceptional $q_i$ with
eigenvalues of equal modulus is itself a critical point (and we are done),
while one with eigenvalues of distinct moduli is far from homotheties, for
every power, on a ball of definite radius around its orbit. Adjoining these
finitely many balls (and finitely many of their preimages) to $U$ covers
$\La^u$; hence every power of $f$ is far from homotheties on all of $\La^u$,
and a final application of theorem \ref{p.existence} yields the dichotomy of
the theorem: \emph{either} ${\rm Crit}(f,\La^u)\neq\emptyset$, \emph{or}
$\La^u$ admits a dominated splitting $E^s\oplus F$, which we assume
henceforth. \emph{Second} (the geometry of $\La^u$),
assuming there is no critical point we analyse how $\La^u$ sits with respect
to the center-unstable foliation: lemma \ref{pre-cu-omega} shows that at
least one center-unstable branch of each point lies in $\La^u$; lemma
\ref{cu-omega} upgrades this along $\omega$-limits; and lemma
\ref{whole-centerunstable} concludes that \emph{both} branches lie in
$\La^u$ except at finitely many boundary periodic points. \emph{Third} (the
trapping region), after a perturbation that turns those finitely many
boundary points into (semi-)attracting fixed points, the local stable arcs
over $\La^u$ assemble into an open, connected, simply connected set $B$ with
$\overline{f(B)}\subset B$ and $\bigcap_{n>0}f^n(B)=\La^u$. The
decomposition of remark \ref{decomposition} applies to $f_{|B}$, and a
Poincar\'e--Bendixson argument rules out homoclinic classes and normally
hyperbolic closed arcs unless a tangency and hence a critical point in
$\Ga^u(p)$ is produced; what survives is exactly the dichotomy of the
theorem. The critical points and normally hyperbolic arcs found for the
perturbed map are valid for $f$, since the two coincide on $\La^u$.

\medskip

Now, let us proceed to formalize that outline. Let $\La^u:= \overline{\cup_{n>0} f^n(\gamma^u)}$  be the accumulation set of the branch $\Ga^u,$ where $\gamma^u$ is a fundamental domain of $\Gamma^u$.   From the fact that $f$ is dissipative it follows that $\Lambda^u$ is a cellular set (a compact connected set of the plane  whose  complement is connected). The sufficiently dissipation condition implies that after removing all the attracting periodic points $(q_i)$ in $\La$ it holds that  for any positive integer $N$, $f^N_{|\La^u\cap [\cup_{i} W^s(q_i)]^c}$ 
is far from homotheties and so one can apply theorem \ref{p.existence}: if there is a critical point in $\La^u\cap [\cup_{i} W^s(q_i)]^c$ then the proof of theorem \ref{disk2} is over, if not, then $\La^u\cap [\cup_{i} W^s(q_i)]^c$ has a dominated splitting.    Observe that this splitting can be extended to a neighborhood $U$ of $\La^u\cap [\cup_{i} W^s(q_i)]^c$, in particular, for any positive integer $N$, $f^N$ is far from homotheties in $U.$  Observe also that except for   a finite number of attracting periodic points $\{q_1, \dots, q_k\}$ in $\La^u$, the orbits of the others are fully contained in $U$. Since the result is about $\La^u$ one has to analyze now what happens in the complement  of $U$ that intersects the basins of attraction. For that, we use the following claim (the proof is straightforward from the definitions and left to the reader, see remark \ref{rmk:why-powers}).

\begin{claim}\label{cl:att-fp} Let $q$ be an attracting fixed point and let $\la$ and $\si$ be the eigenvalues of $D_qf.$
\begin{itemize}
    \item[--] if $|\la|<|\si|$ then all the powers of $f$ are far from homotheties in the local basin of attraction of $q$;

     \item[--] if $|\la|=|\si|$ then  $q$ is a critical point.
\end{itemize}
    
\end{claim}

If one of the attracting periodic points $\{q_1, \dots, q_k\}$ (taking an iterate, we may  assume that they  are   fixed) has  eigenvalues of equal moduli, then $\La^u$  has a critical point. So, we have to consider the case that all the attracting  periodic points $\{q_1, \dots, q_k\}$  have eigenvalues of distinct moduli. In that case, for $r>0$ small enough it follows that for any positive integer $N$ we have $f^N$ restricted to $B=\cup_{i=1}^k B_{r}(O(q_i))$   is far from homotheties.  Fixing a positive integer $m$, we also see that $f^N$ restricted to $B_m=\cup_{i=1}^mf^{-i}(B)$ is also far from homotheties for any positive integer $N$.  Taking $m$ large, $B_m\cup U$ covers $\La^u$, therefore $f^N_{|\La^u}$ is far from homotheties for any positive integer $N$ and so one can apply theorem \ref{p.existence}: either there are critical points in $\La^u$ or $\La^u$ has a dominated splitting. So, either the proof of theorem \ref{disk2} is complete or we continue assuming that $\La^u$ has a dominated splitting.

\medskip

The following claim was proved in \cite{CPT} (see proposition 2.7).

\begin{claim}\label{cpt-lemma} Let $f$ be mildly dissipative and let $p,q$ be  two saddle points and $\Ga^u(p)$ and $\Ga^u(q)$ two of  their branches. If the accumulation set of $\Ga^u(p)$ intersects $\Ga^u(q)$ then it follows that the accumulation set of $\Ga^u(q)$ is contained in the accumulation set of  $\Ga^u(p).$ 
    
\end{claim}

From the previous claim it follows that if $q$ is a periodic point such that one branch  of the unstable manifold of $q$ intersects $\La^u$, then the whole branch is in $\La^u.$

If $\Lambda^u$ is a single point $q$, then $\La^u$  is a fixed point (if $\La^u$ is  finite, then it is a periodic point). The local center-unstable branch of $q$ that is accumulated by $\Ga^u$ can neither be contained in a normally hyperbolic arc nor be contained in the unstable manifold of $q$  otherwise, by claim \ref{cpt-lemma}, the accumulation set of $\Gamma^u$ should contain one of the unstable branches of the saddle $q$ or a normally hyperbolic invariant arc. Therefore, the fixed point is a (semi)attracting fixed point and its basin of attraction contains $\Gamma^u$. 

To conclude let us assume that $\La^u$ is not a single point and does not contain critical points. Therefore,  $f_{|\La^u}$ has dominated splitting and from the fact that $f$ is dissipative then it holds that  the splitting exhibits a contractive direction ($E^s\oplus F$). In particular, for any point in $\La^u$ there is a local stable manifold and two  center-unstable branches. The next lemma is the key lemma to prove theorem \ref{disk2}.

\begin{lemma}\label{whole-centerunstable} For any $x\in \La^u$ it holds that  there exists $\e(x)>0$ such that

\begin{itemize}
    \item[--]either  $W^{cu}_{\e(x)}(x)\subset \La^u,$ meaning that both (sufficiently small) branches of the center-unstable of $x$ are contained in $\La^u$, 

    \item[--] or only  one of the center-unstable branches  is in $\La^u$ and in that case $x$ is a periodic point, that we call boundary periodic point. 
    
\end{itemize}

    Moreover, there are only a finite number of boundary periodic points.
    
\end{lemma}
The proof of the previous lemma goes through two lemmas; the first one  states that at least one center-unstable branch is in $\La^u$ (lemma \ref{pre-cu-omega} whose proof needs claim \ref{attracting}), the second  states a similar statement to lemma \ref{whole-centerunstable} but for points in the omega limit of another point in $\La^u$ (lemma \ref{cu-omega},  its proof also needs claim \ref{attracting}).

\begin{claim}\label{attracting} Let $z$ be a (semi)attracting periodic point and let $x$ be a point in $\La^u$ and in the local basin of attraction of $z.$ Then there is a normally hyperbolic arc contained in $\La^u$ such that one of the extreme point is $z.$ 
    
\end{claim}

\begin{proof} Let us assume that $z$ is fixed. Let $l$ be an arc contained in the unstable branch of $p$ (the fixed point from the hypothesis of theorem \ref{disk2}) with one endpoint is close to $x$ and the other is in the boundary of the local basin of attraction of $z$; that point exists otherwise $\La^u$ would be reduced to $z$. The forward iterates of $l$ is  a sequence of arcs $(l_n)$ that are in the unstable manifold of $p$ such that one of their extreme points is in the boundary of the local basin of $z$ and the other converges to $z$. Let $\hat l$ be a continuous arc that is limit of the arcs $(l_n)$ and observe that the interior of $\hat l$ is contained in the local basin of $z$, one  of the boundary points of $\hat l$ is $z$ and the other, named $y$, is in the boundary of the local basin of $z$.  Since $y\in \La^u$ let us consider the local stable manifold of $y$; let  $y'$ be a point sufficiently close to $y$ in $l$,   and observe that the local stable manifold of the  forward iterates of $y'$ converge to  $z$ and since its stable manifold separated the disk, it follows that the stable manifold of an iterate of $y'$ also separates the disk and so also the stable manifold of $y'$; given that $y'$ can be taken arbitrary close to $y$ it follows that also the stable of $y$ separates the disk. Let  $\ga_n$ the connected  part of the stable manifold of $f^{-n}(y)$ intersected with the disk  that contains $y$ and observe that $\ga_n$ separates the disk; observe that $\ga_1$ is in the connected component of $\DD\setminus\{\ga_0\}$ that does not contain $z$ and similarly $\ga_{n+1}$ is in the connected component of $\DD\setminus\{\ga_n\}$ that does not contain  $\ga_{n-1}$. In that way, we get a sequence $(\ga_n)$ of stable leaves ordered by the  inclusion  of the connected components and so there is a limit  $\hat \ga$ of the sequence of stable arcs $(\ga_n)$; since $f(\ga_n)\subset \ga_{n-1}$ it follows that $\hat \ga$ is a continuous invariant arc and so it has a fixed point $z'$ that is also the limit of $f^{-n}(y)$. Observe that the connected component $\hat \DD$ of $\DD$ bounded by $W^s_\DD(z)$ and $\hat \ga$ is a forward invariant disk foliated by local stable manifolds of points in   $\tilde l= \cup_{n>0}f^{-n}(\hat l).$  Since $\hat l$ is forward invariant, then $\tilde  l$  is also forward invariant so it is a normally hyperbolic arc.

\end{proof}

\begin{lemma}\label{pre-cu-omega} For any $x\in \La^u$ there is $\e(x)>0$ such that one of the branches of $W^{cu}_{\e(x)}(x)$ is contained in $\La^u.$
    
\end{lemma}

\begin{proof}
If $x$ is a (semi)attracting periodic point, the lemma follows from claim \ref{attracting}. So, let us assume that  $x$ is not a (semi)attracting periodic point.
For any point $x\in \La^u$ and $\e$ small, let $\La^{u}_\e(x)$ be the connected component of $\La^u\cap B_\e(x)$ that contains $x$ and observe that  is a continuous set that contains $x$. If $\La^u_\e(x)$ is not in one of the branches of the  local center-unstable of $x$, by lemma \ref{centerunstable}, it holds that  backward iterates of $\La^u_\e(x)$ converge to the whole stable set of a point in $\La^u$, but that is not possible since by the assumption of mild dissipation, the stable manifold cannot be contained in the disk. 
\end{proof}

\begin{lemma}\label{cu-omega} Let $x\in \La^u$; for any $z\in \omega(x)$ it follows

\begin{itemize}
    \item[--] either  there exists $\e(z)>0$ such that $W^{cu}_{\e(z)}(z)\subset \La^u$, meaning that both (sufficiently small) branches of the center-unstable of $z$ are contained in $\La^u$;

    \item[--] or $z$ is a periodic point in $\La^u$ such that $x\in W^s(z)$ and at least one center-unstable branch of $z$ is contained in $\La^u.$
    




\end{itemize}

    Moreover, there are only a finite number of periodic points that satisfy that only one center-unstable branch is in $\La^u$.
\end{lemma}



    

\begin{proof}

By  lemma \ref{pre-cu-omega}, given $x$ there is one   branch $J$ of $W^{cu}_\e(x)$  contained in $\La^u$. Either $length(f^n(J))\to 0$ or there exists a sequences of integer $(n_k)$ such that $length(f^{n_k}(J))$ is uniformly bounded by below. 

In the first case, by lemma \ref{centerunstable2} it follows that $(f^n(J))$ converges to a (semi)attracting periodic point $z$. If the point is attracting, then $f^n(J)$ is in the interior of the basin of attraction of $z$; if $z$ is semi-attracting, then  the interior of  $f^n(J)$ is in the interior of the local  basin of attraction of $z$ and $x \in W^s(z)$ (meaning that the iterates of $x$ in in the boundary of the basin of attraction of $z$). In both cases, one can apply claim \ref{attracting} and so one center-unstable branch of $z$ is contained in $\La^u.$



In the second case, one can assume also that there is $z$  such that $f^{n_k}(x)\to z.$ Since $length(f^{n_k}(J))$ are uniformly bounded by below it follows that here is an arc $I$ containing $z$ and contained  in one center-unstable branch of $z$ that  is also  contained in the limit of $f^{n_k}(J)$ and so $I\subset \La^u.$ 

If there are two integers $n_k< n_{k'}$ such that  $f^{n_{k'}}(x)\in W^s_\de(f^{n_k}(x))$ it follows the there is a periodic point $z'$ in the local stable manifold of $f^{n_k}(x)$ and so $z=z'$  and in that case    $x$ is in the stable manifold of it and by lemma \ref{pre-cu-omega} one of its center-unstable branches is in $\La^u$. In the other case, let $I'$ be an arc containing $z$ in the other center-unstable branch of $z$ disjoints with $I.$ Either 

\begin{itemize}
    \item[--] i) there is an  iterate  $f^{n_k}(x)$ close to $z$  in the interior of $W^s_\de(I)$,

    \item[--] ii) or there are infinitely many iterates  $f^{n_k}(x)$ accumulating on $z$ in the interior of $W^s_\de(I').$
\end{itemize}

  In the first case, let $\hat J$ be an arc in the center-unstable of $f^{n_{k_0}}(x)$ that one extreme point is given by that point and the other by the intersection of the local stable of $z$ and the center-unstable of $f^{n_{k_0}}(x).$ It follows immediately that $\hat J$ is disjoint with $f^{n_{k_0}}(J)$ and observe that since $\hat J$ is in the local basin of $I$ then its forward iterates remain in a neighborhood of $\La^u$ and so  we can apply again lemma \ref{centerunstable2}; if the length of the iterates of $\hat J$ go to infinity, then the accumulation of the  forward iterates of $f^{n_k}(\hat J)$ are in $\La^u$ and  contain the other center-unstable branch $I'$ of $z$ disjoint with $I$ and concluding that the whole center-unstable of $z$ is in $\La^u.$ If the lengths of the iterates of of $\hat J$ does not go to infinite,  $x$ is in the basin of a normally hyperbolic arc and $z$ is a periodic point such that one of its center-unstable branch is contained in $\La^u.$  
  
  In the second case, i.e. that $f^{n_k}(x)$ is in the local basin of $\hat I$, then we can take $n_{k'}$ such that $f^{n_{k'}}(x)$ is closer to $z$ and again we can consider an arc $\hat J$ is the center-unstable branch of $f^{n_{k'}}(x)$ that one extreme point is given by that point and the other by the intersection of the local stable of $f^{n_k}(x)$ and the center-unstable of $f^{n_{k'}}(x).$ It follows immediately that $\hat J$ is disjoint with $f^{n_{k'}}(J)$ and observe that since $\hat J$ is in the local stable manifold of  $f^{n_k}(J)$, that is contained in $\La^u,$ then its forward iterates remain in a neighborhood of $\La^u$ and so  we can apply again lemma \ref{centerunstable2} and argue as before.

To prove the last claim of the lemma (about the finiteness of periodic points satisfying that  only one center-unstable branch is contained in $\La^u$),  first we prove that their periods are uniformly bounded. That is immediate for  periodic points that are either   (semi)attracting or contained in a normally hyperbolic arc. For the remaining periodic points, by proposition \ref{PS2-1}   any  pair of periodic points with large period and close enough are homoclinically related, but this implies that given three periodic points sufficiently close, it follows that the local stable manifold of  one  of them, named $q$, intersects the unstable branch of another that is contained in $\La^u$ and so by the classical inclination lemma it follows that the unstable of $q$ is contained in $\La^u$, a contradiction.  
So, if there are infinitely many periodic points, given that their periods are uniformly bounded, one can regroup them into a finite number of normally periodic arcs that are contained in $\La^u.$
\end{proof}


The next remark is a simple observation  that follows immediately from the dynamics of fixed points in the dissipative context. 

\begin{remark}\label{att-NH} Let $z$ be a periodic point  and let us choose one of the center-unstable branches. Either that branch is

\begin{itemize}
    \item[--]in the basin of attraction of $z$ or,

    \item[--] a normally hyperbolic arc or, 

    \item[--] contained in the unstable manifold of $z.$
\end{itemize} 
Moreover, given a sequence of points accumulating on $z$ and inside the connected component of $B_\e(z)\setminus W^s_\e(z)$ that intersects the chosen center-unstable branch, then either

\begin{itemize}
    \item[--]the forward iterates of that sequence converge to $z$  or to points in a  normally hyperbolic arc,

    \item[--] or there is a sequence of forward iterates that accumulates in a fundamental domain of the unstable manifold of $z.$
\end{itemize}

\end{remark}


\begin{proof}[Proof of lemma \ref{whole-centerunstable}]

Let $y\in \alpha(x)$ and let $z\in \omega(y)$; on one hand it follows immediately that $z\in \alpha(x)$, on the other, we can apply lemma \ref{cu-omega} to the point $z$. 

Let us consider first the case that  both center  unstable of $z$  are in $\La^u$.   Let $(n_k)$ be a sequence of positive integers  such that $f^{-n_k}(x)\to z$. It could hold that  
\begin{itemize}
\item[--] either there are iterates $(f^{-n_{k}}(x))$ accumulating on $z$  that are contained in the  local center-unstable of $z,$

\item[--] or there are two iterates $f^{-n_k}(x), f^{-n_{k'}}(x)$ that belong to the local stable of $z,$

\item[--] or there are infinitely many iterates $f^{-n_k}(x)$ that are in the interior of a  quadrant $Q$ determined  by the local stable and center-unstable of  $z.$

\end{itemize}

In the first case since both branches of $z$ are in $\La^u$ it holds by continuity of the center-unstable manifold along the orbit the same holds for $f^{-n_k}(x)$ and so for $x$. In the second case, $f^{n_{k'}-n_k}(x)$ is in the local stable manifold of $x$ and so $x$ is a periodic point and therefore $x=z$ and again both center-unstable  branches of $x$ are in $\La^u.$ In the third case, let  $J_k^{+}$ be one arc in the center-unstable of $f^{-n_k}(x)$ and $J_k^{-}$  in the other, and let us treat them separately. For each $J^+_k$ (the arcs contained in the quadrant $Q$), let $\hat J^+_k$ be the arc in the center-unstable given by the intersection of the local stable manifold of $J^+_k$ with the center-unstable of $z$ and observe that this arc in $\La^u$. If it holds that there is a subsequence of $(f^{n_{k'}}(J^+_{k'}))$ with length uniformly bounded by below, since $f^{n_{k'}}(\hat J^+)\in W^s_{\la^{n_{k'}}.\de}(f^{n_{k'}}(J_{k'}^+))$ it follows also that the length of the subsequence $(f^{n_{k'}}(\hat J^+_{k'}))$ is also bounded by below, therefore  there is an arc $\hat J$ containing $x$ and contained  in one of the branches of the center-unstable of $x$ (observe that $f^{n_{k'}}(J^+_{k'})$ contains $x$) that is in $\La^u.$ If the lengths of $f^{n_k}(J^+_k)$ goes to zero, for each $k$ one takes $k'$ larger than $k$ such that $f^{-n_{k'}}(x)$ is closer to $z$ than $f^{-n_k}(x)$ and since are accumulating in the same quadrant let $x_{k'}$ in the intersection of the local stable of $f^{n_k}(x)$ and $J^+_{n_{k'}};$ from the fact that $f^{-n_k}(x)$  is in the local stable manifold of $x_{k'}$ and the length of $f^{n_{k'}}(J_{k'})$ goes to zero follows that that $f^{n_{k'}}(x_{k'})$ is arbitrary close to $x$ and $f^{n_{k'}-n_k}(x)$ provided that $k$ and $k'$ are large enough; therefore $x\in \omega(x)$ and so by lemma \ref{cu-omega} it follows that both center-unstable branches of $x$ is in $\La^u$ or it is a periodic point with at least  one center-unstable branch in $\La^u$. A similar argument can be made for the other branch $J^-_k$ with a caveat; if the length of the iterates of $f^{n_k}(J^-_k)$ go to zero, one considers  the arcs  $\hat J^-_k$  in the center-unstable of $z$ given by the intersection of the local stable manifold of $J^-_k$ with the center-unstable of $z$ and observe that these arcs are in $\La^u$ and each of them contain one local center-unstable branch of $z$, named $J_z$; therefore, if the length of the iterates of $f^{n_k}(J^-_k)$ go to zero, also the length of $f^{n_k}(J_z)$ and by lemma \ref{centerunstable2} follows that $z$ is in the basin of attraction of a (semi)attracting periodic point or it is in the basin of a normally hyperbolic arc and so converging to a periodic point; since $z$ is  in $\alpha(x)$ it follows that $x$ is  a periodic point.






Let us consider the second option of lemma  \ref{cu-omega}, i.e., $z$ is a periodic point. If $z$ is an attracting periodic point, and $z\in \alpha(x)$ then $x=z$ and so $x$ is an attracting periodic point in $\La^u$ and by claim \ref{attracting} it follows also  that $x$ belongs to a normally hyperbolic arc in $\La^u$. If $z$ is a semi-attracting periodic point then  the pre-iterates of $x$ that accumulate on $z$ are either in the basin of attraction of $z$ or they belong to the connected component of $B_\e(z)\setminus W^s_\e(z)$ that is not in the basin of attraction of $z.$ In the first case, it follows that $x=z$ and again by claim \ref{attracting} it follows also  that $x$ belongs to a normally hyperbolic arc in $\La^u$. In the second case, by remark \ref{att-NH} it follows that either $x=z$, or  it belongs to an invariant normally hyperbolic arc or there exists $z'\in W^u_{\e}(z)$ accumulated by pre-iterates of $x$ and so $z'$ is in $\La^u$ and by claim \ref{cpt-lemma} it follows that the unstable branch of $W^u(z)$ that contains $z'$ is in $\La^u$ and so there are pre-iterates of $x$ that accumulates at a point such that both center-unstable branches are in $\La^u;$ therefore we can argue as done before.


\end{proof}

\begin{proof}[Proof of theorem \ref{disk2}]

Given a point $x\in \La^u$, let $\hat W^{cu}_{\e(x)}(x)$ be the piece of the center-unstable manifold of $x$ that is contained in $\La^u.$ From lemma \ref{whole-centerunstable} it follows that the whole center-unstable $W^{cu}_{\e(x)}(x)$ is contained in $\La^u$ or $x$ is a boundary periodic. Now, we consider the box $B(x)$ defined as follows  $$B(x):=\cup_{y\in W^{cu}_{\eps(x)}(x)}W^s_\de(y).$$

It follows that either $x$ is in the interior of $B(x)$, or $x$ is a boundary
periodic point lying in the boundary of $B(x)$. In the latter case we modify
the diffeomorphism as follows. Recall (lemma \ref{whole-centerunstable})
that there are only finitely many boundary periodic points; taking an
iterate of $f$, we assume that they are fixed. For each boundary fixed point
$x$, let $C_x$ be the connected component of $B_\e(x)\setminus W^s_\e(x)$
that does not intersect the center-unstable branch of $x$ contained in
$\La^u$; in particular $C_x\cap \La^u=\emptyset$, and the components $C_x$
can be chosen pairwise disjoint. We choose a $C^2$ dissipative
diffeomorphism $g$ of the disk such that:
\begin{itemize}
\item[--] $g=f$ on the complement of $\bigcup_x C_x$;
\item[--] for each boundary fixed point $x$, the center-unstable branch of
$x$ contained in $C_x$ lies, for $g$, in the basin of attraction of $x$;
that is, $x$ is a (semi-)attracting fixed point of $g$.
\end{itemize}

The following three observations justify replacing $f$ by $g$ in the rest of
the argument.

First, the perturbation does not affect $\La^u$, neither topologically nor
at the level of the derivative: $g=f$ on $\La^u$, so $\La^u$ is
$g$-invariant; and since $g$ and $f$ coincide on the closure of the
complement of $\bigcup_x C_x$, which accumulates on every point of $\La^u$,
the continuity of the derivatives gives $D_zg=D_zf$ for every $z\in\La^u$.
Hence the derivative cocycles of $f$ and $g$ along $\La^u$ coincide; in
particular $\La^u$ carries the same dominated splitting for both maps. Moreover, since
$\Ga^u(p)\cup\{p\}\subset \La^u$ is disjoint from $\bigcup_x C_x$, the
unstable branch of $p$ is the same for $f$ and for $g$, and $\La^u$ is the
accumulation set of that branch for both maps.

Second, $g$ is again mildly dissipative: an ergodic $g$-invariant measure
either gives no weight to $\bigcup_x C_x$, in which case it is
$f$-invariant with the same stable branches at almost every point, or
charges some $C_x$; in the latter case, since every point of $C_x$ is
attracted under $g$ to the fixed point $x$, the measure is the Dirac measure
at $x$, and the stable branches of $x$ for $g$ contain its stable branches
for $f$, which reach the boundary of the disk by the mild dissipation of
$f$.

Third, by construction each boundary fixed point $x$ is a (semi-)attracting
fixed point of $g$, so that the extended box $B(x)$ obtained by adding
to $B(x)$ the component $C_x$  which is attracted to $x$, contains $x$ in
its interior.


After this modification (we keep calling $g$ as $f$), we can cover $\La^u$ by $$B=\cup_{z\in \La^u} \overset{\circ}{B}(z).$$

Observe that $B$ satisfies:

\begin{itemize}
     \item[--] it is open and connected, 
\item[--] $\overline{f(B)}\subset B,$
 \item[--] it is simply connected.
\end{itemize}

The  first item follows from the fact that $B$  is an open union of neighborhood  of all the points of a connected set. To prove the second point, observe that 
 given $y\in B$ there is $x\in \La^u$ such that $y\in B(x)$, so  there is $x'$ in the interior of $ W^{cu}_{\e(x)}(x)\cap \La^u$ such that $y\in W^s_{\e}(x')$ and so $f(y)\in W^s_{\la.\e}(f(x'))\subset W^s_\e(f(x'))$ and so $f(y)\in B(f(x'))\subset B.$ If $B$ is not simply connected, there are connected components of $\DD\setminus B$ that do not intersect the boundary of $\DD$; let $U$ be the union of these components and since $\overline{f(B)}\subset B$ it follows that $f(U)$ contains $U$, contradicting  the fact that $f_{|\DD}$ is dissipative.
 
 

Observe that $\cap_{n>0} f^n(B)=\La^u$ and so the limit set of $f_{|B}$ has dominated splitting. Now we can apply the remark \ref{decomposition}. 
If there is a non-trivial homoclinic class in the limit set of $f$ inside $B$,  then there is a loop $\ga$ formed by an unstable arc $\ga^u$ and a stable arc $\ga^s$ contained in $B$. On the other hand, observe that the stable foliation in $B$ is transversal to $\ga^u$ and $\ga^s$ is a stable leaf of that foliation; however, there is a point in $\ga^u$ tangent to the stable foliation; a contradiction: a closed curve transversal to a foliation on a disk would imply the loop bounds a disk filled by leaves (by Poincar\'e–Bendixson), which would force a leaf to terminate inside the disk; contradiction with the leaves being stable manifolds of points with no terminating dynamics. So, the limit set is just formed by a  finite number of periodic points and finite normally hyperbolic  periodic arcs. Taking an iterate of $f$, one can assume that the periodic points and arcs are fixed. Let us first conclude that there are not normally hyperbolic arcs. If the $\Ga^u(p)$ is not a normally hyperbolic arc (otherwise the theorem is proved) but  accumulates on a non-trivial normally hyperbolic invariant arc $J$, $\Ga^u(p)$ has to intersect the stable local manifold of of one of the  endpoints of $J$, that is a fixed point $q$; this implies that there is a closed arc $\ga=\ga^u\cup \ga^s$ formed by an arc  $\ga^u\subset \Ga^u(p) $ and an arc $\ga^s\subset W^s_\DD(q)$ that is contained in $B,$ and arguing as before one gets a tangency between $\ga^u$ and the stable foliation in $B$ and therefore a critical point in $\Ga^u(p)$. Now we  conclude that there is only one fixed point $q$ in $\La^u$ and it is a (semi-)attracting fixed point: if in $\La^u$ there is a saddle fixed point $q$ such that one of its unstable branches is in $\La^u$, then $\Ga^u(p)$ intersects the local stable of $q$ (if it accumulates without intersecting one gets a point in the local stable manifold that it is in the limit set of $\La^u$ but is not a fixed point); if $\Ga^u(p)$ intersects $W^s_\DD(q)$ without crossing, one gets a tangency between $\Ga^u(p)$ and $W^s_\DD(q)$ and so a critical point in $\Ga^u(p);$  if $\Ga^u(p)$ crosses  $W^s_\DD(q)$, arguing as before we get a closed arc in $B$ and again a tangency between $\ga^u$ and the stable foliation in $B$. So, there are only (semi)attracting fixed points. Let $z$ be one of these (semi)attracting fixed points and let $x$ be in $\La^u$ and in the basin of attraction of $z;$ by claim \ref{attracting}, there is a normally hyperbolic arc $\Ga$ in $\La^u$ that one of its extreme points is  $z$  and arguing as before, $\Ga\subset \Ga^u(p)$ and if it does not coincide with $\Ga^u(p)$ then one gets a  critical point in $\Ga^u(p).$

Finally, the conclusions obtained for the modification of $f$ transfer to $f$. The critical
points produced by the tangency arguments above lie in $\La^u$ (they are
obtained either on $\Ga^u(p)$ or on its accumulation set), and along
$\La^u$ the derivative cocycles of $f$ and its modification coincide, so they are
critical points of $f$ as well. Likewise, a normally hyperbolic invariant
arc contained in $\La^u$ for the modification of $f$  is normally hyperbolic for $f$, since
normal hyperbolicity only involves the restriction of the derivative cocycle
to the arc.

\end{proof}

\section{Misiurewicz diffeomorphisms.  Proof of theorem \ref{thm:finite classes}}\label{sec:misiu}



Observe that by theorem \ref{disk2}, any periodic point in a non-trivial homoclinic class has an unstable branch that accumulates on the critical set.  Therefore, 
for the proof of theorem \ref{thm:finite classes}, we first address saddle periodic points for which at least one unstable branch accumulates on the critical set, and show that they can be grouped into a finite number of homoclinic classes; later it is shown that the remaining periodic points   are contained in the interior of   normally hyperbolic arcs. Throughout this section, for a saddle periodic point $q$ and one of its
unstable branches $\Ga^u(q)$, we write $\La^u(q)$ for the accumulation set
of that branch, in the sense of section \ref{sec:disk}.

Before going  into the proof, let us clarify a few definitions. Recall that for mildly dissipative diffeomorphisms, any branch of a stable manifold of a point is not contained in the disk $\DD$; in particular, one can define $W^s_\DD(x)$ as the connected component of $W^s(x)\cap \DD$ that contains $x$; similarly, given a stable branch $\Ga^s(x)$ one can define $\Ga^s_\DD(x).$ Observe that by the mildly dissipative hypothesis, $\DD\setminus W^s_\DD(x)$ has two connected components $\DD^+_x$ and $\DD^-_x$. 

\begin{definition}\label{defi:crossing} Let $p$ and $q$ be two saddle periodic points and let $x\in W^s(p)\cap W^u(q)$; we say that $W^u(q)$ {\it crosses} $W^s(p)$ at $x$,  if there are two connected arcs $\ga^+$ and $\ga^-$ in the unstable manifold of $q$ and $x$ is an extreme point of both arcs satisfying that 

\begin{itemize}
    \item[--] $\ga^+$ intersects $\DD^+_x$ and is contained in the closure of  $\DD^+_x$ and 
    \item[--] $\ga^-$ intersects $\DD^-_x$ and is contained in the closure of  $\DD^-_x.$
    
\end{itemize}

The point $x$  is also called a crossing intersection of $W^s(p)$ and $W^u(q).$

It is said that the {\it unstable manifold of $q$ crosses the stable of $p$} if there exists $x\in W^s(p)$ such that $W^u(q)$ intersects both components $\DD^+_x$ and $\DD^-_x.$
    
\end{definition}

Observe that the unstable of a periodic point crosses the stable of other periodic point if and only if  there is a crossing intersection.

Classically, a homoclinic class of a periodic point is defined as the closure of all the transversal intersections between the stable and unstable manifold of that periodic point. We extend that definition as follows:

\begin{definition}\label{defi:generalizedHC}
Given a periodic point $p$, the {\it generalized homoclinic class} is the closure of all the crossing points of the stable and unstable of $p.$

Given two periodic points $p$ and $q$, it is said that they are generalized homoclinically related if the unstable manifold of $q$ crosses the stable of $p$ and vice versa. 
    
\end{definition}

If the  diffeomorphism is $C^\infty$, it follows that if two periodic points are generalized homoclinically related then they are also homoclinically related in the sense that there is an intersection between the stable and unstable manifold that is transversal (see \cite{BCS}).

For brevity, we keep saying  homoclinic class when we are actually referring to generalized homoclinic class.

Coming back to the proof of theorem \ref{thm:finite classes}, we prove the following next proposition and lemma.

\begin{proposition}\label{critical-Misiu} 
Under the hypothesis of theorem \ref{thm:finite classes} it holds that all saddle periodic points that are close enough to a same critical point and for which at least one unstable branch accumulates on the critical set are homoclinically related.     
\end{proposition}

For the other  saddle periodic points, one can prove the following next lemma.

\begin{lemma}\label{critical-Misiu2} Under the hypothesis of theorem \ref{thm:finite classes}, any saddle periodic point such that none of its unstable branches accumulates on a critical point belongs to the interior of a normally hyperbolic periodic arc.
    
\end{lemma}

\begin{proof}


The proof follows immediately from theorem \ref{disk2}. In fact,  if unstable set of the periodic point is trivial, then it  belongs to a normally hyperbolic periodic arc; its unstable branches accumulate either on a semi-attracting periodic point or on a normally hyperbolic periodic arc. Adding these points or arcs to the unstable branches yields a closed periodic curve without critical points, which is therefore a normally hyperbolic arc.

\end{proof}

Before proving proposition \ref{critical-Misiu}, we show how the previous proposition and lemma combined with proposition \ref{PS2-1} provide the proof of theorem \ref{thm:finite classes}.

\paragraph{\it Proof of theorem \ref{thm:finite classes}:} After lemma \ref{critical-Misiu2} it is enough to consider saddle periodic points with branches accumulating on the critical set. By proposition \ref{critical-Misiu} for each critical point $c$ there is $\e(c)>0$ such that  the periodic points  whose orbits intersect $B_{\e(c)}(c)$ are homoclinically related and so they belong to a homoclinic class. There is a finite number of critical points $\{c_1,\dots, c_n\}$ such that  $B=\cup_{i=1}^n B_{\e(c_i)}(c_i)$ contains the critical set and all periodic points with their orbits intersecting $B$ are therefore  contained in a finite number of homoclinic classes. The closure of the orbits of the saddle  periodic points that do not intersect $B$ has dominated splitting and therefore   from proposition \ref{PS2-1} it is also concluded that they also belong to a finite number of homoclinic classes. 

\qed

\subsection{Proof of proposition \ref{critical-Misiu}.}


We split the proof of Proposition \ref{critical-Misiu} into two parts according to whether  a sequence of saddle periodic points accumulates on critical points whose $\omega$ or 
$\alpha$ limits are non-trivial (interior points or heterocycles), or only on boundary critical points (see definition \ref{defi:boundary}). In the first case, treated below, Proposition \ref{critical-Misiu} follows directly from Lemmas \ref{interior} and \ref{HC}. The second case reduces, via the notion of a  $*-$sequence (definition \ref{def:*-seq}), to proposition \ref{prop:*-seq}, whose proof occupies the rest of the section.


For a  Misiurewicz diffeomorphism, it follows that the  omega-limit of any critical point is a set with dominated splitting and so by proposition \ref{dicho-omega} it is either a periodic orbit, or a heterocycle or a set that contains interior points.

\begin{definition}\label{defi:boundary} 
    
Let ${\rm Crit_I}(f)$ be the set of critical points $c$ such that  $\omega(c)$ or $\alpha(c)$  either contains interior points, is a heterocycle, is contained in a normally hyperbolic arc, or lies in the interior of the basin of attraction of a (semi-)attracting periodic point. Let us call ${\rm Crit_{II}}(f)$ its complement. 

The critical points in ${\rm Crit_{II}}(f)$ are called {\it boundary critical points}  and the periodic points whose stable or unstable manifolds contain a boundary critical point are called {\it boundary saddle periodic points}. 
\end{definition}

Now, let us consider a sequences $(p_n)$ of saddles that accumulates on a critical point. Observe that if there is a sequence $(p_n)$ of periodic points with unbounded period that accumulates at a periodic point, then that periodic point has to be a saddle periodic point that can not be contained in the interior of a normally hyperbolic arc. Moreover, if a sequence of periodic points $(p_n)$ accumulates at a critical point, then their periods must be unbounded.

Given a sequence $(p_n)$ of saddle periodic points, let us define ${\rm Crit}(f, (p_n))$ the set of the critical points accumulated by $(p_n)$ and they can be divided into two sets: ${\rm Crit_I}(f, (p_n))= {\rm Crit_I}(f)\cap {\rm Crit}(f, (p_n))$ and ${\rm Crit_{II}}(f, (p_n))= {\rm Crit_{II}}(f)\cap {\rm Crit}(f, (p_n))$

\paragraph{Case ${\rm Crit_I}(f, (p_n))\neq \emptyset$: }

\begin{lemma}\label{interior} Under the hypothesis of theorem \ref{thm:finite classes}, let $(p_n)$ be  a sequence of saddle periodic points such for  each of them at least one unstable branch accumulates on the critical set and they accumulate on an interior point $x$ of a   set $\La$ with dominated splitting; it follows  that the periodic points of the sequence $(p_n)$ close enough to $x$  are homoclinically related. 

\end{lemma}

 \begin{proof}

From the fact that $\La$ has dominated splitting and $x$ is an interior point,  
one can get a small rectangle $R$   satisfying: 

\begin{itemize}
    \item[--]  it is bounded by local stable and unstable arc of recurrent points in $\La,$
    \item[--] $x$ is in the interior of $R$,
    \item[--] it is disjoint from the  critical points.

\end{itemize}

Given a periodic point $p_n$ let us consider $g=f^{per(p_n)}$. Observe that by mild dissipation, the stable branch of any $p_n$ close enough to $x$ is not contained in $R$ and therefore it intersects the unstable boundary of $R.$  Similarly, if $R$ is sufficiently small then there are no critical points in $R$ and so if at least one of the  unstable branches of $p_n$  accumulates on the critical points it follows  that  the unstable branch is  not contained in $R$ and so  it  intersect the stable boundary of $R.$ From that fact, the $\la-$lemma, and using   that the points that provides the stable and unstable arcs that bound the rectangle are recurrent, it follows that there are   stable and unstable arcs of any periodic $p_n$ close to $x$ that accumulates on the stable and unstable arcs that bound the rectangle; that implies that the stable and unstable manifold on any pair of periodic points close to $x$ intersects each other.

\end{proof}

\begin{lemma}\label{HC} Under the hypothesis of theorem \ref{thm:finite classes}, let $(p_n)$ be  a sequence of saddle periodic points such that  for each of them at least one unstable branch accumulates on the critical set and the sequence $(p_n)$ accumulates at a point $x$ such that $\omega(x)$ is a heterocycle; it follows  that the  periodic points in the sequence $(p_n)$ close enough to $x$  are homoclinically related. 

\end{lemma}

\begin{proof} From the remark \ref{rmk:heterocycle} one gets  a rectangle $R$  containing a periodic point $q$ of the heterocycle in one corner and bounded by two parallel stable arcs and two parallel unstable arcs of that point. Arguing  as in lemma \ref{interior} it follows that  any periodic $p_n$ close to the heterocycle is homoclinically related to $q$ and so homoclinically related to each other.


\end{proof}





    


As a consequence of the previous two lemmas, it follows an immediate corollary:

\begin{corollary}\label{cor:crit1}
     Let $(p_n)$ be  a sequence of saddle periodic points such that for  each of them  at least one unstable branch accumulates on the critical set and  the orbits of the sequence $(p_n)$ accumulate on $Crit_I(f)$; it follows that  they belong to a finite number of homoclinic classes.
\end{corollary}


\paragraph{Case ${\rm Crit_I}(f, (p_n))=\emptyset$.} In that case, ${\rm Crit}(f, (p_n))= {\rm Crit_{II}}(f, (p_n)).$ By proposition \ref{PS2-1} there are only a finite number of boundary saddle periodic points that contain the critical points in $Crit_{II}(f).$  So, to conclude proposition \ref{critical-Misiu} it remains to consider the sequence of periodic points that accumulate only on boundary critical points.

\begin{definition}\label{def:*-seq}  
Let $(p_n)$ be a sequence of saddle periodic points such that for each of them at least one of its unstable branches accumulates at a critical point. Let $\La$ be the accumulation set of the orbits of the periodic points $(p_n)$; we say that the sequence $(p_n)$ is a {\it $*-$sequence} if 

\begin{itemize}
    \item[--] the set of critical points in $\La$ is finite and they are all boundary  critical points,

    \item[--] any  compact invariant set with dominated splitting  contained in $\La$ is formed by a finite number of saddle periodic orbits.
\end{itemize}

\end{definition}

Therefore, as a consequence of  corollary \ref{cor:crit1},  to finish the proof of proposition \ref{critical-Misiu} it remains to prove the following proposition.

\begin{proposition}\label{prop:*-seq} Under the hypothesis of theorem \ref{thm:finite classes}, let $(p_n)$ be a $*-$sequence, then it holds that they belong to a finite number of homoclinic classes.
\end{proposition}

To prove proposition \ref{prop:*-seq}, the goal is to show that there is a ``touching cycle" (defined below) between  periodic points that ``encloses" some iterates of  the $*-$sequence and that is used to show  that the periodic points in  $*-$sequence are all homoclinically related.

\paragraph{Quadrants, sequences and cycles.} Given a saddle (or semi-attracting) periodic point  $q$, its local  stable and center-unstable branches (denoted as $\Ga_l^s(q)$ and $\Ga^{cu}_l(q)$ respectively) define at most $4$ quadrants ($2$ in case that $q$ is semi-attracting); more precisely, the complement of the union of the local  stable and center-unstable set $q$ in a small neighborhood of it has $4$ connected components; let us denote by $(q,Q)$ one of these quadrants that can also be identified with the triplet $(q, \Gamma^s_l(q), \Gamma^{cu}_l(q))$ that indicates  the branches that bounds the quadrant. When the center-unstable branch is an unstable branch, we denote $\Gamma^{cu}_l(q)$ with $\Ga^u_l(q).$ 

Given a point $x$ that belongs to the local stable branch of a periodic saddle $q$ and a sequence $(p_n)$ that accumulates at $x$ on one side of the local stable manifold  of $q$, there exists a quadrant $Q$ bounded by branches $\Gamma^s_l(q), \Gamma^u_l(q)$ such that 
\begin{itemize}
    \item[--] the sequence $(p_n)$ is contained in $Q$,
    \item[--] there exist $y\in \Gamma^u_l(q)$ and positive integers $(k_n)$ such that $f^{k_n}(p_n)\to y$,
    \item[--] for any $n$ and any  $0\leq j \leq k_n$ it holds that $f^j(p_n)\in Q.$
\end{itemize}


\begin{definition}\label{def-cycle-points}
Given a finite set of saddle periodic points $(q_i)$  and respective quadrants $(q_i, \Ga^s_l(q_i), \Ga^u_l(q_i))$, we say that there is a {\it cycle} if for each $i$, 
\begin{itemize}
    \item[--] there is a point $(y_i)$ in $\Ga^{u}_l(q_i),$
    \item[--] there is a point $(x_i)$ in $\Ga^{s}_l(q_i),$
     \item[--] a sequence of forward iterates $(f^{n_j}(y_i))$ contained in the quadrant bounded by  $\Ga^s_l(q_{i+1})$ and $\Ga^u_l(q_{i+1})$ and  accumulating on  $x_{i+1}\in \Ga^{s}_l(q_{i+1})$.
\end{itemize}

\noindent We say that the cycle is a {\it critical  cycle} if the points $(x_i)$ are critical points.

\noindent If for any $i$ there exists $n_i>0$ such that $f^{n_i}(y_i)=x_{i+1} \in\Ga^s_l(q_{i+1})$ and there is an arc in $\Ga^u(q_i)$ that contains $x_{i+1}$ and intersects $Q_{i+1}$, we say that the cycle is a {\it touching cycle}.

\end{definition}

To incorporate all the data defined above, a cycle between periodic points and quadrants is noted as a finite sequence of quadruplets $(q_i, Q_i, x_i, y_i).$



\begin{figure}[h]
\centering
\includegraphics[scale=0.45]{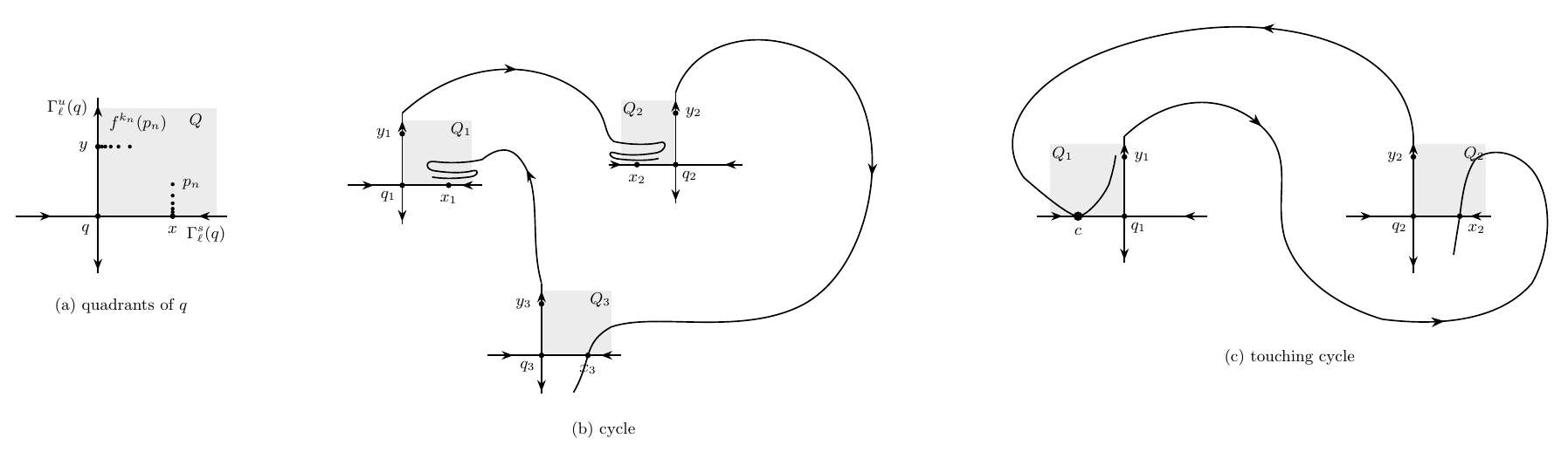}
\caption{Quadrants and cycles}
\label{cycles}
\end{figure}

\begin{proposition}\label{touching-cycle}  Given a touching cycle  of a mildly dissipative diffeomorphisms of the disk, for any saddle $q$ in the cycle there is a sequence $(q_n)$ of hyperbolic saddle periodic points  and  a point $x\in W^s_{\DD}(q)$ such that 

\begin{itemize}
\item[--] the sequence  $(q_n)$ accumulates at $x$ and  is contained in the quadrant $Q$ of $q$ involved in the cycle,
    \item[--] all saddles $(q_n)$ are homoclinically related, 
    \item[--]$W^s_{\DD}(q_n)$ intersects transversely $ \Ga^u_{l}(q)$ (the unstable branch of $q$ that bounds $Q$)  and  $W^s_{\DD}(q_n)\to W^s_{\DD}(q)$, 
    \item[--]  there is an arc $\ga^u_n\subset W^u(q_n)$ contained in $Q$ that joins $q_n$ with a point in  $W^s_\DD(q_{n+1}).$
\end{itemize}

\end{proposition}

\begin{proof} The proof extends that of lemma 5.8 in \cite{CPT}. We recall
its mechanism and indicate the additional input. Given the saddle $q$ in the
cycle (assumed fixed) and $x\in W^s_\DD(q)$ accumulated, inside the quadrant
$Q$, by the unstable branch $\Ga^u_l(q)$ involved in the cycle, lemma 5.8 of
\cite{CPT} produces small rectangles $R_n$ that accumulate on $x$,
parallel to $W^s_\DD(q)$, together with return times $r_n$ such that
$f^{r_n}(R_n)$ crosses $R_n$ transversally in at least two components. Each
such crossing yields, by the classical horseshoe construction, a hyperbolic
saddle $q_n$ contained in a non-trivial homoclinic class, with $q_n\to x$;
moreover the local stable manifolds satisfy $W^s_\DD(q_n)\to W^s_\DD(q)$ and
cross $\Ga^u_l(q)$ transversally.

The additional point, not needed in \cite{CPT}, is the last item of the
statement: that an unstable arc $\ga^u_n\subset W^u(q_n)$ inside $Q$ joins
$q_n$ to a point of $W^s_\DD(q_{n+1})$. This is obtained from the dissipation
exactly as the transversal crossing above: the rectangles $R_n$ can be
chosen so that $f^{r_n}(R_n)$ meets not only $R_n$ but also $R_{n+1}$, and
the component of $W^u(q_n)\cap Q$ realizing the crossing of $R_{n+1}$
furnishes the arc $\ga^u_n$ reaching $W^s_\DD(q_{n+1})$. In particular
$W^u(q_n)$ crosses $W^s_\DD(q_{n+1})$ and, symmetrically, $W^u(q_{n+1})$
crosses $W^s_\DD(q_n)$; hence consecutive saddles $q_n,q_{n+1}$ are
homoclinically related, and therefore all the $q_n$ belong to a single
homoclinic class.

\end{proof}

\begin{corollary}\label{touching} Under the hypothesis of theorem \ref{thm:finite classes}, given a touching cycle, it follows that  all saddle periodic points satisfying 

\begin{itemize}
    \item[--]  for  each of them  at least one unstable branch accumulates on the critical set,
    \item[--] they are close enough to at least one saddle involved in the  cycle and  are  contained in the correspondent quadrant involved in the cycle,
\end{itemize}
are homoclinically related.
    
\end{corollary}

\begin{proof} Let $(q, Q)$ be a saddle and a quadrant involved in the touching cycle and observe that nearby $q$ there are no critical points. Let $(q_n)$ be the sequence of saddle periodic points provided by proposition \ref{touching-cycle}. Given a saddle periodic point $p$ sufficiently close to $q$, one can get two consecutive saddles $q'_m$ and $q'_{m+1}$ that are two forward iterates of two points in the sequence $(q_n)$ such that the arcs $W^s_\DD(q'_m), W^s_\DD(q'_{m+1}), {\ga^u_n}'$ and $\Ga^u_q$ provided by proposition \ref{touching-cycle} form a rectangle $R_n$ close to $q$ that contains $p$ in its interior. Since there are no critical points nearby $q$ and the unstable manifold of $p$ accumulates at a critical point then it has to cross either $W^s_\DD(q_m)$ or $W^s_\DD(q_{m+1});$ on the other hand, $W^s_\DD(p)$ has to cross either ${\ga^u_m}'$ of $\Ga^u_l(q)$ and observe that  if $W^s_\DD(p)$ crosses $\Ga^u_l(q)$ then the stable manifold of $p$ accumulates on $W^s_\DD(q)$ from the side of $Q$ and so it crosses an arc $\ga^u_k$ of some saddle point $q_k$     in the sequence $(q_n).$ Since all the saddles $(q_n)$ are homoclinically related it follows that $p$ is also homoclinically related to them.

\end{proof}

To conclude proposition \ref{critical-Misiu}  it is enough to prove the following proposition.

\begin{proposition}\label{prop:*-touching} Under the hypothesis of theorem \ref{thm:finite classes} given a $*-$sequence it follows that  there exists a touching cycle such that the $*-$sequence satisfies the hypothesis of corollary \ref{touching}.

\end{proposition}



\subsubsection{Proof of proposition \ref{prop:*-touching}}

The proof proceeds in three stages. First, starting from the 
$*-$sequence we iteratively construct a chain of boundary saddle periodic points and quadrants; this chain closes into a {\it maximal critical cycle} with a marked sequence of boundary critical points. Second, we develop the geometric tools (a rectangle lemma, Pixton disks, and a transversality dichotomy for unstable branches (Proposition \ref{prop trans})) that govern how unstable manifolds interact with the boundary structure produced by the cycle. Third, we show that any maximal critical cycle is either itself a touching cycle or can be enlarged to one (Proposition \ref{existence of touching}); combined with Proposition \ref{touching-cycle} and Corollary \ref{touching}, this concludes the proof.

\paragraph{Building the maximal critical cycle.} The next corollary follows from  claim \ref{cpt-lemma}.

    

\begin{corollary}\label{jumping} Let $f$ be mildly dissipative and let $q_0, q_1, q_2$ be  three  saddle points and  $Q_1, Q_2$ two quadrants of $q_1, q_2$. Let  $\Ga^u_l(q_1)$ and $\Ga^u_l(q_2)$ be two  local unstable branches of $q_1$ and $q_2$ respectively, such that $\Ga^u(q_1)$  bounds $Q_1$ and $\Ga^u(q_2)$  bounds $Q_2$. If the iterates of $\Ga^u(q_0)$ accumulate on $q_1$ along $Q_1$ and $\Ga^u(q_1)$ accumulates on $q_2$ along $Q_2$, it follows that $\Ga^u(q_0)$ accumulates at $q_2$ along $Q_2$ and the accumulation set of $\Ga^u(q_0)$ contains $\Ga^u_l(q_2).$ 
    
\end{corollary}

Let $(p_n)$ be a $*-$sequence of periodic points, and let $c_0$ be a boundary critical point accumulated by the sequence. Let $q_0$ be the boundary periodic point that contains $c_0$ in its stable manifold;  let $Q_0$ be a  quadrant of $q_0$ that contains infinitely many points of the sequence $(p_n)$ and taking a subsequence one can assume that they all belong to the same quadrant.  Observe that there exists  a point $y_0$ in the local unstable branch $\Ga^u(q_0)$ that bounds $Q_0$ such that $f^{k_n}(p_n)\to y_0$ and $f^j(p_n)\in Q_0$ for any $0\leq j\leq k_n$ (observe that  there is also a sequence of iterates $(f^{j_n}(p_n))$ such that  $f^{j_n}(p_n)\to q$ and $0< j_n< k_n$).  Let us consider $\omega(y_0)$ and observe that any point in that set is also accumulated by iterates of $(p_n)$; either $\omega(y_0)$ contains a critical point $c_1$ that belongs to the stable manifold of a boundary periodic point or $\omega(y_0)$ has dominated splitting.  In the first case, it has to be a boundary critical point that belongs to one of the finite boundary periodic points. In the second case by the definition of $*-$sequence it follows that $\omega(y_0)$ is  a periodic orbit and so, there is a saddle periodic point $q_1$ and a forward  iterate of $y_0$ named $c_1$ that is in the local stable manifold of $q_1$.     In conclusion,  we obtain a new  saddle periodic point $q_1$,  a quadrant $Q_1$ and a point $c_1$ satisfying:

\begin{itemize}
    \item[--] $q_1$ and $c_1$ are  accumulated by  iterates of the sequence $(p_n)$ that are  contained in $Q_1$ and  

    \item[--] $q_1$ is accumulated  by $\Ga^u(q_0)$ along the quadrant $Q_1$.
\end{itemize}
Moreover, it follows that either 
\begin{itemize}
    \item[--] $c_1$ is a critical boundary point that belongs to the local stable manifold of $q_1$ (a boundary saddle period point)  and infinite forward iterates of $y_0$ accumulate on $c_1$, let us call this case type 1,  or 

    \item[--] $q_1$ is a   saddle periodic point (that could be a boundary saddle periodic point) and  there is a forward iterate of $y_0$ equal to  $c_1$ (that could be a boundary critical point), let us call this case type 2.
\end{itemize}

Informally, type 1 is the favorable case: the construction has reached a
 boundary critical point $c_1$ that belongs  in the stable manifold of the
boundary saddle $q_1$, and the chain can close there. Type 2 is the
inconclusive case: the forward orbit of $y_0$ has only passed through a
saddle $q_1$ without yet landing on a boundary critical point, so the
construction must be continued from $q_1$.

We repeat the argument from $q_1$, producing a chain of periodic points and
quadrants $(q_n,Q_n)$ each of type 1 or type 2; the goal is to show that the
chain either closes into a touching cycle directly or, after removing  the
inconclusive consecutive  type 2  via corollary \ref{jumping}, one gets a critical cycle
of boundary critical points. More precisely, if there are finitely many consecutive ones in that chain that are of type 2 followed by a saddle of type 1, using corollary \ref{jumping}, one can remove those consecutive points of type 2 and the next one in the chain is a critical boundary saddle point (type 1).  If there are infinitely many consecutive ones  in that chain  that are of type 2 and such that their orbits are uniformly bounded away from the critical set, by the definition of $*-$sequence it follows that they have to repeat and they form a touching cycle proving proposition \ref{prop:*-touching}.  If there are infinitely many consecutive ones in that chain that are of type 2 and such that their orbits accumulate on a  critical point, given that the periodic points in the chain are accumulated by the orbits of the sequence $(p_n)$, it follows that  their orbits also accumulate on that critical point and therefore it is a boundary critical point and, again, using corollary \ref{jumping} one can replace the consecutive chain of periodic points of type 2 by a critical boundary point (type 1). 

Therefore either  a touching cycle that encloses the $*-$sequence is obtained, or using the fact that the number of critical boundary points is finite there is a cycle of quadrants and marked points $(q_i, Q_i, c_i, y_i)$ such that  
\begin{itemize}
\item[--] $q_i$ is a boundary saddle periodic point,

\item[--] $Q_i$ is a  quadrant bounded by $\Gamma^s_l(q_i)$ and $\Gamma^u_l(q_i)$ that contains $c_i$ and $y_i$ respectively,

\item[--]  $c_i$ is a boundary critical point,

\item[--] the  forward  iterates of $y_i$ accumulate  on $c_{i+1}$ inside $Q_{i+1}$.

\end{itemize}
moreover,  the orbits of the sequence $(p_n)$ satisfy the following properties (named $(P)$)
 \begin{itemize}
\item[--] accumulates on the boundary critical points $(c_i)$ inside $Q_i$,

\item[--] accumulates on the periodic points $(q_i)$ inside $Q_i$.


\end{itemize}
 
The next remark shows that any critical cycle satisfying property (P) can be enlarged to other critical cycle also satisfying property (P) and thereby obtaining a maximal critical cycle.

\begin{remark}\label{enlarging}
    
Observe now that given a critical cycle, $(q_i, Q_i,c_i, y_i)$ if there are two consecutive periodic points  $q_i, q_{i+1}$ in the cycle and a third periodic point $\bar q$ such that 

\begin{itemize}
    \item[--] there is a critical point $\bar c$ that belongs to some stable branch of $\bar q$ that  is  accumulated by iterates of $y_i,$

    \item[--] there is a point $\bar y$ in the unstable branch of $\bar q$ that is also accumulated by iterates of $y_i$ such that forward iterates of $\bar y$ accumulate on $c_{i+1},$

\end{itemize}

 then   one can enlarge the cycle including  $(\bar q, \bar Q, \bar c, \bar y)$ and that the orbits of the sequence $(p_n)$ also satisfy the properties  $(P)$ listed above, and so  $\bar c $ is a boundary critical point. In that way, using that the boundary critical points are finite,  a critical cycle can be enlarged to a {\it 
 maximal critical cycle}.
 \end{remark}

Proposition \ref{prop:*-touching} follows from next  proposition. 


\begin{proposition}\label{existence of touching} Given a maximal critical cycle then either the cycle is touching or it is contained in a larger touching cycle (that it is not necessarily a critical cycle).

\end{proposition}

The rest of the section focuses on proving the previous proposition.  For that, first we need to introduce a few tools. 

\paragraph{Rectangles, Pixton disks and transversality  dichotomy.} 
The next three results give the geometric infrastructure used in the proof of Proposition \ref{existence of touching}: a rectangle lemma controlling the local stable geometry near an accumulation point, the notion of a Pixton disk associated with a non-touching cycle, and a transversality dichotomy for unstable branches.

We start with a simple lemma for mildly dissipative diffeomorphisms of the disk.

    


\begin{lemma}\label{rectangle} For any  mildly dissipative diffeomorphism of the disk it follows that given a saddle periodic point $q$ and $x\in W^s_{\DD}(q)$ that is accumulated by the unstable branch of a saddle periodic point $p$, there exists a rectangle $R$ arbitrarily close to $x$ (see figure \ref{rectangle-fig}) bounded by two pairs of parallel  arcs $(\ga_1, \ga_2)$ and $(\be_1, \be_2)$ satisfying that
\begin{itemize}
    \item[--] $\ga_1$ is contained in the local stable manifold of $q$ and $x$ belongs to the interior of $\ga_1,$
    \item[--] there is a sequences of connected arcs $(\si_n)$ contained in the unstable branch of $p$ such that 
    \begin{itemize}
     \item[--] each $\si_n$ is contained in $R$ and is disjoint with $\be_1$ and $\be_2,$
     \item[--] one extreme point of the arc $\si_n$ belongs to $\ga_2$ and the other extreme points $(x_n)$ satisfy that $x_n\to x.$

    \end{itemize}

\end{itemize}

\end{lemma}

\begin{proof} One can choose an open arc $\delta$ in $W^s_\DD(q)$ whose backward iterates are contained in the complement of the disk. Then, one can choose two consecutive iterates $\delta_i, \delta_{i+1}$ of $\delta$ such that the point $x$ is in the interior of an  arc $\gamma$ contained in the local stable branch bounded by $\delta_i$ and $\delta_{i+1}$ (observe that since  $x$ is accumulated by the unstable branch of a periodic point $p$, its full orbit is contained in the disk). Extending the arc $\gamma$ to an arc $\gamma_1$  (also inside the local stable branch of $q$), one can assume that the extreme points of $\gamma_1$ are contained in $\delta_i$ and $\delta_{i+1}$ respectively. Now, one can choose two parallel transversal  arcs $\beta_1, \beta_2$ in $Q$ (the quadrant of $q$ where the unstable branch of $p$ accumulates at $x$) and such that one of the  extreme points of each arc coincides with the opposite extreme points of $\gamma_1$; observe that if both arcs are small enough then their backward iterates are fully contained in the complement of the disk. Then, one chooses an arc $\gamma_2$ parallel to $\gamma_1$ and close to it, to form a rectangle $R$. The unstable branch $\Ga^u(q)$ that  accumulates at $x$ is contained in the disk so it cannot intersect the boundaries $\beta_1, \beta_2$ and given that $p$ is not in $R$, its unstable branch intersects  $\gamma_2$; now, one can take all the connected components $\Ga^u(q)$ intersected with $R$ and   some of those connected components  accumulate on $x.$
    
\end{proof}

\begin{figure}[h]
\centering
\includegraphics[scale=0.75]{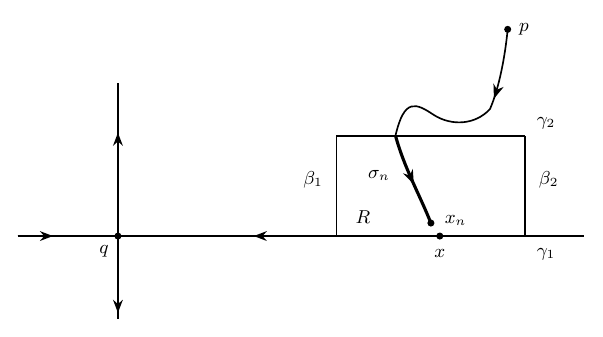}
\caption{Rectangle in lemma \ref{rectangle}}
\label{rectangle-fig}
\end{figure}

Now we introduce a definition that extends one provided in \cite{CPT} (see definition 5.9 in that paper).

\begin{definition} Given a cycle  of periodic points and quadrants  $(q_i, \Ga^s_l(q_i), \Ga^u_l(q_i))$ such that $\Ga^u(q_i)$ does not intersect  $\Ga^s(q_{i+1})$, a {\it Pixton disk} (see figure \ref{fig_pixton}) associated with that sequence  is a topological disk $D$ bounded by a collection of arcs $(\ga^s_i, \ga^u_i, \de^i)$ such that

\begin{itemize}
   \item[--] $\ga^u_i$ is an arc in $\Ga^u(q_i)$,

   \item[--] $\ga^s_i$ is an arc in $W^s_\DD(q_{i})$ with one endpoint is $q_i$ and another endpoint is a point $x_i$ that is accumulated by $\Ga^u(q_{i-1}),$

   \item[--] $\de_i$ is an arc that joins $\ga^u_i$ and a point $x_{i+1}$ in $\ga^s_{i+1}$ satisfying that $f(\de_i)\cap \de_i=\emptyset$ (it could be the case that $\de_i$ is empty and in that case $x_{i+1}\in \ga^u_i\cap \ga^s_{i+1}$).
\end{itemize}

\end{definition}

\begin{figure}[h]
\centering
\includegraphics[scale=0.75]{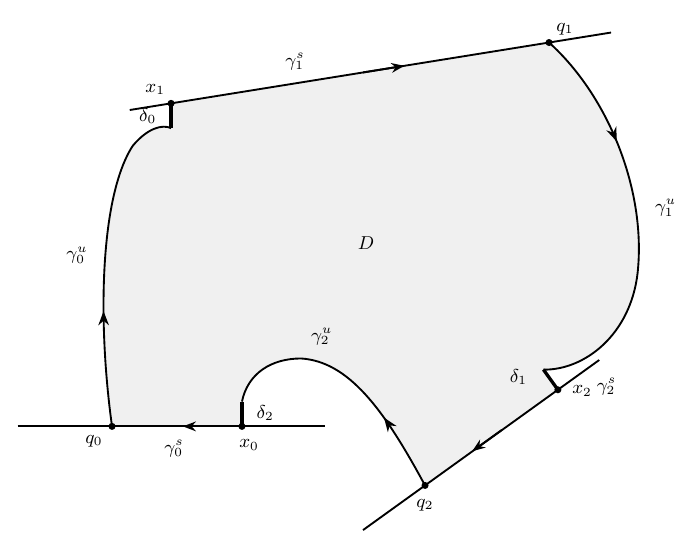}
\caption{Pixton disk involving three saddles.}
\label{fig_pixton}
\end{figure}
A version of the next  lemma  is stated in \cite{CPT} and its proof follows from lemma 5.10. For completeness (and because the assumption in that lemma are not quite similar) we include the proof.

\begin{lemma}\label{rmk:pixton} Given a  Pixton disk, there  is a point $x$ in one of the unstable branches such that  its forward iterates are not in the interior of the Pixton disk and therefore neither  $\omega(x)$.
    
\end{lemma}

 \begin{proof}
     
The proof follows from showing that $K=\alpha(x_i)$ is disjoint from the interior of the Pixton disk. In fact, if for some $x_i$ it holds that its $\alpha$-limit is inside $D$, let $f^{-n_k}(x_i)\in D$. From the mild dissipativeness and the fact that the unstable branches of the points in the cycle  do not intersect the stable branches, it follows that $W^s_\DD(f^{-n_k}(x_i))$ intersects some $\de_j$. Since $f(\de_j)\subset D$ (provided that $\de_j$ is small) it follows that $f(W^s_\DD(f^{-n_k}(x_i)))\cap D\neq\emptyset$ and since $f(W^s_\DD(f^{-n_k}(x_i)))\subset W^s_\DD(f^{-n_k+1}(x_i))$ it follows again that $W^s_\DD(f^{-n_k+1}(x_i))\cap \de\neq \emptyset.$ Inductively, it follows that $W^s_\DD(x_i)\cap \de_j\neq \emptyset,$ a contradiction.
\end{proof}

The next proposition is a preparatory one to prove proposition \ref{existence of touching} and states that under the hypothesis of theorem \ref{thm:finite classes} any unstable branch of a saddle periodic point that is not contained  in the basin of attraction of a (semi)attracting periodic point has to cross the stable of another saddle. The proof uses the dichotomy of  theorem \ref{disk2} and that if the critical points accumulated by the unstable branch of the saddle are not interior critical points, then there is a cycle involving boundary saddle periodic points and therefore a Pixton disk as defined before. 

\begin{proposition} \label{prop trans} Under the hypothesis of theorem \ref{thm:finite classes},  given a saddle periodic point $p$  and one of its unstable branches $\Ga^u(p)$, it holds  

\begin{itemize}
  \item[--] either it  crosses the stable manifold   of a saddle periodic point or, 

 \item[--] there is a (semi)attracting periodic point such $\Ga^u(p)$   intersects the interior of its  basin of attraction.
    
\end{itemize}

\end{proposition}

\begin{proof} 

 If  the closure of the unstable branch of $p$ does not have critical points then by theorem \ref{disk2} either it is contained in the basin of attraction of a (semi)attracting periodic point or it is a non-trivial normally hyperbolic attracting interval; in the last case the unstable branch crosses the stable manifold of a saddle in the normally hyperbolic arc. 

If it does have  critical points $c$  then $\omega(c)$ has dominated splitting so either it contains interior points, or is a heterocycle, or it is a periodic point.  In the first two  cases the proposition is proved.  In the last case, it has to be a boundary saddle periodic point. 

So, it remains to consider the case that all the critical points in the closure of the unstable branch are  boundary critical points. In that case, recall that the boundary critical points  belong to a finite number of boundary saddle points, named $\{q_1
\dots q_n\}$. 

\emph{Iterating the dichotomy.} For each boundary saddle point $q_i$, one considers its unstable branch  $\Ga^u(q_i)$ that is accumulated by the unstable branch of $p$ , and apply the same dichotomies as before: if  some  branch $\Gamma^u(q_i)$  does not have critical points, again by theorem \ref{disk2}  the proposition is concluded. So, it remains to consider the case where for all $q_i$ the set $\La^u(q_i)$ (the closure of any $\Ga^u(q_i)$ that is accumulated by $\Ga^u(p)$) has critical points that all are boundary critical points.

\emph{Creating cycles.} Reasoning as  before, for each critical  boundary point $c_i$  there is  a critical cycle (named ${\mathcal C}_i$) of saddle boundary points such that there is no transversal homoclinic intersection between the unstable and stable branches of those saddle points (otherwise the proposition is also proved); moreover, by remark \ref{enlarging} one can assume that the cycle ${\cal C}_i$  is maximal. If the cycle is a touching cycle, by proposition \ref{touching} the proposition  is concluded. 

\emph{Closing into Pixton disks and escaping from them.} It remains to  consider the case that all the  maximal cycles ${\cal C}_i$ are not  touching cycles. For each ${\cal C}_i$ one produces a  Pixton disk $D_{j_i}$. By lemma \ref{rmk:pixton} it follows that there is $x_{j_i}$  in the unstable branch of one of the saddles that bounds the Pixton disk $D_{j_i}$ such that $\omega(x_{j_i})$ is disjoint with  the interior of the Pixton disk $D_{j_i}$. If for  some $x_{j_i}$ it follows that $\omega(x_{j_i})$ does not have a critical point then, arguing as before, $\omega(x_{j_i})$ is a periodic orbit $O(\hat q)$ contained in $\La^u(p)$ and  $x_{j_i}\in W^s_\DD(\hat q)$. 

Now we have a simple dichotomy: either i) $\hat q\notin \{q_1\dots q_n\}$ or ii) $\hat q\in \{q_1\dots q_n\}$. In the first case, the unstable branch accumulated by $\Ga^u(p)$ does not have critical points and therefore, $\La^u(\hat q)$ is either a (semi)attracting periodic point or a normally hyperbolic arc and in both cases the proposition is proved.  In the second case,  the unstable branches of the saddles in the cycle ${\cal C}_{j_i}$  accumulate on  another Pixton disk $D_{j_i'}$. One can repeat the argument with the Pixton disk $D_{j_i'}$ and observe that the proposition is either proved  or the unstable branches of the periodic points in the cycle ${\cal C}_{j_i'}$ involved in the  Pixton disk has to accumulate on another Pixton disk;   since the cycle in each Pixton disk is maximal there cannot be a cycle between Pixton disk, then it follows that for some $x_{j_i}$ its $\omega(x_{j_i})$ does not have critical points and the proposition is concluded.

\end{proof}

\paragraph{Proof of proposition \ref{existence of touching}.} 

Let us consider two saddle $q_0$ and $q_1$ involved in the maximal critical cycle and let $c_1$ be the critical point in  $\Gamma^s(q_1)$ (one of the branches of $q_1$) that is accumulated by $\Gamma^u(q_0)$ and the orbits of the $*-$sequence $(p_n)$. The goal is to show that either $\Gamma^u(q_0)$ intersects $\Gamma^s(q_1)$ at $c_1$ or there is another saddle  periodic point $\bar q$ such that $\Gamma^u(q_0)$ intersects $\Gamma^s(\bar q)$ and  $\Gamma^u(\bar q)$ intersects $\Gamma^s(q_1).$ Applying this argument to each consecutive pair in the cycle, one concludes that the cycle is either a critical touching cycle or can be enlarged to a touching cycle. 

Since $\al(c_1)$ has dominated splitting and since the orbit of the $*-$sequence $(p_n)$ also accumulates at $\alpha(c_1)$ it follows that $\al(c_1)$ is a saddle periodic point named $\bar q$.  If $\bar q$ is equal to $q_0$ then $\Ga^u(q_0)$ intersects $\Ga^s(q_1)$. If $\bar q\neq  q_0$, observing that $\Ga^u(q_0)$ and the orbit of the $*-$sequence $(p_n)$ also accumulate on $\bar q$,  there is a point $\bar x$ in $\Ga^s_\DD(\bar q)$ (one of the local stable branches of $\bar q$) that is accumulated by $\Ga^u(p_0)$ and the orbits of $(p_n)$ along a same quadrant of $\bar q.$ If $\bar x$ is a critical point, it follows that adding $\bar q$ to the cycle one gets a larger critical cycle contradicting the fact that the initial critical cycle was  maximal.

To continue, one has to consider the case that $\bar x$ is not a critical point, and now we consider different possibilities related to $\al(\bar x)$: either it does not contain a critical point or it has it. Similarly, as before, the goal is to prove that $\Ga^u(q_0)$ intersects $\Ga^s(\bar q)$ and in that way we enlarge the cycle adding a periodic point such that now that part of the cycle is touching.

If $\al(\bar x)$ does not have  a critical point then it has a dominated splitting and by the fact that the  pre-iterates of the sequence $(p_n)$ accumulate on $\al(\bar x)$ it follows by the definition of $*-$sequence that $\alpha(\bar x)$ is a saddle periodic orbit (that one can assume it is fixed), named $\hat q$. In that case, $\bar x$ belongs to $\Ga^u(\hat q)$. The intersection of the unstable manifold of $\hat q$ and $\bar q$ at $\bar x$ is transversal, otherwise by lemma \ref{lem:tang-crit} there is a point in the orbit of $\bar x$ that is a critical point; since the $*-$sequence and $\Ga^u(q_0)$ also accumulate on the orbit of $\bar x$ we get a contradiction with the assumption that the cycle is maximal.   Since $\Ga^u(q_0)$ accumulates at $\bar x$, it also accumulates at $\hat q$; let $\hat Q$ be the quadrant of $\hat q$ that contains points of $\Ga^u(q_0)$ accumulating on $\hat q.$ Let $\Ga^s_\DD(\hat q)$ be the local stable branch of $\hat q$ that bounds $\hat Q$, and observe that  is accumulated by  connected arcs $\bar \ga_n$ contained in $\Ga^s(\bar q)$ that one extreme points are pre-iterates of $\bar x$ and the others are in the boundary of the disk. Then one gets a sequence of rectangles bounded by $\Ga^s_\DD(\hat q)$, $\bar \ga_n$, arcs in $\Ga^u(\hat q)$ and arcs in the boundary of the disk. These rectangles  contain in the interior iterates of $\Ga^u(q_0)$, and therefore it has to cross  the arcs $\bar \ga_n$ a so $\Ga^u(q_0)$ intersects $\Ga^s(\bar q)$ as we wanted.

It remains to consider the case that  $\al(\bar x)$  has critical points and let $\hat c$ be one of that critical points; since $\hat c$ is also accumulated by $\Ga^u(q_0)$ and pre-iterates of the sequence $(p_n)$ it follows that it is a boundary critical point and so there is a boundary periodic point $\hat q$,  that one can assume that it is fixed,  and $\hat c \in \Ga^s_\DD(\hat q)$.   
By proposition \ref{prop trans} (and the fact that $\Ga^u(\hat q)$ cannot be fully contained in the basin of attraction of a (semi)attracting periodic points since that branch is accumulated by periodic orbits), either there is a saddle $\tilde q$ such that $\Ga^u(\hat q)$ (the unstable branch of $\hat q$ that it is accumulated by the pre-iterates of $\bar x$) that crosses $\Ga^s(\tilde q)$ or there is a (semi)attracting periodic point (also denoted as $\tilde q$) such that the interior of its basin intersects  $\Ga^u(\hat q)$; in both cases, one gets an arc $\tilde \ga$ in  $W^s_\DD(\tilde q)$  that cross $\Ga^u(\hat q).$  Now, one  gets  that $\Ga^s_\DD(\hat q)$   is accumulated by  connected arcs $\tilde \ga_n$ obtained as pre-iterates of arcs in  $\tilde \ga$ whose endpoints are in $\Ga^u(\hat q)$ and in the boundary of the disk. Moreover, these arcs intersect the quadrant $\hat Q$ that contains the pre-iterates of $\bar x$ that accumulate on $\hat c$. Then one gets a sequence of rectangles $R_n$ bounded by $\hat \ga$, $\tilde \ga_n$, arcs in $\Ga^u(\hat q)$ and arcs in the boundary of the disk, that contains pre-iterates of $\bar x$. Let us pick $f^{-k}(\bar x)$ close to $\hat c$ and let $\bar \ga_k\subset W^s_\DD(f^{-k}(\bar x))$ be an arc in $\Ga^s(\bar q)$ that one end point is $\bar q$ and the other is $f^{-k}(x).$ Provided that $n$ and $k$ are  large, the   arc $\bar \ga_k$ has to exit $R_n$ and given that it cannot intersect neither $\hat \ga$ nor $\tilde \ga_n$ then it has to intersect either $\Ga^u(\hat q)$ or the boundary of the disk. In both cases, either pre-iterating or post-iterating $\bar \ga_k$,  there is a connected arc $\bar \ga_k'$ in $\Ga^s(\bar q)$ inside $R_n$ such that one extreme point close to $f(\hat c)$ and the other close to $f^{-1}(\hat c).$ On the other hand let us consider a rectangle $R$ as in lemma \ref{rectangle}, where the arcs $\si_m$ are contained in $\Ga^u(q_0).$ Taking $R\cap R_n$ observe that both arcs $\be_1$ and $\be_2$ are crossed by  $\bar \ga_k'$ and therefore, the arcs $\si_n$ has to cross $\bar \ga_k'$, showing again that $\Ga^u(q_0)$ intersects transversely $\Ga^s(\bar q)$ as we wanted  (see figure \ref{final} for the last configuration).

\qed

\begin{figure}[h]
\centering
\includegraphics[scale=0.75]{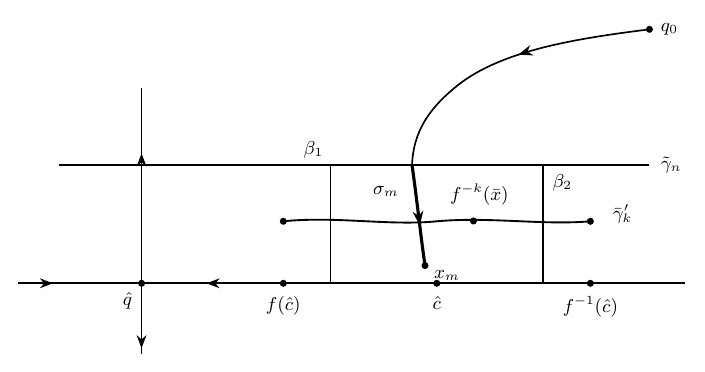}
\caption{Final of the proof of proposition \ref{existence of touching} }
\label{final}
\end{figure}

That concludes the proof of theorem \ref{thm:finite classes}.

\section{Further questions}\label{sec:questions}

 The critical set is a new
object, and several natural questions about it are left open by this paper.

\paragraph{Is ${\rm Crit}(f,\La)$ totally disconnected?} In all the examples of
section \ref{examples} the critical set is at most a Cantor set (for the
Wang-Young attractor (paragraph iv-) it is explicitly the intersection of a
nested sequence of rectangles) and none exhibits a critical arc. When
$f_{|\La}$ is far from homotheties the critical direction is unique and
varies continuously, which constrains the set; we do not know whether this,
or a dissipation hypothesis, forces total disconnectedness, nor whether a
continuum of critical points can occur in general.

\paragraph{When is a critical point strong?} Following \cite{CLPY} (paragraph
v-), one may strengthen definition \ref{def:crit} by requiring exponential
growth of the projective cocycle along the critical direction in both time
directions, $|g^n(v)|>(1+\de)^{|n|}$ for all $n$ and some $\de>0$. The
strong notion carries genuine geometry: at such a point there are a local
stable and a local center-stable manifold, tangent at the critical point, so
that the tangency picture of paragraph i- is recovered as actual invariant
manifolds rather than merely as a condition on the cocycle. For critical
points coming from tangencies this strengthening holds automatically
(paragraph i-), and proposition \ref{exponential recurrence} controls the
recurrent case; it would be of interest to identify general conditions,
on the recurrence of the critical set or on the dynamics of $\La$, under
which every critical point is strong, since, as \cite{CLPY} shows, the strong
notion carries substantial structural information.

\paragraph{Critical-orbit conditions and SRB measures.} More speculatively, the
critical set may support, for surface diffeomorphisms, the program that in
one dimension extracts ergodic information from conditions on the critical
orbit: a Collet--Eckmann condition, or summability and slow-recurrence
conditions in the spirit of Misiurewicz, yield absolutely continuous
invariant measures and, more generally, SRB measures. The Misiurewicz
condition of section \ref{sec:misiu} is the non-recurrent extreme of such a
hierarchy; one may ask whether intermediate conditions, call it a two-dimensional
Collet--Eckmann condition along the critical directions, or controlled
recurrence of the critical set, yield SRB measures, as they do in
dimension one and as they do, by different methods, in the work of
Benedicks--Carleson and Wang--Young (paragraph iv-).

\begin{tabular}{l l l}
\emph{Sylvain Crovisier} &&
\emph{Enrique Pujals}\\

Laboratoire de Math\'ematiques d'Orsay &&
Graduate Center \\

CNRS - Univ. Paris-Saclay &&
CUNY\\

Orsay, France &&
New York, USA\\

&&\\

\end{tabular}

\end{document}